\documentclass[review,onefignum,onetabnum]{article}

\usepackage{subfig}
\usepackage{amsmath,amstext,amssymb,amsfonts,amsthm}
\usepackage{mathrsfs}
\usepackage{bm}
\usepackage{graphicx}
\usepackage{empheq}
\usepackage{bbold}
\usepackage{color}
\usepackage{comment}
\usepackage[ruled,vlined,linesnumbered]{algorithm2e}
\usepackage{setspace}

\def\x{\mathbf{x}}

\def \rr {\mathbb{R}}

\def \rd {\mathbb{R}^2}

\def \RR {\mathbb{R}}

\def \yp {y_{ref}}

\def \eps {\varepsilon}

\newtheorem{remark}{Remark}
\newtheorem{coro}{Corollary}
\newtheorem{definition}{Definition}
\newtheorem{lemma}{Lemma}
\newtheorem{proposition}{Proposition}
\newtheorem{theorem}{Theorem}
\newtheorem{hypothesis}{Hypothesis}[section]
\newcommand{\email}[1]{\texttt{#1}}
\newcommand{\IMT}{Universit\'e de Toulouse; UPS, INSA, UT1, UTM, Institut de Math\'ematiques de Toulouse,
CNRS, Institut de Math\'ematiques de Toulouse UMR 5219, F-31062 Toulouse, France; \email{fabrice.deluzet@math.univ-toulouse.fr}}
\newcommand{\IEC}{Universit\'e de Lorraine, CNRS, IECL, F-54000 Nancy, France; \email{forename.name@univ-lorraine.fr}}
\newcommand{\Dollars}{This work has been supported by a public grant from the "Laboratoire d'Excel\-lence Centre International de Math\'ematiques et d'Informatique"  (Labex CIMI) overseen by the French National Agency (ANR) as part of the "Investissement d'Avenir" program (reference ANR-11-LABX-0040) in the frame of the PROMETEUS project (PRospect of nOvel nuMerical modEls for elecTric propulsion and low tEmperatUre plaSmas).\newline
  Support from the ``F\'ed\'eration de Fusion pour la Recherche par Confinement Magn\'etique'' (FrFCM) in the frame of the project ``NEMESIA: Numerical mEthods for Macroscopic models of magnEtized plaSmas and related anIsotropic equAtions'' is also acknowledged
  }

\newcommand{\One}[1]{#1}   
\newcommand{\Two}[1]{#1}   

\begin{document}

\title{Beyond the analytic solution manufacturing for anisotropic elliptic equations
\footnote{\Dollars}}
\author{Fabrice Deluzet\thanks{\IMT} \and  Vladimir Latocha\thanks{\IEC} \and Fr\'ed\'eric Robert\footnotemark[3]}

	\maketitle
	
	\begin{abstract}
	This article is aimed at introducing a new procedure to define reference solutions to which the numerical approximations of anisotropic problems arising in plasma physics may be compared. In this context, the main issue is related to the construction of functions with no gradients along the vector field defining the anisotropy direction. The so-called solution manufacturing, implemented so far in the literature, only applies to framework where all the computations can be carried out analytically. In the present paper, we propose a new method in which this requirement is not mandatory. Indeed, via a Riemannian metric, we use a propagation of the boundary value by the flows related to the vector field defining the anisotropy direction. This permits to extend the application of the solution manufacturing to a broader range of contexts.
	\end{abstract}
	


\section{Introduction}
The physics of plasma evolving under intense magnetic fields is  driven by equations with a severe anisotropy \cite{chodura_plasma_1986,hazeltine_plasma_2003,chen_introduction_2006}. The transport of charged particles is indeed more effective along the direction aligned with the magnetic field than in the transverse directions. This feature is at the origin of a large difference in the force imbalance existing between the transverse and the aligned directions to the magnetic field, the almost infinite mobility of particles along the magnetic field lines acting to restore the forces balance \cite{degond_asymptotic-preserving_2013}. This property has motivated the development of dedicated reduced models and related numerical methods (see for instance \cite{coulette_numerical_2013,besse_gyro-water-bag_2009,snyder_gyrofluid_1999,crouseilles_isogeometric_2012,grandgirard_5d_2016}). 

\smallskip\noindent One class of anisotropic problems arising in the physics of magnetized plasma \cite{degond_asymptotic-preserving_2013} is illustrated by means of the following toy model.
Let $b(x,y)$ be a regular unit vector pointing in the direction of the magnetic field, with $(x,y)\in \Omega$ an open bounded subset of $\mathbb{R}^2$, $\partial \Omega=\Gamma_D \cup \Gamma_N$ denoting its boundary. \One{The generic model problem is defined as}
\begin{subequations}\label{sys:prob:mod}
\begin{equation}\left\{
	\begin{aligned}
	&- \Delta_\perp u^\eps  - \frac{1}{\eps}\Delta_\parallel u^\eps = f^\eps \,,&\qquad &  \text{in } \Omega \,,\\
	&u^\eps = g_D\,,& \qquad  &\text{on } \Gamma_D \,,\\
	&(\nabla_\perp u^\eps + \frac{1}{\eps} \nabla_\parallel u^\eps)\cdot n = g_N \,,&\qquad  &\text{on } \Gamma_N\,,
	\end{aligned}\right.
\end{equation}
\One{with, $g_D$ and $g_N$ functions defining the boundary values on respectively $\Gamma_D$ and $\Gamma_N$, and},
\begin{equation}
	\begin{aligned}
		\nabla_\parallel u &= (b\cdot \nabla u) b \,,& \qquad \nabla_\perp u &= \nabla u - \nabla_\parallel u \,,\\
		\Delta_\parallel u &= \nabla \cdot \nabla_\parallel u \,,& \qquad \Delta_\perp u &= \Delta u - \Delta_\parallel u \,.
	\end{aligned}
\end{equation}
\end{subequations}
\One{The solution of the problem, denoted $u^\varepsilon$, is a scalar-valued function. Depending on the applications, the solution to the anisotropic problem may refer to either the electrostatic potential, from which the electric field is derived in quasi-neutral models or, to the plasma density as well as temperature in the so-called drift approximation \cite{degond_asymptotic-preserving_2013}.} In these equations, the anisotropy direction is defined by the vector field $b$ and the anisotropy strength by means of the  asymptotic parameter $\eps^{-1}$. The regime of interest is related to large anisotropies (or magnetic field magnitudes), hence $\eps\ll1$.  

\paragraph{\bf Numerical challenge as $\eps\ll 1$}
\medskip\noindent The solutions to Eqs.~\eqref{sys:prob:mod} are prone to develop gradients with a small amplitude in the parallel direction (with respect to the $b$-field) compared to the perpendicular ones, translated into the following scaling relations
\begin{equation}\label{eq:prop:gradients}
	\nabla_\parallel u^\eps = \mathcal{O}(\eps) \,, \qquad \nabla_\perp u^\eps = \mathcal{O}(1)\,. 
\end{equation}
 This property also gives rise to a source term $f^\eps$ remaining finite in the limit of infinite anisotropy, with
$\lim_{\eps\to0} f^\eps = f^0$ with $\|f^0\|_\infty  < \infty$.
Assuming that the solution is regular with respect to $\eps$ the following Ansatz is therefore introduced 
\begin{subequations}\label{sys:grad:para:ansatz}
\begin{equation}
	u^\eps(x,y) = u^0(x,y) + \eps u^1(x,y) \,, \label{eq:grad:para:ansatz}
\end{equation}
where the component $u^0$ is constant along the $b$-field lines, yielding
\begin{equation}
	\nabla_\parallel u^0(x,y) = 0 \,. \label{eq:grad:para:ansatz:property}
\end{equation}
\end{subequations}
Owing to the definition \eqref{sys:grad:para:ansatz}, the gradients (or the Laplacians) of the solution reduces to 
\begin{equation}
	\begin{aligned}		
		\nabla_\parallel u^\eps &= \eps \nabla_\parallel u^1 = \mathcal{O}(\eps)\,,&  \qquad \nabla_\perp u^\eps &= \nabla_\perp u^0 + \eps \nabla_\perp u^1 = \mathcal{O}(1)\,, \\
		\qquad \Delta_\parallel u^\eps &= \eps \Delta_\parallel u^1 = \mathcal{O}(\eps) \,, & \qquad \Delta_\perp u^\eps &= \Delta_\perp u^0 + \eps \Delta_\perp u^1 = \mathcal{O}(1) \,,
	\end{aligned}
\end{equation}
entailing a finite source term $f^\eps$ for vanishing $\eps$.

\medskip\noindent  This is a challenge for numerical methods to reproduce the anisotropy in the gradient as stated by Eqs.~\eqref{eq:prop:gradients}, specifically when the mesh and the anisotropy direction are misaligned \cite{yang_accuracy_2025}. This issue has motivated a prolific literature aiming at developing sophisticated discretization techniques. To cite few of them, one can first mention the use of high order discretizations \cite{piraccini_spatial_2022,piraccini_recent_2022,giorgiani_high-order_2020,crouseilles_comparison_2015}, specific reconstructions using flux coordinate \cite{hariri_flux-coordinate_2013,van_es_finite-difference_2014,van_es_finite-volume_2016} as well as a parametrization method \cite{guillaume_numerical_2005} and asymptotic preserving schemes \cite{tang_asymptotic_2017,wang_uniformly_2018,degond_asymptotic-preserving_2012,degond_duality-based_2012}, this list being non-exhaustive.
\paragraph{\bf Evaluating numerical methods solving (\ref{sys:prob:mod}) via solution manufacturing}
\medskip\noindent
Solution manufacturing is a widely employed technique for verifying simulation codes and assessing the reliability of numerical approximations. By constructing exact or high-fidelity solutions to benchmark problems, this method enables rigorous validation of computational methods, ensuring their accuracy and robustness in simulating complex physical phenomena across diverse application fields \cite{roache_code_2002,tranquilli_deterministic_2022,freno_manufactured_2024,amor-martin_rigorous_2024,etienne_manufactured_2012,veeraragavan_use_2016}. 

In the present context, it boils down to: first computing analytically a function $u^0$ with no parallel gradient; second, choosing an arbitrary function $u^1$; and last, computing the source term of the anisotropic equation $f^\eps= -\Delta_\perp (u^0+\eps u^1)-\Delta_\parallel u^1 $. This quantity is then used as right-hand side for testing numerical methods aimed at solving the system \eqref{sys:prob:mod} with the Ansatz \eqref{sys:grad:para:ansatz}, and the computed solution is compared to $u^0+\eps u^1$. This approach has been successfully implemented in previous works \cite{degond_asymptotic-preserving_2012,deluzet_two_2019} for specific $b$-field configurations, where one can analytically associate a solution and a $b$-field.

\paragraph{\bf Key contributions: Broadening the scope of numerical method validation beyond conventional analytical solution manufacturing, with specific focus on addressing the challenges posed by singular magnetic fields}
\medskip\noindent
Analytical approaches narrow the scope of treatable magnetic fields. Because applications involve solving \eqref{sys:prob:mod} in configurations such as plasma confinement by magnetic cusps \cite{jiang_magnetic_2020,deluzet_numerical_2023,garrigues_acceleration_2024}, it is natural to extend the computation of the right-hand side of \eqref{sys:prob:mod} to cases where $b$ is chosen according to physical considerations. 
Despite the simplicity of the equations, high anisotropy limits result reliability, motivating the specific treatment addressed in this article. In the present work, the solution manufacturing procedure is extended in two directions.
\begin{enumerate}
\item A richer set of $b$-field definitions can be handled using the flow-transported-solution method introduced herein. This novel approach extends beyond the direct analytical computation of the solution component $u^0$, as well as its perpendicular Laplacian $\Delta_\perp u^0$ required to define the source term in problem \eqref{sys:prob:mod}. The method relies on introducing a local chart, its derivatives, and the associated Riemannian metric to define a system of ordinary differential equations (ODEs). These ODEs are then numerically integrated to yield point-wise values of $u^0$ and $\nabla_\perp u^0$ at arbitrary spatial locations.

Moreover, the flow-transported-solution method supports a more general Ansatz for the solution, of the form $u= \check{u}^0+\eps u^1$, where $\Delta_\parallel \check{u}^0=0$ but $\nabla_\parallel \check{u}^0 \neq 0$ (see Theorem \ref{Th:B2}). The present work focuses on implementing the method for the simpler Ansatz $u= {u}^0+\eps u^1$ where $u^0$ has no parallel gradient $\nabla_\parallel u^0 = 0$, this choice is motivated by the need to validate the method within an analytical framework, using a reference solution for direct comparison. However, the flow-transported-solution method itself is fully general and can accommodate both Ansatz forms: the case compliant with the analytical framework $\nabla_\parallel u^0 = 0$ and the more general case with $\nabla_\parallel \check{u}^0 \neq 0$. Numerical exploration of this later case is left for future work.
\item This flow-transported-solution method is genuinely local, in the sense that both $u^0$ and $\Delta_\perp u^0$ are obtained by integrating appropriate boundary data, specified in the sequel, along magnetic field lines. The method is proven to be stable (see Corollary \ref{coro:stab}), regardless of the presence of singularities in the magnetic field. This property is instrumental in assessing the reliability of classical numerical methods, which require global resolution of the problem over the entire computational domain, including regions containing magnetic field singularities, thereby raising concerns about their accuracy and robustness. In contrast, the flow-transported-solution method provides a means of computing high-fidelity point-wise values while maintaining proven stability even in the vicinity of singular regions.
\end{enumerate}  

\medskip\noindent The paper is structured as follows. The flow-transported-solution method is formally introduced in Sec.~\ref{sec:laplacian} to present the key concepts and tools employed. For simplicity, this presentation is confined to the simplified framework derived from the Ansatz $u= {u}^0+\eps u^1$ with $\nabla_\parallel u^0=0$. This framework is resumed and specified in Sec.~\ref{sec:numerics}, where the effectiveness of the flow-transported-solution method is assessed within an analytical framework.
The more general class of solutions, of the form $u= \check{u}^0+\eps u^1$ with $\Delta_\parallel \check{u}^0= 0$ and $\nabla_\parallel \check{u}^0\neq 0$, is addressed in the mathematical analysis provided in appendix~\ref{appendix:fred}. There, rigorous hypotheses are established, and it is proven that the solution to the anisotropic problem, along with its Laplacian, can be recovered by integrating boundary data along flows associated with the $b$-field. Existence and stability results are also derived in this appendix.

\section{Computing the Laplacian in aligned coordinates}
\label{sec:laplacian}





\subsection{Introduction and hypotheses}
The purpose of this section is to introduce the material within a simplified framework and to detail the procedure that replaces the analytical computation of $u^0$ and $\Delta u^0$, $u^0$ being the component of the anisotropic problem solution with no parallel gradient as defined by Eqs.~\eqref{sys:grad:para:ansatz}. The objective is thus to provide a numerical method for approximating the functions $(x,y) \mapsto u^0(x,y)$  as well as  $(x,y) \mapsto \Delta u^0(x,y)=\Delta_\perp u^0(x,y)$. To establish clear milestones, the main result is first presented, followed by a specification of the tools used in this computation. Before delving into details, the framework is defined, and key properties are stated.

We consider a two-dimensional domain $\Omega = (0,1) \times (0,1) \subset \RR^2 $ with $\partial \Omega = \Gamma_x \cup \Gamma_y$, 	
		$\Gamma_x = (0,1) \times \{0,1\}$, $\Gamma_y = \{0,1\} \times (0,1)$. The vector field (or magnetic field)  $\vec B : \RR^2 \rightarrow \RR^2$ defining the anisotropy direction is assumed to satisfy the following properties
    \begin{equation}	\label{eq:def:proporties:b}
          \vert \vec B\cdot \vec n \vert>0 \quad \text{ on } \Gamma_x \,,\qquad
        \vec B\cdot \vec n = 0	\quad \text{ on } \Gamma_y \,,
      \end{equation}
$\vec n$ being the outward normal to the domain.
      Geometrically, this amounts to field lines that do not cross the
boundary $\Gamma_y$ and are normal to the boundary $\Gamma_x$.  The magnetic field is normalized, defining the unit vector $\vec b$ as
$\vec b(x,y) = {\vec B(x,y)}/{\|\vec B(x,y)\| }$,
where here and in the sequel $\|\cdot\|$ denote the usual Euclidean norm on $\rr^2$. The computation of $\Delta u^0(x,y)$ relies on propagating relevant quantities from the domain boundary $\Gamma_x$ along the flows associated with the 
$\vec b$ field, under the assumptions that every field line intersects this boundary and that field lines do not cross. A rigorous analysis, including issues related to the regularity of the vector field $\vec B$ as well as the singularity of $\vec b$ in points $\vec B(x,y) = (0,0)$ is proposed within appendix~\ref{appendix:fred}. The contains of this section remains formal, the aim being to present the concepts in the framework of the numerical investigations conducted in Sec.~\ref{sec:numerics}.
\subsection{Main result}
Let us denote $(x_0, t)$ the coordinate system aligned with the $\vec b$ field starting from the reference line $\{y=\yp\}$ for some $\yp\in [0,1]$. We will mostly take $\yp=0$ in this section, so that the reference line is a portion of $\Gamma_x$.
\One{Under appropriate assumptions regarding the geometry of $\vec b$ (see Appendix~\ref{appendix:fred}), it is possible to define a bijection, denoted $F^{-1}$ between the Cartesian coordinates $(x,y)$ and the coordinates $(x_0,t)$. This mapping, $F^{-1}: (x,y) \mapsto (x_0,t)$, is illustrated in Fig.~\ref{figMag}. Its inverse, denoted $F$ is referred to as the pullback mapping.}
\begin{figure}[htbp]
    \centering
    \begin{minipage}{0.4\textwidth} 
		\vspace*{-0.5cm}
        \centering
        \includegraphics[width=0.8\linewidth]{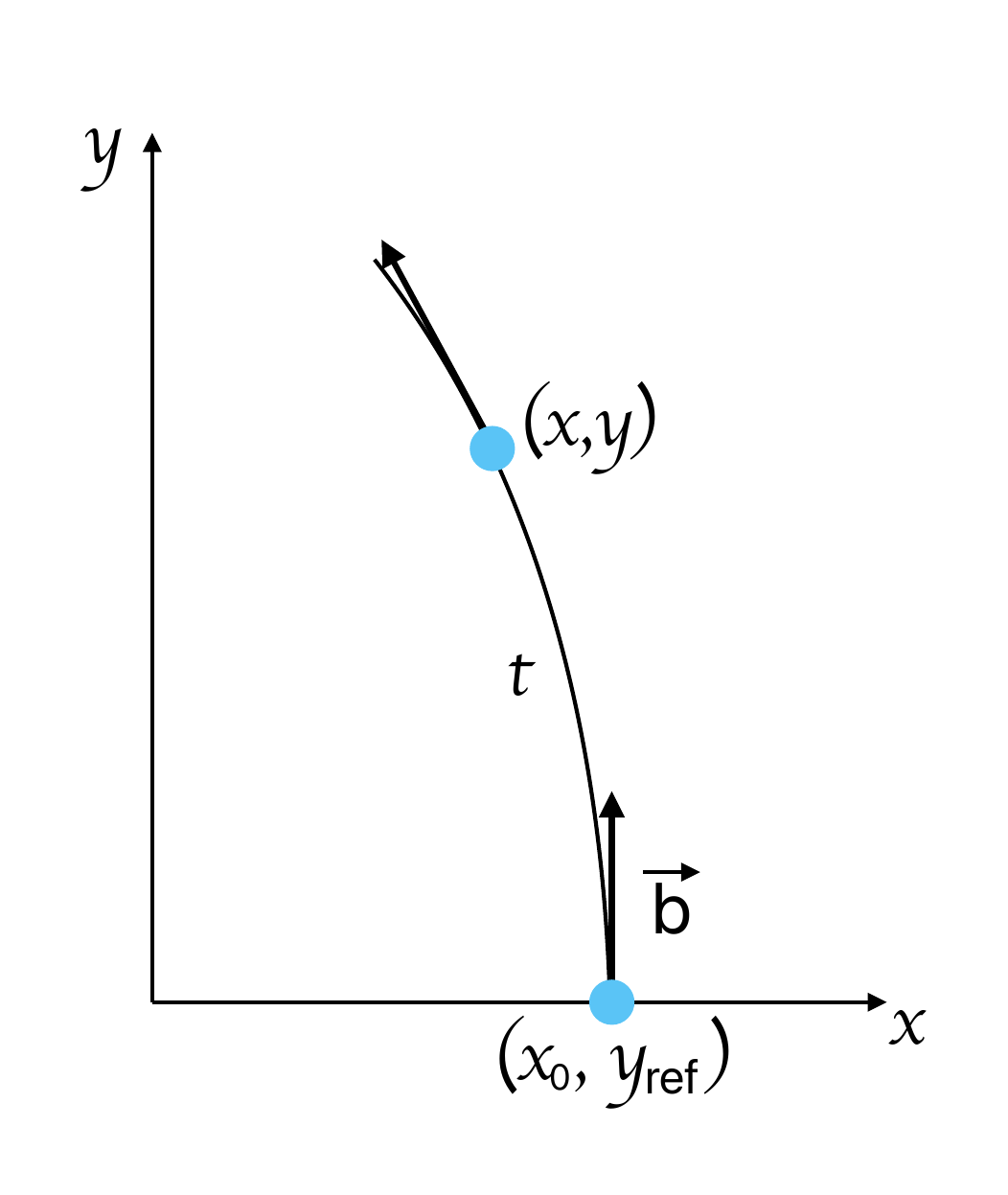}
		\vspace*{-0.5cm}
    \end{minipage}%
    \hfill
    \begin{minipage}{0.55\textwidth} 
        \captionof{figure}{Representation of the corresponding variables $(x,y) \leftrightarrow (x_0,t)$: 
        in this example, the point with coordinates $(x,y)$ lies on the magnetic field line crossing the
        boundary at $(x_0, \yp)$, and (owing to the fact that
        $||\vec{b}||\equiv 1$) $t$ is the length of the curve between
        $(x_0,\yp)$ and $(x,y)$. 
        }
        \label{figMag}
    \end{minipage}
\end{figure}
\One{The variable $x_0$ denotes the abscissa of the foot $p_0=(x_0,\yp)$ of the field line passing through the point $p=(x,y)$. The variable $t$ represents the arclength along the field line connecting $p$ and $p_0$.}

\smallskip\noindent Differential calculus yields the following formula, elucidated in the remainder of this section:
\begin{align}
	u^0\circ F(x_0,t)&= u^0(x_0,\yp)\,,\label{eq:def:u0:Flow}\\
(\Delta u^0)\circ F(x_0,t)&=\sum_{i,j}g^{ij}\partial_{ij}(u^0\circ F)-\sum_k\left(\sum_{p,q}g^{pq}\Gamma_{pq}^k\right)\partial_k(u^0\circ F)\,,\label{eq:Lap:u0}
\end{align}
where $(g^{ij})_{1\leq i,j \leq 2}$ is \One{the metric of $F$ the pullback mapping} (see Eq.~\eqref{exp:invg:0}) and the Christoffel symbols $\Gamma_{ij}^k$ (defined by Eq.~\eqref{def:christoffel}) satisfy 
\begin{equation}\label{eq:Gammakpq}
  \sum_{i,j}g^{ij}\Gamma_{ij}^k=\sum_{i,j,m}g^{ij}g^{km}\partial_ig_{mj}-\frac{1}{2}\sum_{i,j,m}g^{ij}g^{km}\partial_m g_{ij}
  \end{equation}

\subsection{The chart and its derivatives}\label{sec:chart}
The computations of the "flow transported approximation"  $(\Delta u^0)\circ F$ relies on three flows introduced in the following lines.
\subsubsection{Definition of $ \Phi_t(x_0,\yp )$,  the flow associated with $\vec{b}$}
From the field $\vec{b}$ the variables $(x_0,t)$ are defined thanks to the flow 
\begin{equation}\label{def:flow:0}
\left\{\begin{array}{l}
\displaystyle \frac{d}{dt}\Phi_t(x,y)=\vec{b}\big(\Phi_t(x,y)\big)\,,\\[1em]
\Phi_0(x,y)=(x,y )\,.
\end{array}\right.
\end{equation}
For any point $(x,y)\in \Omega$, the field line issued from $(x,y)$ is followed backwards until its foot located on the starting line $y=\yp $ as sketched on Fig.~\ref{figMag}. The coordinate $x_0$ is defined as the
abscissa of the field line foot, $t$  being the arc length of the field line between $(x_0,\yp )$ and $(x,y)$.
%
Following this procedure, a function constant along the $\vec b$ field lines can readily be constructed, transporting its value on the boundary. To this end, the backward flow $\check \Phi_t$ is introduced. \One{Defined as the downstream flow associated with the vector field b, it satisfies the identity $\check \Phi_t \circ \Phi_t=Id$, yielding}
\begin{equation}\label{def:flow:0:backward}
	\left\{\begin{array}{l}
	\displaystyle \frac{d}{dt}\check\Phi_t(x,y)=-\vec{b}\big(\check\Phi_t(x,y)\big)\,,\\[1em]
	\check \Phi_0(x,y)=(x,y)\,,
	\end{array}\right.
	\end{equation}
	and, denoting $t_0$ such that $(x_0,t_0)=F^{-1}(x,y)=\check\Phi_{t_0}(x,y)=(x_0,\yp )$, $(x_0,\yp )$ being the foot of the field line issued from $(x,y)$, the function 
$ (x,y)\mapsto u^0(x,y) = u^0\big(\check\Phi_{t_0}(x,y)\big)$, so that $F(x_0,t_0)=(x,y)$, is constant along the $b$-field lines.

\subsubsection{Definition of $S_t^{(1)}(x_0)=\partial_{x_0}\Phi_t(x_0,\yp )$,  a first order
  derivative of $\Phi_t$}
Differentiating $\Phi_t(x_0,\yp )$ defined by Eq.~\eqref{def:flow:0} with respect to $x_0$, yields
\begin{equation}\label{eq:Dx0flow:bis}
\left\{\begin{array}{l}
\displaystyle \frac{d}{dt}\big(\partial_{x_0}\Phi_t(x_0,\yp )\big)= D\vec{b}\big({\Phi_t(x_0,\yp )}\big)
\; \left[ \partial_{x_0}\Phi_t(x_0,\yp )\right]\,, \\[0.5em]
\partial_{x_0}\Phi_0(x_0,\yp )=(1,0)\,,
\end{array}\right.
\end{equation}
 where $D\vec{b}\left(\Phi_t(x_0,\yp )\right)$ is the Jacobian matrix
 of $\vec{b}(x,y)$ computed at $(x,y)=\Phi_t(x_0,\yp )$. 
 Since the only remaining variable to determine in $p_0=(x_0,\yp )$ is $x_0$, $\yp $ being a parameter, this yields to the definition of the flow $S^{(1)}_t(x_0) = \partial_{x_0}\Phi_t(x_0,\yp )$ solution to the Cauchy problem
\begin{gather*}
	\label{eq:Dx0flowPhi}
\left\{
\begin{array}{l}
\displaystyle \frac{d}{dt}\left(Y_t\right)= D\vec{b}\left({\Phi_t(x_0,\yp )}\right) \; \left[ Y_t\right]\,,\\[0.5em]
Y_{0}=(1,0) \,.
\end{array}\right.
\end{gather*}
\subsubsection{Definition of $S_t^{(2)}(x_0)=\partial^2_{x_0x_0}\Phi_t(x_0,\yp )$,  a second order
  derivative of $\Phi_t$}
Differentiating  \eqref{eq:Dx0flow:bis} with respect to $x_0$ and denoting $D^2\vec{b}$ the Hessian of $\vec b$, yields%
\begin{equation}\label{def:flow:3:bis}
\left\{\begin{array}{l}
\displaystyle \frac{d}{dt}\left(\partial^2_{x_0x_0}\Phi_t(x_0,\yp )\right)=
D\vec{b}\big(\Phi_t(x_0,\yp )\big)\left[\partial^2_{x_0x_0}\Phi_t(x_0,\yp )\right]\\
\hspace*{8em}+
D^2\vec{b}\big(\Phi_t(x_0,\yp )\big)
\left[\partial_{x_0}\Phi_t(x_0,\yp ),\partial_{x_0}\Phi_t(x_0,\yp )\right]\,,
\\
\partial^2_{x_0x_0}\Phi_0(x_0,\yp )=(0,0)\,.
\end{array}\right.
\end{equation}
The flow $S^{(2)}_t(x_0) = \partial^2_{x_0x_0}\Phi_t(x_0,\yp )$ is the solution to the Cauchy problem
\begin{gather*}	\label{eq:D2x0flowPhi}
\left\{
\begin{array}{l}
\displaystyle \frac{d}{dt}\left(Z_t\right)= D\vec{b}\big({\Phi_t(x_0,\yp )}\big) \; \left[ Z_t\right] + D^2\vec{b}\big(\Phi_t(x_0,\yp )\big)
\left[S^{(1)}_t(x_0),S^{(1)}_t(x_0)\right]\,,\\[0.5em]
Z_{0}=(0,0)\,.
\end{array}\right.
\end{gather*}

\subsection{Metric of the pullback F}

  The pullback metric, denoted $g$, and its inverse are defined as 
  \begin{equation}
    g=\left(g_{ij}\right) = \left(\begin{matrix}
      g_{tt} & g_{t x_0}\\ g_{tx_0} & g_{x_0x_0}
      \end{matrix}\right)\,, \qquad 
    g^{-1}=\left(g^{ij}\right) =\frac{1}{\det g} \left(\begin{matrix}
  g_{x_0x_0} & -g_{t x_0}\\ -g_{tx_0} & g_{tt}
  \end{matrix}\right)\label{exp:invg:0}.
  \end{equation}
with  
  \begin{equation}\label{exp:detg1}
    \begin{aligned}
    &g_{tt}=\langle \partial_t F,\partial_t F\rangle=1\,,\\
&g_{tx_0}=g_{x_0t}=\langle \partial_t F,\partial_{x_0}
                     F\rangle=\langle \vec{b}\big( \Phi_t(x_0, \yp )\big),
                     S_t^{(1)}(x_0)\rangle\,,\\
 &g_{x_0x_0}=\langle \partial_{x_0}F,\partial_{x_0} F\rangle=\Vert S_t^{(1)}(x_0)\Vert^2\,,
    \end{aligned}
\end{equation}
where $\langle \cdot, \cdot \rangle$ is the usual scalar product in $\rr^2$.
\One{We emphasize the fact that $g$ and $\det g$ (and therefore 
  $g^{-1}$) are explicit: the coefficients of the metric are straightforwardly obtained thanks to the flow $S^{(1)}_t$ which boils down to integrating the Cauchy problem defined by Eq.~\eqref{eq:Dx0flow:bis}}. 

\subsection{Expression of $\Delta u^0$ in the chart $F$}
\subsubsection{The derivatives of $u^0\circ F$}
\label{sec:deriv}
The expression of $(\Delta u^0)\circ F$ as stated by Eq.~\eqref{eq:Lap:u0} relies on the partial derivatives of $u^0\circ F$ as defined by Eq.~\eqref{eq:def:u0:Flow}, with
\begin{equation}     \label{eq:derivatives:UcircF:simple} 
	\begin{aligned}        
  \partial_t(u^0\circ F)&=\partial^2_{tt}(u^0\circ F) = \partial^2_{t x_0}(u^0\circ F)=0 \,,\\
  \partial_{x_0}(u^0\circ F)&=\partial_{x_0}u^0(x_0, \yp )\,,\\
\partial^2_{x_0 x_0}(u^0\circ F)&=\partial^2_{x_0x_0}u^0(x_0, \yp )\,.
	\end{aligned}
\end{equation}
We refer to appendix~\ref{appendix:fred} for the derivation of these quantities in a fairly general case.
\subsubsection{Estimates of $\sum_{p,q}g^{pq}\Gamma_{pq}^k$}
The coefficients of $g^{-1}=\left(g^{ij}\right)$ are defined by Eqs.~\eqref{exp:detg1}. We are left with computing the derivatives of the
metric coefficients to explicit the Christoffel symbols $\Gamma^k_{pq}$ as defined by Eq.~\eqref{eq:Gammakpq}. This yields
\begin{equation}\label{eq:def:dmgij}
\begin{aligned}
 & \partial_t g_{tt}=\partial_{x_0}g_{tt}= 0 \\
&\partial_t g_{t x_0}= \langle D\vec{b} \big(\Phi_t(x_0, \yp )\big)\big[\vec{b}\big(\Phi_t(x_0,\yp )\big)\big], S^{(1)}_t(x_0)\rangle\\
 &\begin{multlined}[0.8\textwidth]
\partial_{x_0} g_{t x_0}=\langle  \vec{b}\big(\Phi_t(x_0,\yp )\big), S^{(2)}_t(x_0)\rangle \\
+\langle D\vec{b}\big(\Phi_t(x_0,\yp )\big)\big[S_t^{(1)}(x_0)\big], S_t^{(1)}(x_0)\rangle\,,
 \end{multlined}\\
&\partial_t g_{x_0 x_0}= 2\langle D\vec{b} \big(\Phi_t(x_0, \yp )\big)\big[S_t^{(1)}(x_0)\big], S_t^{(1)}(x_0)\rangle\,,\\
&\partial_{x_0} g_{x_0 x_0}= 2\langle S_t^{(2)}(x_0), S_t^{(1)}(x_0)\rangle\,.
\end{aligned}
\end{equation}
The sequential steps for constructing $\Delta u^0(x,y)$ using the flow-transported method are detailed in Algorithm~\ref{algo:flow}.

\begin{algorithm}[H]
\small
\setstretch{0.9}
\caption{Computation of $\Delta u^0(\mathbf{x})$ via the flow-transported-solution method}
\label{algo:flow}
\KwIn{Evaluation point $\mathbf{x} = (x, y)$, vector field $\vec b$, boundary data on $y=\yp$}
\KwOut{$\Delta u^0(x,y)=\Delta u^0\circ F(x_0,t)$ as defined by Eq.~\eqref{eq:Lap:u0}.}

\BlankLine
\tcp{Step 0: Pre-processing via flows integration.}
Determine $(x_0, t)$ by integrating the backward flow defined by Eq.~\eqref{def:flow:0:backward} from $\mathbf{x}$ to the boundary located on $\yp \in \{0, 1\}$.\\;
Set $x_0$ as the abscissa of the foot of the current field line  and $t$ its arclength. \\;
Solve the Cauchy problems \eqref{eq:Dx0flow:bis} and \eqref{def:flow:3:bis} to obtain $S_{t}^{(1)}(x_{0})$ and $S_{t}^{(2)}(x_{0})$.\;

\BlankLine
\tcp{Step 1: Metric Tensor and Derivatives}
Compute the metric tensor $g$ using Eq.~\eqref{exp:detg1} and its inverse $g^{-1} = (g^{ij})$ from Eq.~\eqref{exp:invg:0}\\;
Compute the derivatives $\partial_{m}g_{ij}$ via Eq.~\eqref{eq:def:dmgij}.\;

\BlankLine
\tcp{Step 2: Pullback Derivatives}
Evaluate the derivatives of the boundary data $\partial_{ij}(u^{0}\circ F)$ and $\partial_{k}(u^{0}\circ F)$ using Eq.~\eqref{eq:derivatives:UcircF:simple}.\;

\BlankLine
\tcp{Step 3: Final Assembly}
Compute the contracted Christoffel symbols $C(k)=\sum_{i,j} g^{ij}\Gamma^k_{ij}$ defined by Eq.~\eqref{eq:Gammakpq}.\\;
Assemble $\Delta u^0(x,y)=\Delta u^0 \circ F(x_0,t)$ via Eq.~\eqref{eq:Lap:u0}, using the precomputed metric, the contracted Christoffel symbols and the pullback derivatives.\;
\end{algorithm}

\section{Numerical investigations}\label{sec:numerics}
\subsection{Solution manufacturing: analytic framework}
An analytic framework, introduced in precedent works \cite{degond_asymptotic-preserving_2012,deluzet_two_2019} to perform the verification of numerical methods discretizing anisotropic elliptic problems, is implemented to assess the precision of the flow transported solutions. An analytic solution and source term may be computed for vector fields derived from a potential $\psi$ and defined, for  $\gamma(x,y) \in \mathbb{R}$, as
\begin{eqnarray}\label{eq:B:grad:perp}
	B(x,y) = \gamma(x,y) \nabla^\perp \psi(x,y) = \gamma(x,y) \left(\frac{\partial \psi}{\partial y}(x,y), -\frac{\partial \psi}{\partial x}(x,y)\right)^T \,,
\end{eqnarray}
so that it is straightforward to verify that $B\cdot \nabla \psi =0$. For this class of magnetic field, it is therefore possible to exhibit functions with no parallel gradients.
Let $(x,y)\in \Omega$ with $\Omega = (0,1)\times (0,1)$. The boundaries of the domain are denoted $\Gamma_x$ and $\Gamma_y$
\begin{subequations}\label{sys:def:B}
\begin{equation}
	\begin{aligned}		
		\Gamma_x &= (0,1) \times \{0,1\} \,,\qquad \Gamma_y &= \{0,1\} \times (0,1) \,.
	\end{aligned}
\end{equation}
The potential from which the magnetic field is derived is defined as 
\begin{equation}
	\psi(x,y)=\sin  \big(\theta(x,y) \big) \,, \quad \theta(x,y) = \Two{\ell} \Big(\pi  x +\beta  \left(x^{2}-x \right) \cos \! \left(\pi  y \right)\Big)\,,
\end{equation}
with the vector field defining the anisotropy direction 
	\begin{equation}\label{eq:def:B}
		\vec B(x,y) = \left(\begin{array}{c}
\pi  \beta  \left(x^{2}-x \right) \sin \! \left(\pi  y \right) 
\\
 \beta  \left(2 x -1\right) \cos \! \left(\pi  y \right)+\pi  
\end{array}\right)\,, \qquad \vec b(x,y) = \frac{\vec B(x,y)}{\|\vec B(x,y) \|} \,.
	\end{equation}
The $B$-field is obtained thanks to Eq.~\eqref{eq:B:grad:perp} with $\gamma^{-1}(x,y)= -\Two{\ell} \cos\big(\theta(x,y)\big)$ and assuming $\gamma(x,y) \neq 0$ for $(x,y)\in \Omega$, although this property may not be met for any values of the parameters \Two{$\ell$} and $\beta$.  
The curvature of $\vec B$ is parameterized by $\beta>0$.  
For $\beta=0$ the field lines are straight lines, the $B$-field being aligned with the $y$-axis. For $\beta> \pi$ the field lines may be either open or closed. Nonetheless, the definition of $\Omega$  entails that any field line intersects one boundary of the domain. 
A graphical representation of the $B$-field is proposed on Fig.~\ref{fig:B:analytic} for $\beta=3$ and $\beta=8$. 
\begin{figure}[!ht]\centering
	\subfloat[$\beta=3$.]{\includegraphics[width=0.48\textwidth]{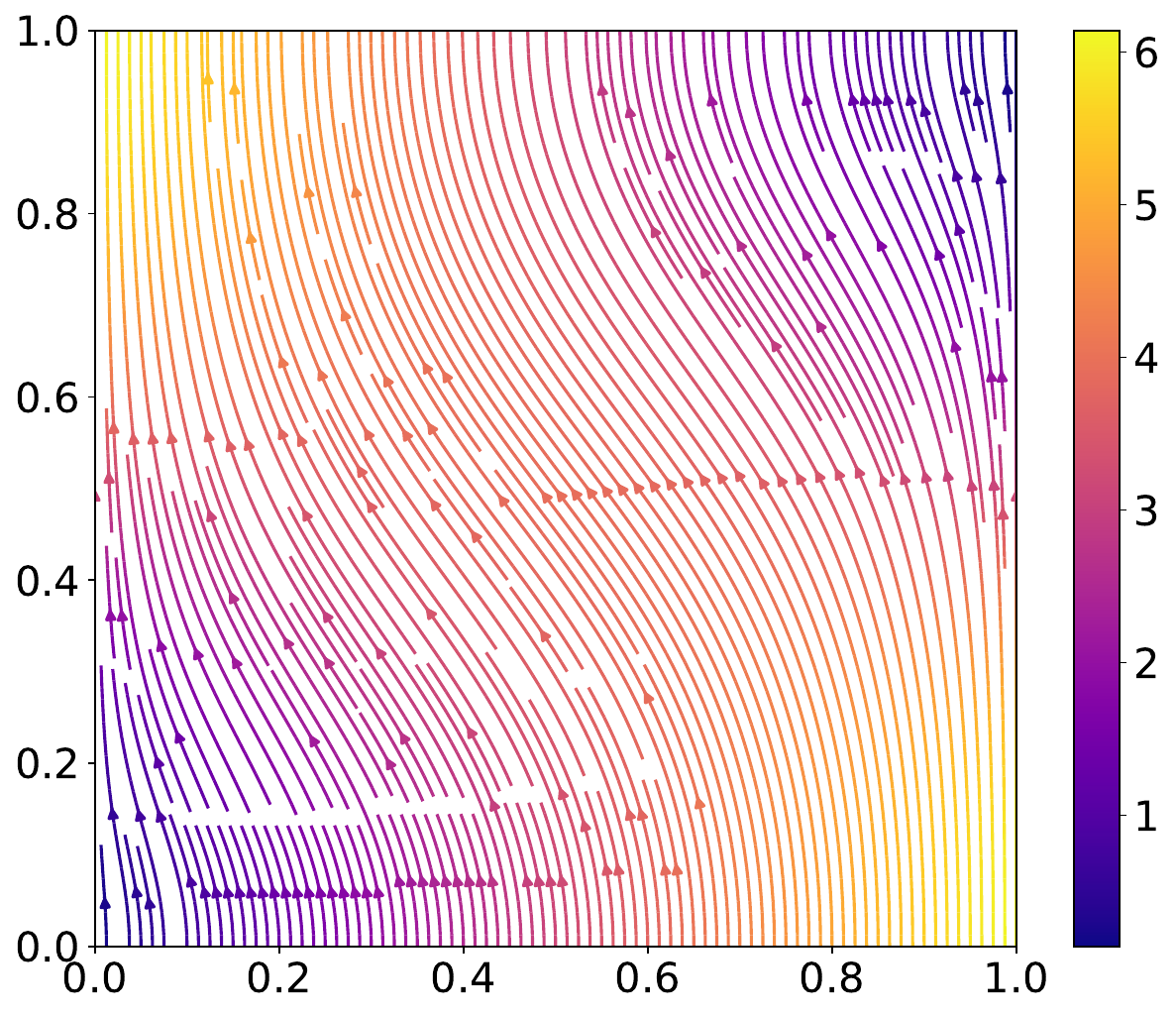}}\hspace*{0.04\textwidth}%
	\subfloat[$\beta=8$.]{\includegraphics[width=0.48\textwidth]{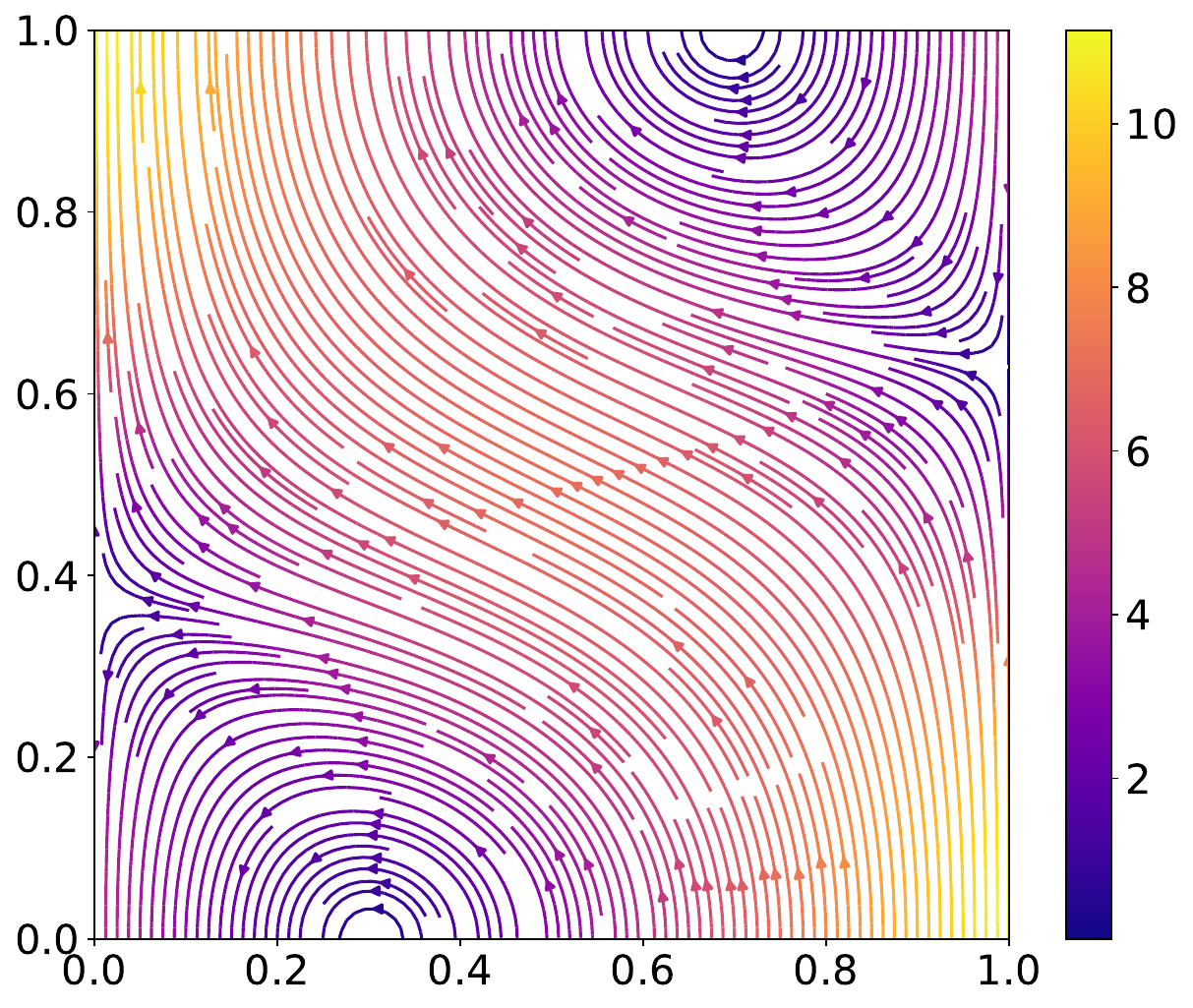}}	
	\caption{Vector field $\vec B$ defined by Eq.~\eqref{eq:def:B} for different $\beta$-values as a function of $(x,y)\in \overline{\Omega}=[0,1]\times[0,1]$.}\label{fig:B:analytic}
	\vspace*{-0.3cm}
\end{figure} 
The normalized vector field pointing in the direction of the anisotropy, denoted $\vec b $, satisfies the properties stated by Eq.~\eqref{eq:def:proporties:b}.%
\end{subequations}%
\begin{subequations}\label{eqs:solution:def}
A function with no parallel gradient may be readily deduced from $\psi$ with for instance
	\begin{equation}\label{eq:def:phi0}
		u^{0}(x,y)= \sin \big(\theta(x,y)\big)\,.
	\end{equation}
The function $u^0$ defined by Eq.~\eqref{eq:def:phi0} satisfies $\vec{b} \cdot \nabla u^0=0$  hence $\Delta_\parallel u^0 = \nabla \cdot \big( (\vec{b} \cdot \nabla u^0) \vec{b} \big) = 0$.
In this expression, \Two{$\ell$}-values parametrize, for a given value of $\beta$, the magnitude of the solution derivatives in the direction perpendicular to $\vec b$. The analytic solution is plotted on Figs.~\ref{fig:u0:manufactured}. Simple algebra yields the analytic expression of $\Delta u^0=\Delta_\perp u^0$:
\begin{figure}[!ht]\centering
	\subfloat[$(\beta,\Two{\ell})=(3,1)$.]{\includegraphics[height=0.39\textwidth]{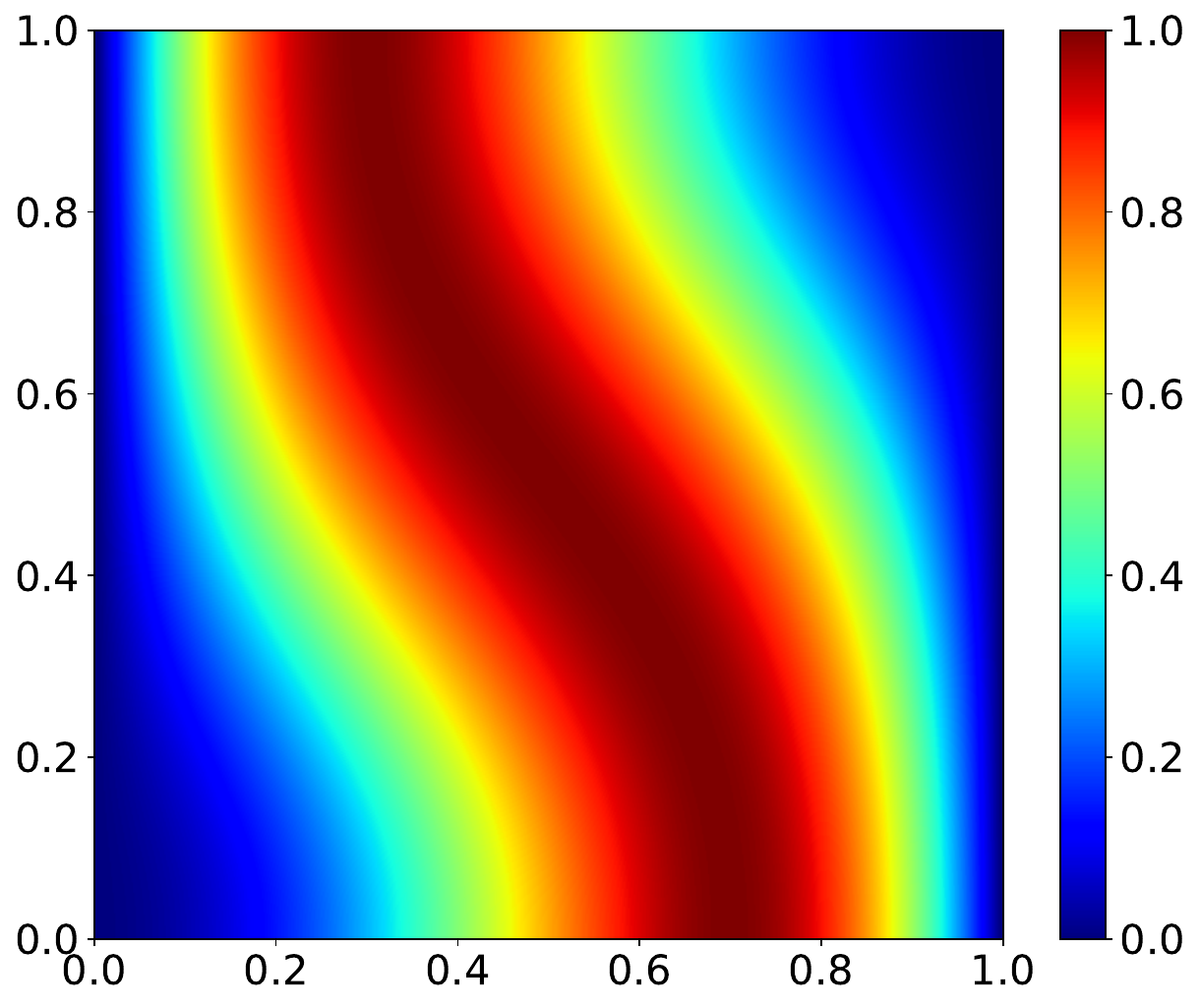}}\hspace*{0.04\textwidth}%
	\subfloat[$(\beta,\Two{\ell})=(8,8)$.]{\includegraphics[height=0.39\textwidth]{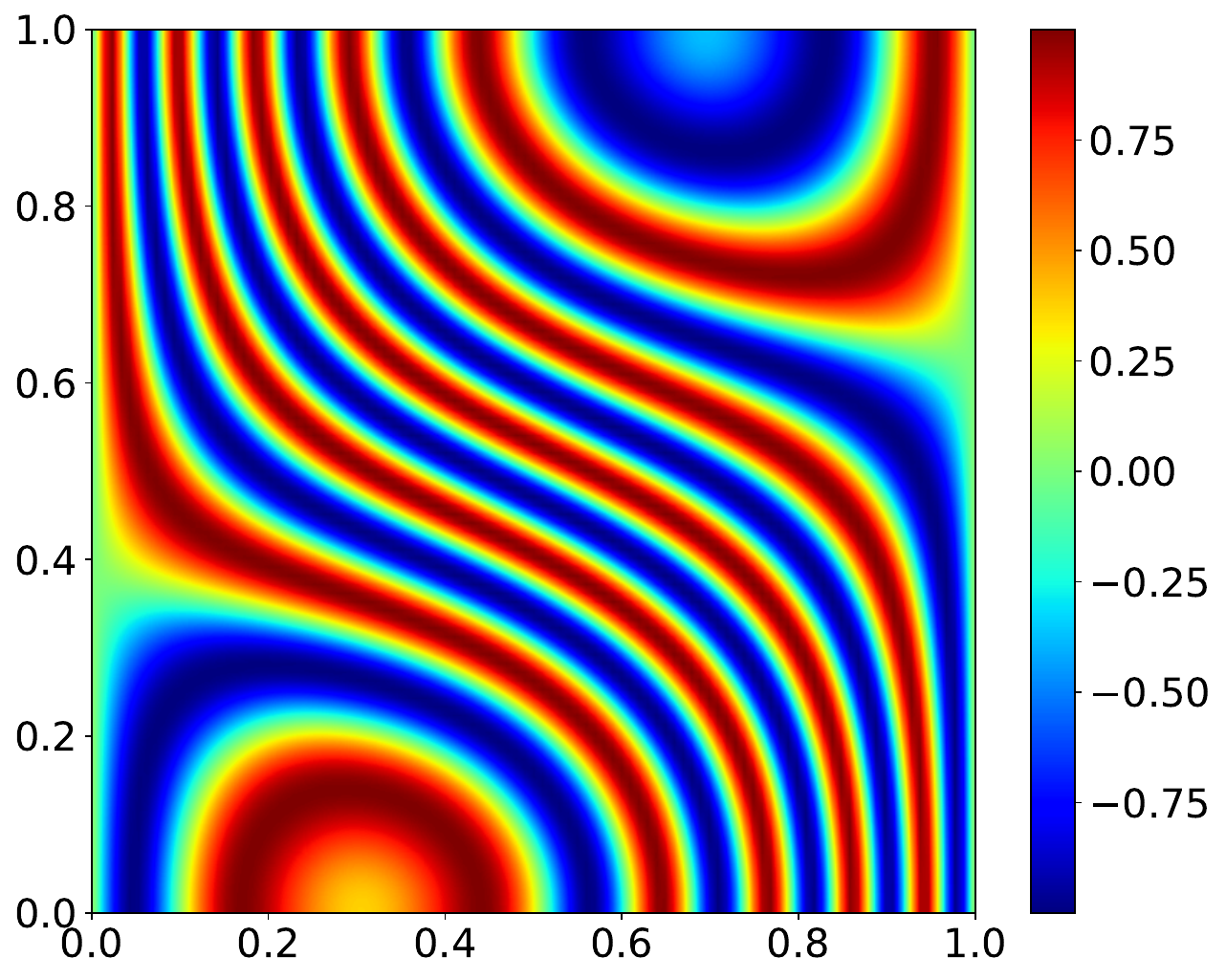}}
	\caption{Manufactured solution $u^0$ defined by Eq.~\eqref{eq:def:phi0} for different values of the parameters $\beta$ and \Two{$\ell$} as a function of $(x,y)\in \overline{\Omega}=[0,1]\times[0,1]$.}\label{fig:u0:manufactured}
	\vspace*{-0.5cm}
\end{figure}%
\begin{equation}\label{eq:def:Delta:phi0}
	\begin{aligned}
	&\begin{multlined}[0.85\textwidth]
	\Delta u^0(x,y) = -\Big(\Two{\ell} \cos \! \left(\pi  y \right) \beta  \big(-2+\left(x^{2}-x \right) \pi^{2}\big)\Big) \cos\big(\theta(x,y)\big)\\
	-\Two{\ell}^{2} \Big(\big(\beta  \left(2 x -1\right) \cos \! \left(\pi  y \right)+\pi \big)^{2}+\\
	\pi^{2} \beta^{2} \left(x^{2}-x \right)^{2} \left(\sin^{2}\left(\pi  y \right)\right)\Big)\, u^0(x,y)\,.
	\end{multlined}
	\end{aligned}
\end{equation}
\end{subequations}
 The iso-values of $\Delta u^0$ are plotted on Figs.~\ref{fig:Deltau0:manufactured} for different values of $\beta$ and $k$.
\begin{figure}[!ht]\centering
	\subfloat[$(\beta,\Two{\ell})=(3,1)$.]{\includegraphics[width=0.48\textwidth]{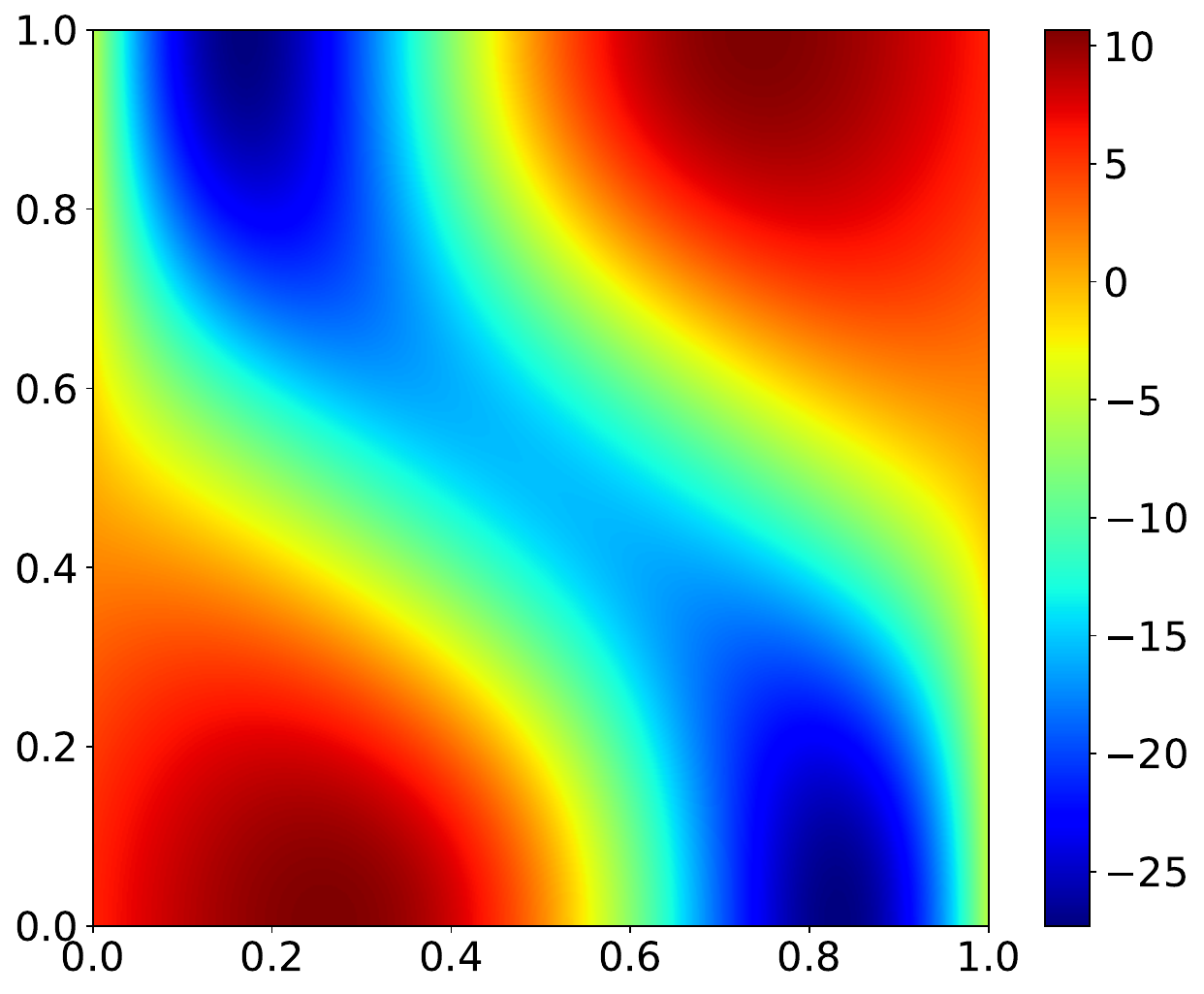}}\hspace*{0.04\textwidth}%
	\subfloat[$(\beta,\Two{\ell})=(8,8)$.]{\includegraphics[width=0.48\textwidth]{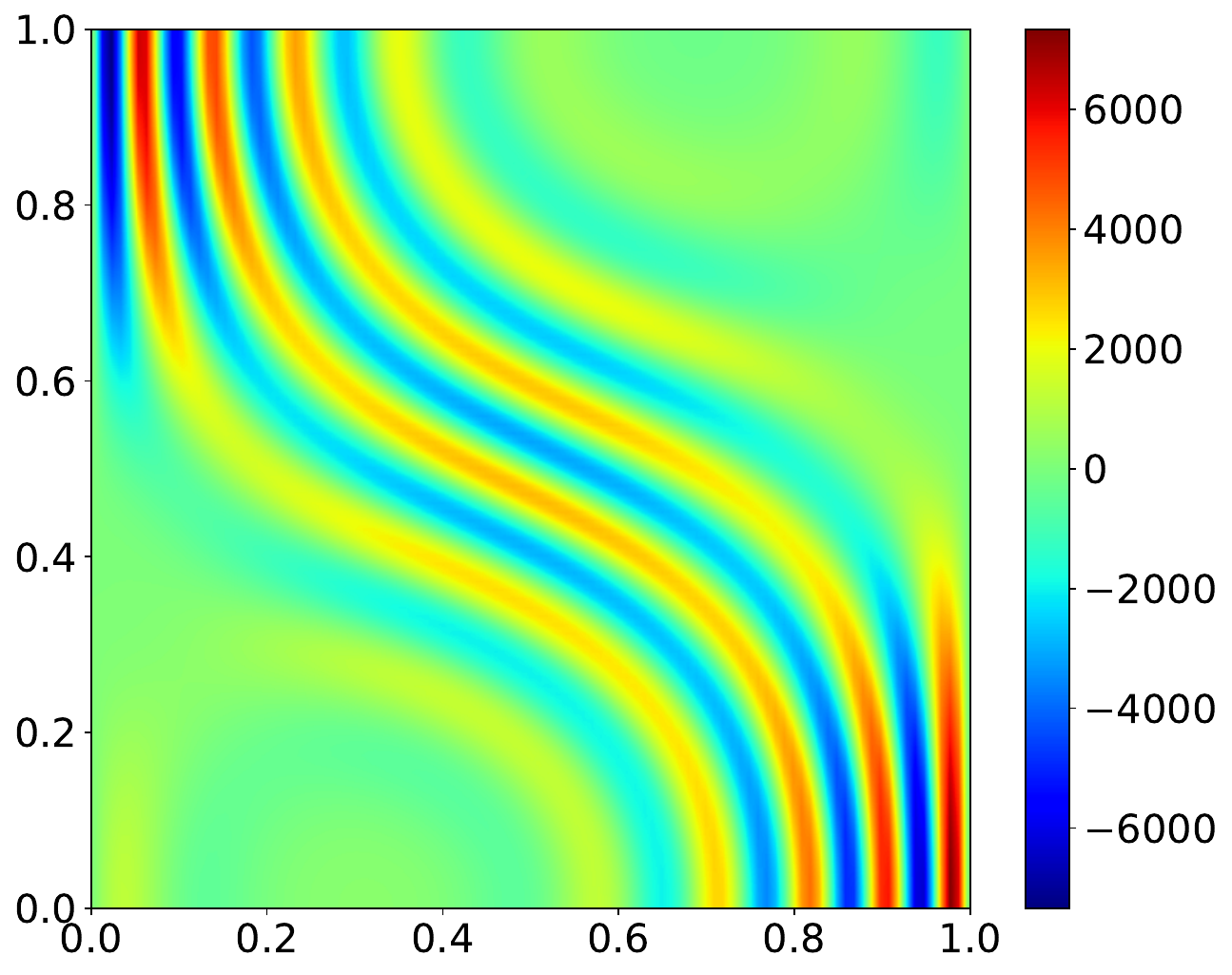}}
	\caption{Laplacian of the manufactured solution $u^0$ defined by Eq.~\eqref{eq:def:Delta:phi0} for different values of the parameters $\beta$ and \Two{$\ell$}.}\label{fig:Deltau0:manufactured}
	\vspace*{-0.5cm}
\end{figure}

The solution manufacturing is routinely implemented to provide a source term $f^\eps$ obtained by evaluating the anisotropic elliptic operator onto the solution $u^\eps = u^0 + \eps u^1$, where $u^1$ is an arbitrary function. This source term is then used to carry out a numerical approximation of $u^\eps$, the problem solution. The precision of the numerical approximation is estimated by comparisons against the analytic expression, to perform the verification of the numerical method. A similar procedure is implemented within this document to assess the precision of the flow transported solutions denoted $\phi^0$ and $\Delta \phi^0$ in the next sections.
\subsection{Flow transported solutions}
 The point values of the flow transported solutions are computed by integrations of the flows defined by Eqs.~\eqref{def:flow:0:backward} for $\phi^0$ the approximation \One{of} $u^0$ and Eq.~\eqref{eq:Lap:u0} for that of $\Delta u^0$ thanks to standard \Two{ODE} integrators.
The local absolute error related to the flow transported solution is defined as
\begin{subequations}\label{eqs:def:errors}
	\begin{align}
		\mathcal{E}^A(\x) = \left|u^0(\x)- \phi^0(\x) \right| \,, 
	\end{align}
where $u^0(\x)$ is the function defined by Eq.~\eqref{eq:def:phi0} evaluated at the point with coordinate $\x=(x,y)\in \Omega$ while $\phi^0(\x)$ is the flow transported approximation. A similar expression defines the error for $\Delta \phi^0(\x)$. The relative error cannot be computed at locations where the solution vanishes. A workaround consists in using the following definition
\begin{equation}
	\mathcal{E}^R (\x) = \left\{ \begin{array}[c]{ll}
		\displaystyle \frac{\mathcal{E}^A(\x)}{|u^0(\x)|} \quad &\text{ if } |u^0(\x)| > \varepsilon^\star \,, \\[1em]
		\mathcal{E}^A(\x) & \text{ otherwise} \,.
	\end{array}\right.
\end{equation}
In the sequel the absolute and relative errors will refer to $\mathcal{E}^A$ and $\mathcal{E}^R$ as defined by Eqs.~\eqref{eqs:def:errors}. The errors of the flow transported solutions $\phi^0$ and $\Delta \phi^0$ are plotted on Figs.~\ref{fig:flow:sol:1} and \ref{fig:flow:sol:2} for different values of the parameters $\beta$ and $k$.%
\end{subequations}
\begin{figure}[!ht]\centering
	\subfloat[Absolute error for $\phi^0$.\label{fig:flow:sol:1:a}]{\includegraphics[width=0.49\textwidth]{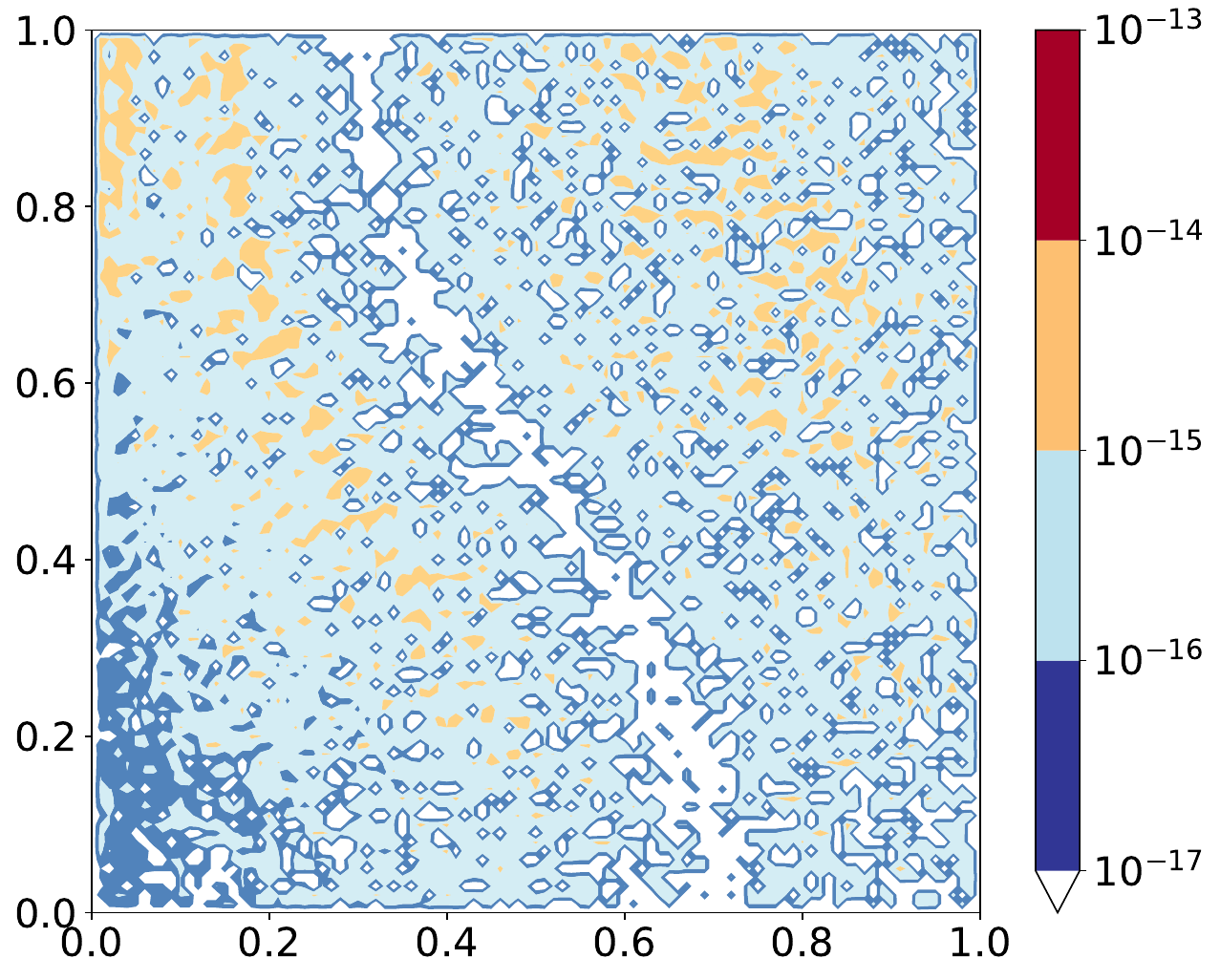}}\hspace*{0.02\textwidth}%
	\subfloat[Absolute error for $\Delta\phi^0$.\label{fig:flow:sol:1:b}]{\includegraphics[width=0.49\textwidth]{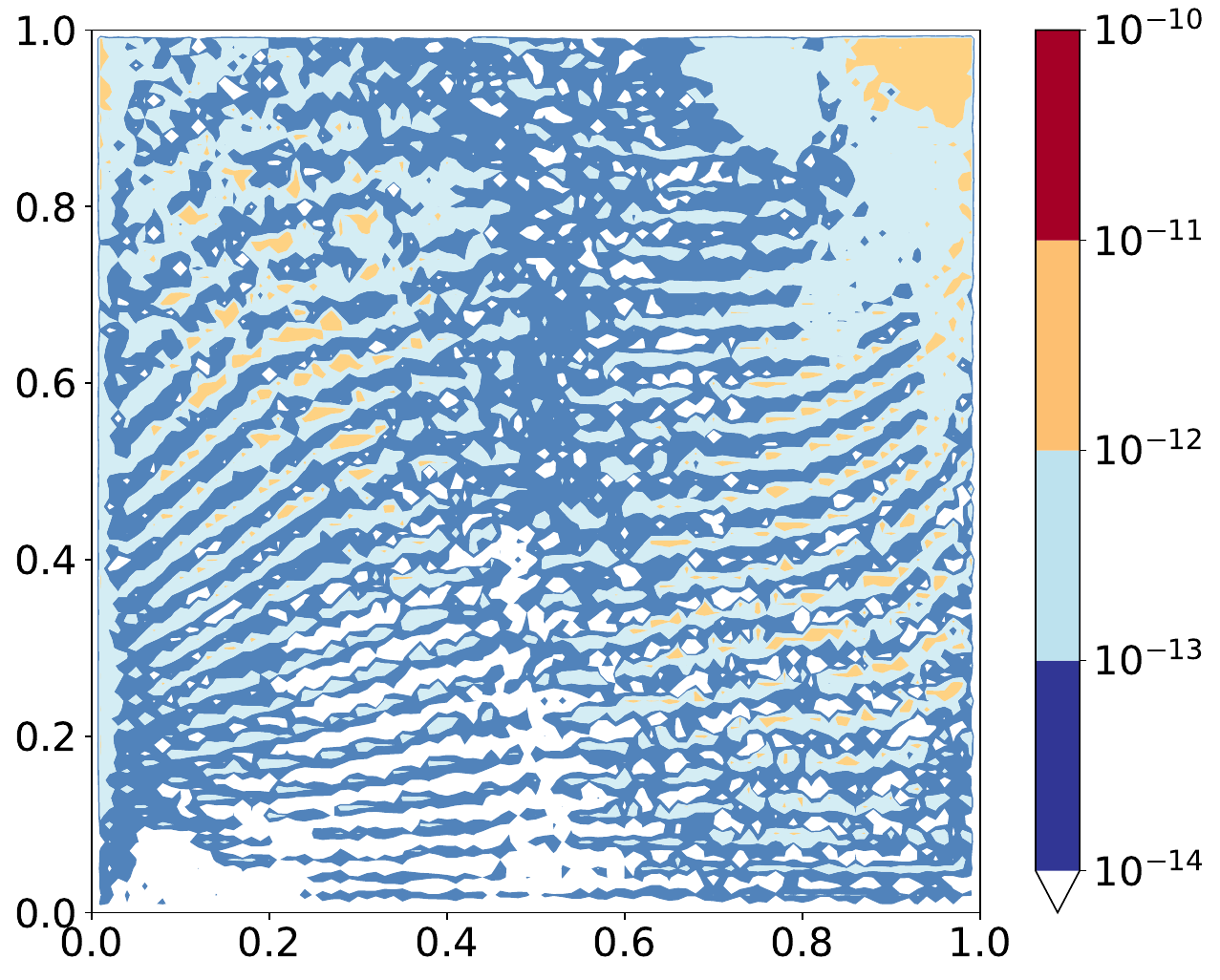}}
	
	\subfloat[Relative error for $\phi^0$.\label{fig:flow:sol:1:c}]{\includegraphics[width=0.49\textwidth]{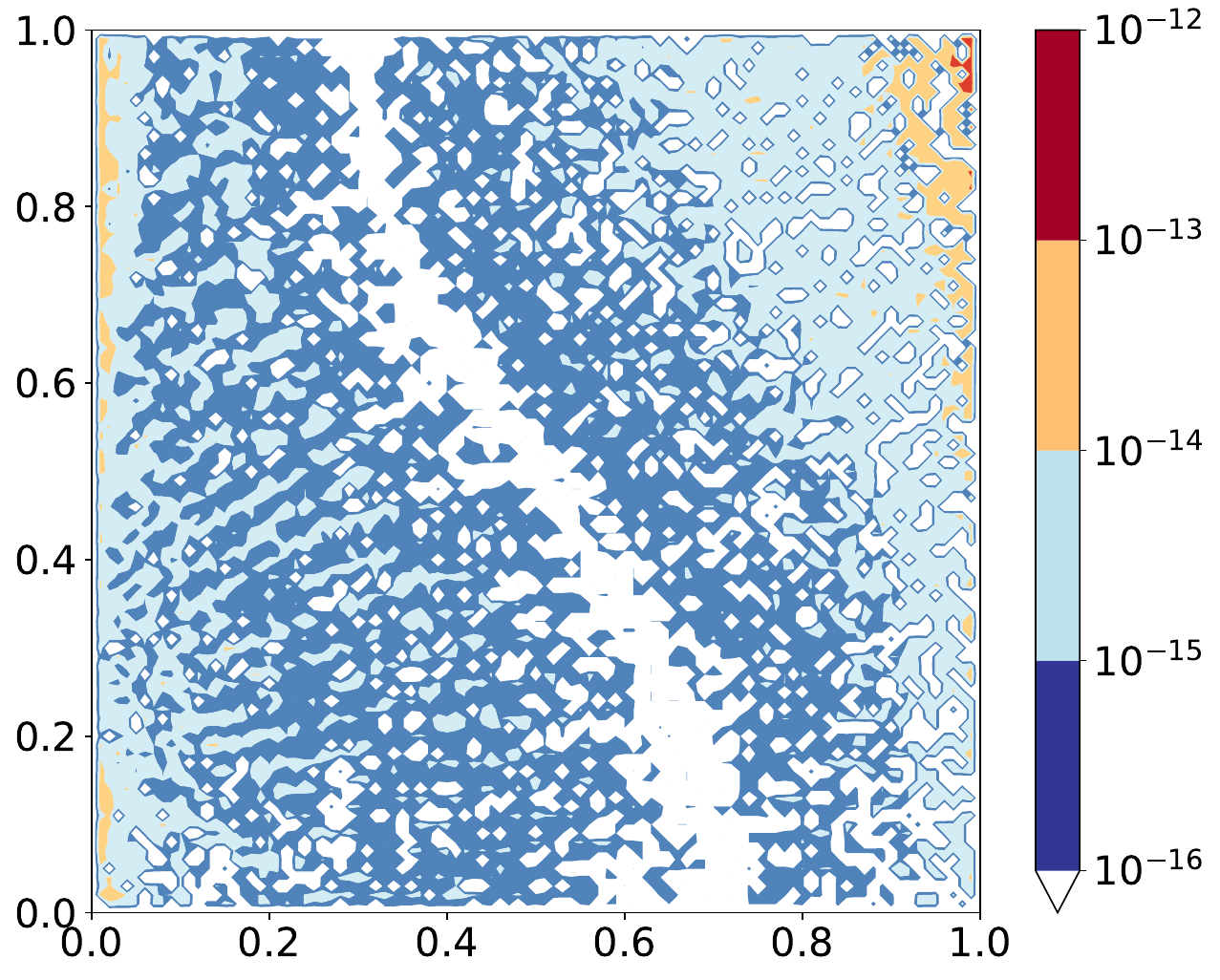}}\hspace*{0.02\textwidth}%
	\subfloat[Relative error for $\Delta\phi^0$.\label{fig:flow:sol:1:d}]{\includegraphics[width=0.49\textwidth]{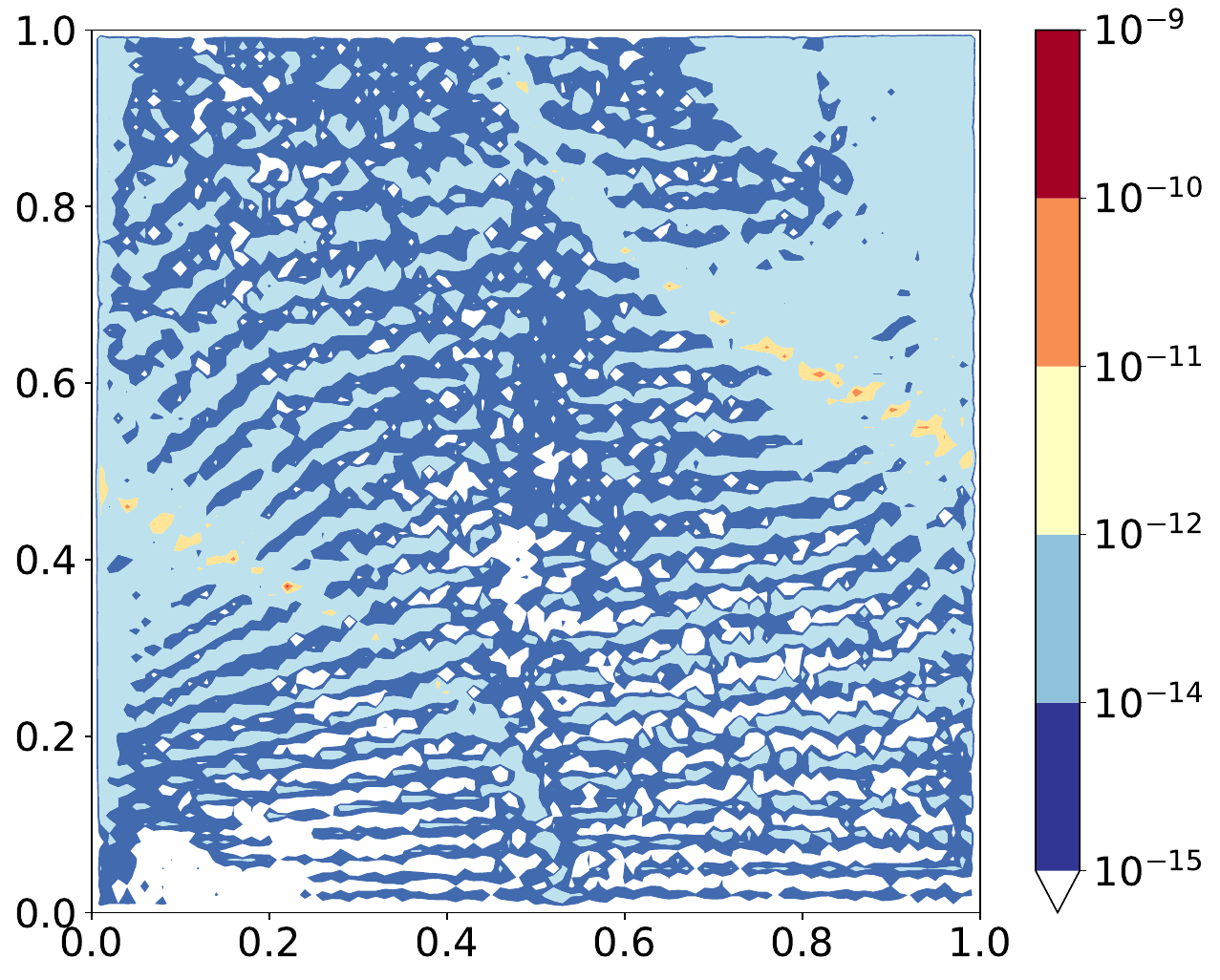}}
	\caption{Precision of the flow transported approximations for $u^0$  and $\Delta u^0$ defined by Eq.~\eqref{eqs:solution:def} for $(\beta,\Two{\ell})=(3,1)$.\label{fig:flow:sol:1}}
	\vspace*{-0.5cm}
\end{figure}
\begin{figure}[!ht]\centering
	\subfloat[Absolute error for $\phi^0$.\label{fig:flow:sol:2:a}]{\includegraphics[width=0.49\textwidth]{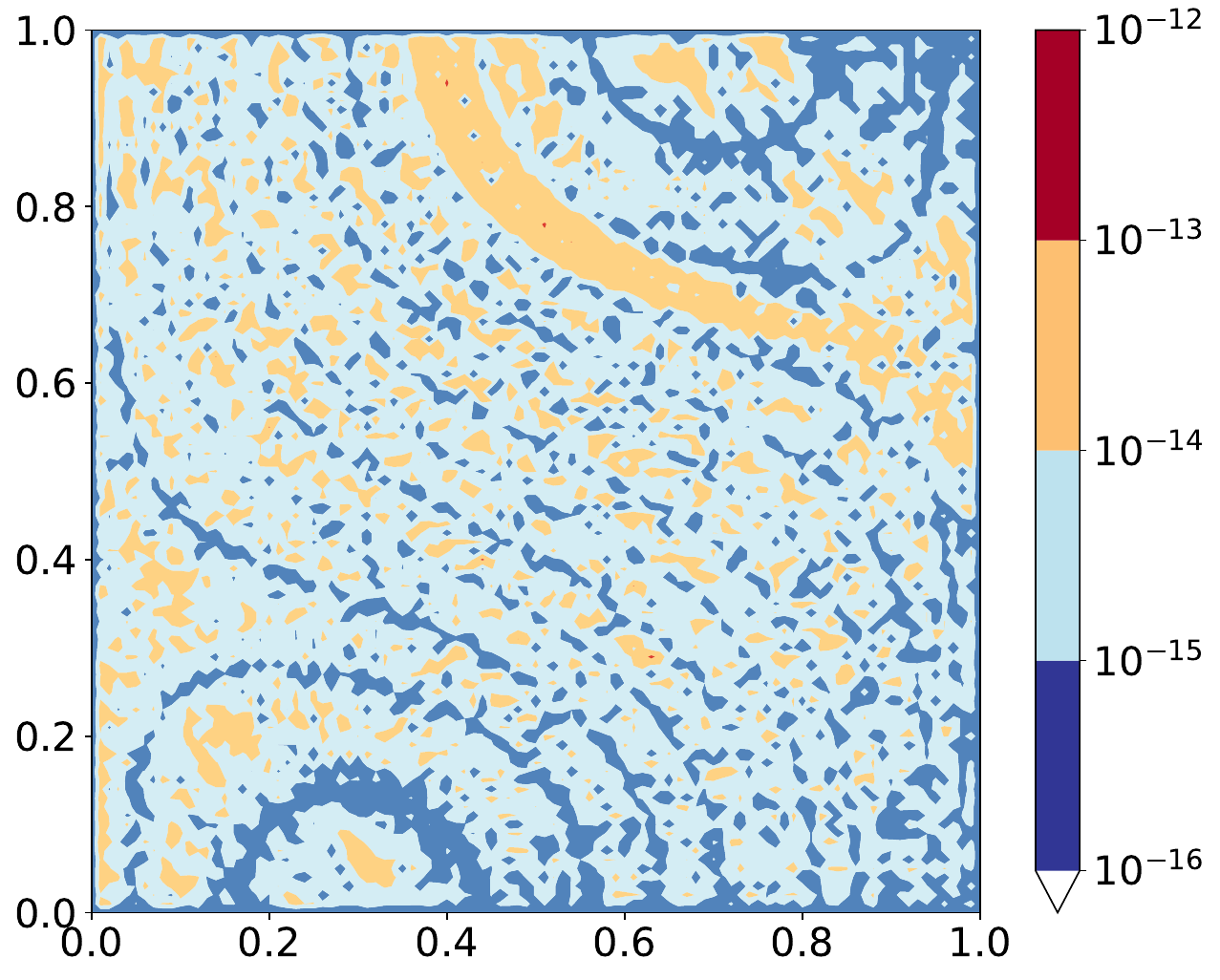}}\hspace*{0.02\textwidth}%
	\subfloat[Absolute error for $\Delta\phi^0$.\label{fig:flow:sol:2:b}]{\includegraphics[width=0.49\textwidth]{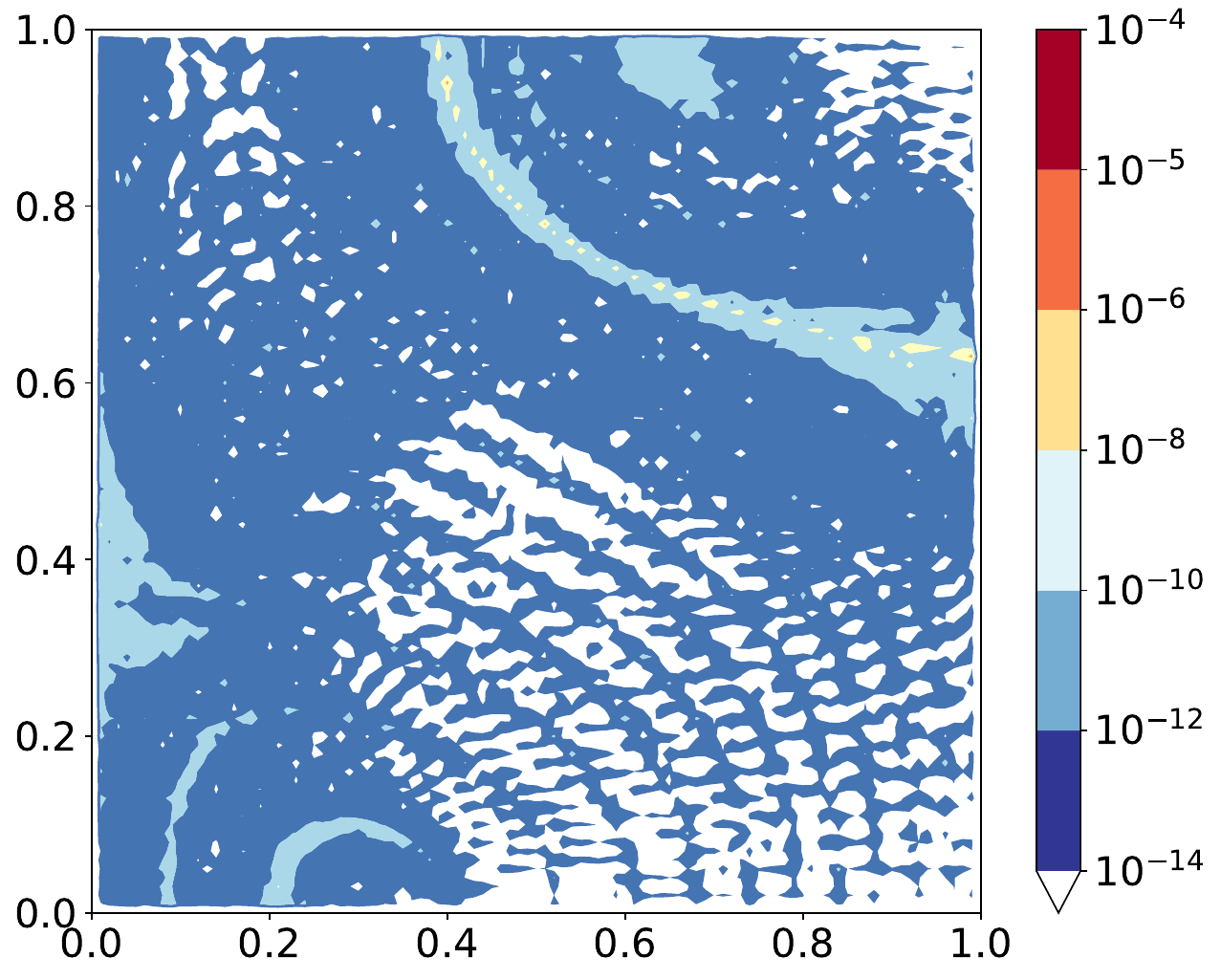}}
	
	\subfloat[Relative error for $\phi^0$.\label{fig:flow:sol:2:c}]{\includegraphics[width=0.49\textwidth]{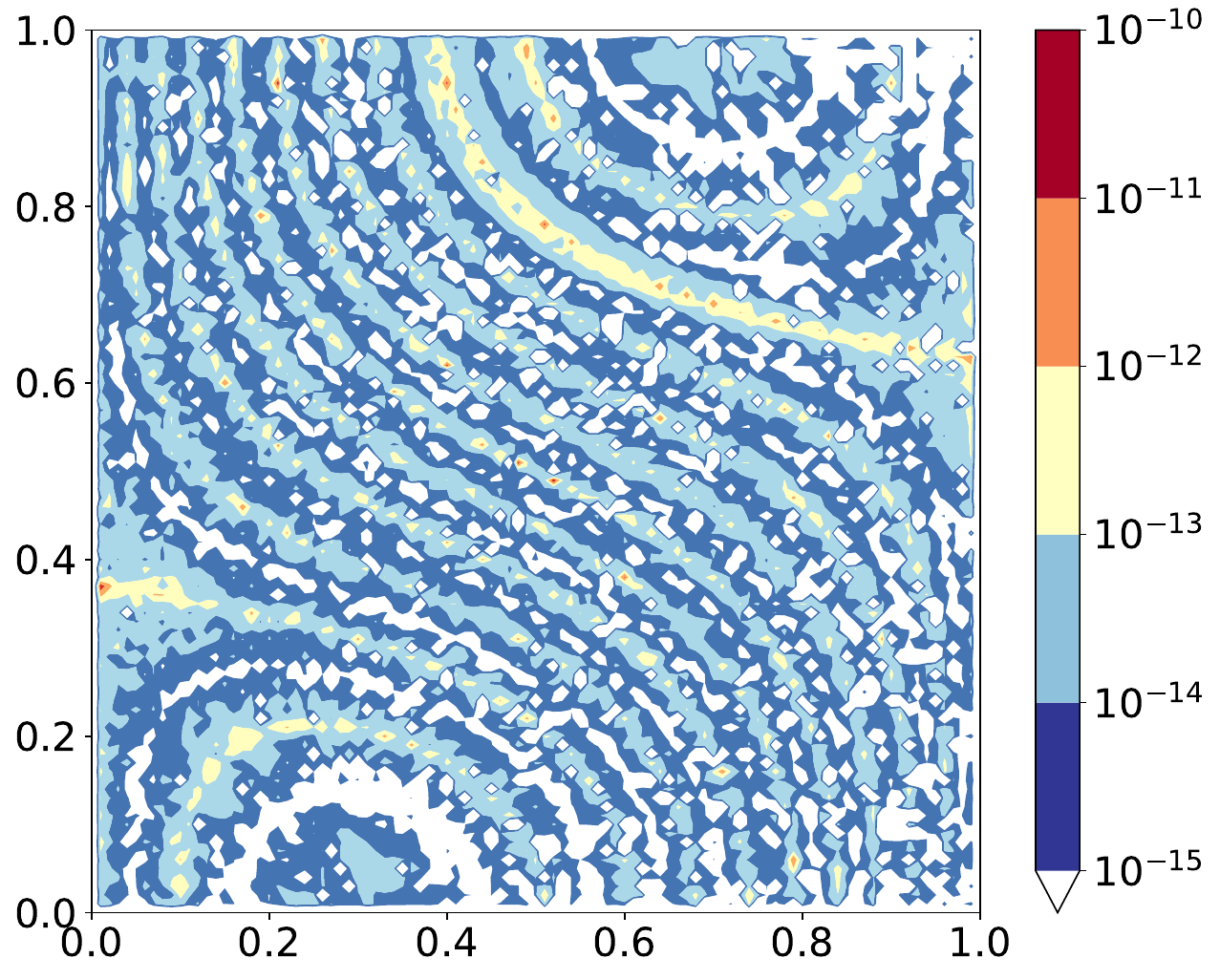}}\hspace*{0.02\textwidth}%
	\subfloat[Relative error for $\Delta\phi^0$.\label{fig:flow:sol:2:d}]{\includegraphics[width=0.49\textwidth]{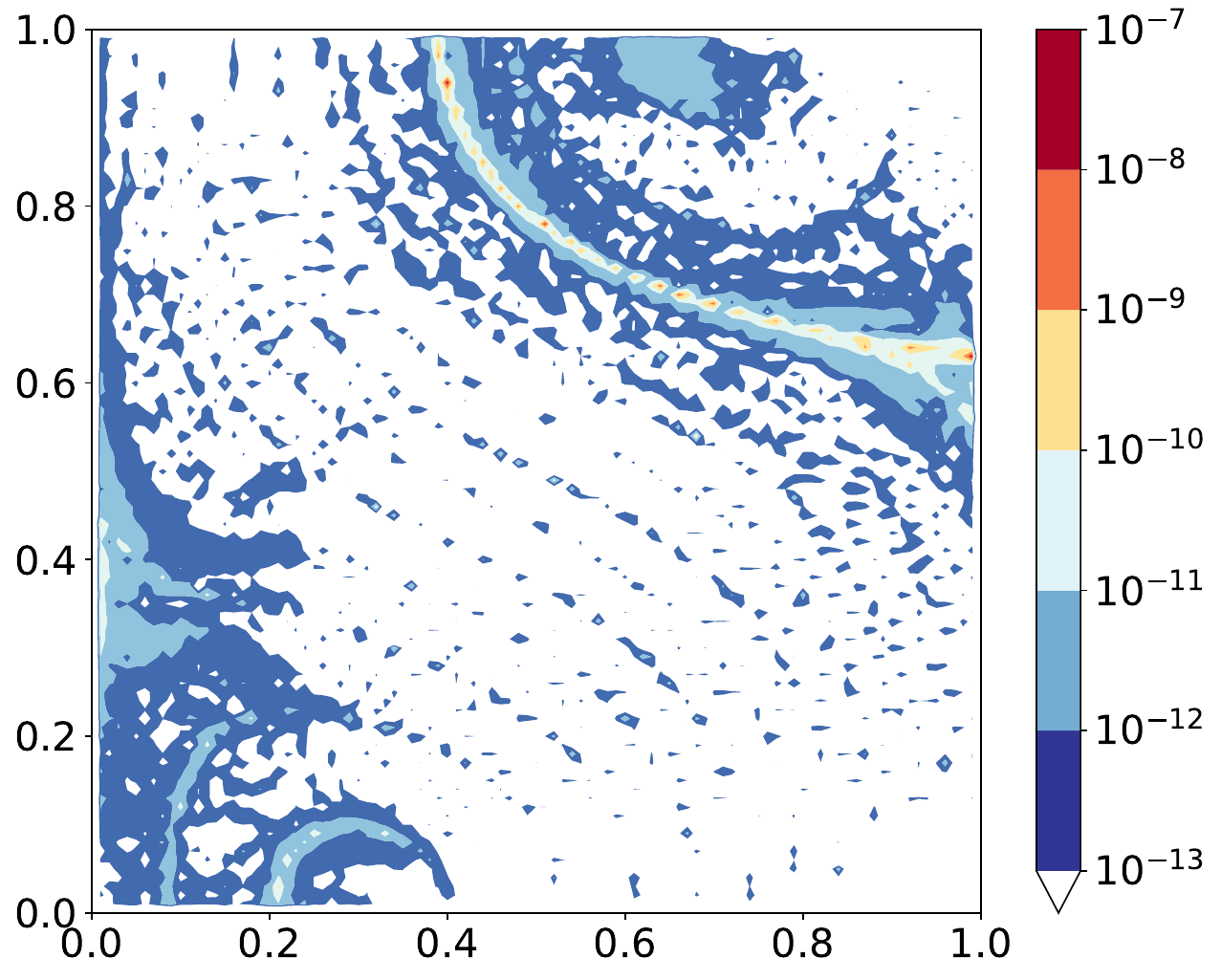}}
	\caption{Precision of the flow transported approximations of $u^0$  and $\Delta u^0$ defined by Eq.~\eqref{eqs:solution:def} for $(\beta,\Two{\ell})=(8,8)$.\label{fig:flow:sol:2}}
	\vspace*{-0.5cm}
\end{figure}
These plots display, on a color scale, the local errors on points distributed into the domain. The  computations are carried out using an explicit high order Runge-Kutta method. Different \Two{ODE} solvers have been tested, the computations carried out within this document are performed using the (9/8) Verner's \cite{verner2010numerically}  method proposed within the \verb+Julia+ package \cite{rackauckas_differentialequationsjl_2017})  with an absolute tolerance $\tau_\text{ODE}$ set to $10^{-12}$.

The parameters $\beta$ and \Two{$\ell$} parametrize the variations in the $b$-field as well as the perpendicular gradients of the solution, the higher their values the more demanding the benchmark for the accuracy of the transported solutions. 
The relative error on $\phi^0$ increases slightly for  $(\beta,\Two{\ell})=(8,8)$ as observed on Figs.~\ref{fig:flow:sol:2:a} and \ref{fig:flow:sol:2:c},  as compared to Figs.~\ref{fig:flow:sol:1:a} and \ref{fig:flow:sol:1:c} obtained with $(\beta,\Two{\ell})=(3,1)$.

The precision of the ODE integrator does not readily give access to an estimator of the flow transported solution precision. Nonetheless, the accuracy of $\phi^0(\x)$ may be estimated from that of the ODE integrator. The flow transported solution is obtained by computing $\x^\star$ the approximated value of $\x_0$ the foot of the field line issued from $\x$ by integrating the backward flow $\hat\Phi_t$, yielding, owing to the property of $u^0$ to be constant along the $b$-field lines, $\phi^0(\x) = u^0(\x^\star)$.
The ODE integrator precision threshold denoted $\tau_\text{ODE}$ provides a control on the accuracy of $\x^\star$ with
\begin{equation}
	\left\| \x^\star - \x_0\right\| \leq \tau_\text{ODE} \,.
\end{equation}
Since $\x_0$ is located on one domain boundary for which the value of the coordinate $y$ is known and equal to $\yp $, the approximated value $\x^\star$ is substituted with $\x^\star = (x^\star,\yp )$ so that
the precision of the transported  approximation reduces to
\begin{equation}\label{eq:phi0:prec:estimate}
	\begin{aligned}
		\left|\phi^0(\x) - u^0(\x) \right| = \left|u^0(\x^\star) - u^0(\x_0) \right| &= \left|u^0(x^\star,\yp ) - u^0(x_0,\yp )  \right|\,,\\
		&\leq \varepsilon_\text{ODE}= \tau_\text{ODE} \, \left| \frac{\partial}{\partial x} u^0 (x_0,\yp ) \right| \,. 
	\end{aligned}
\end{equation}
From this estimate, the precision of the numerical approximation is anticipated to decrease linearly with the ODE integrator threshold ($\tau_\text{ODE}$), while deteriorating with the increase of the solution variations, hence with the values of both $\beta$ and $k$. Indeed, the derivatives of the solution trace increases with the values of these parameters, a feature illustrated by the plots of Fig.~\ref{fig:trace:derivative}.%
\begin{figure}[!ht]\centering
	\subfloat[$\beta\in\{3,8\}, \, \Two{\ell}=1$.]{\includegraphics[height=0.33\textwidth]{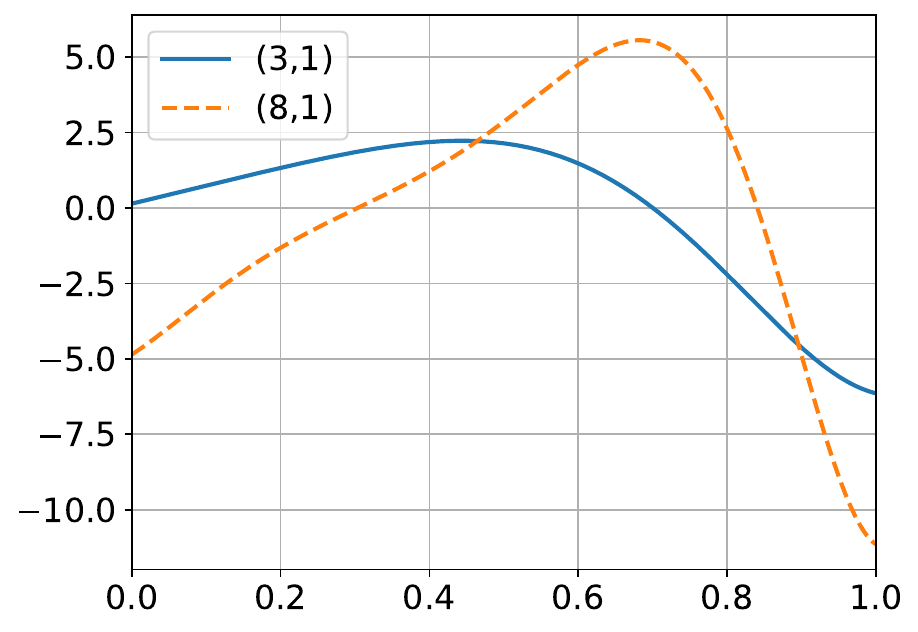}}\hspace*{0.05\textwidth}%
	\subfloat[$\beta\in\{3,8\}, \, \Two{\ell}=8$.]{\includegraphics[height=0.33\textwidth]{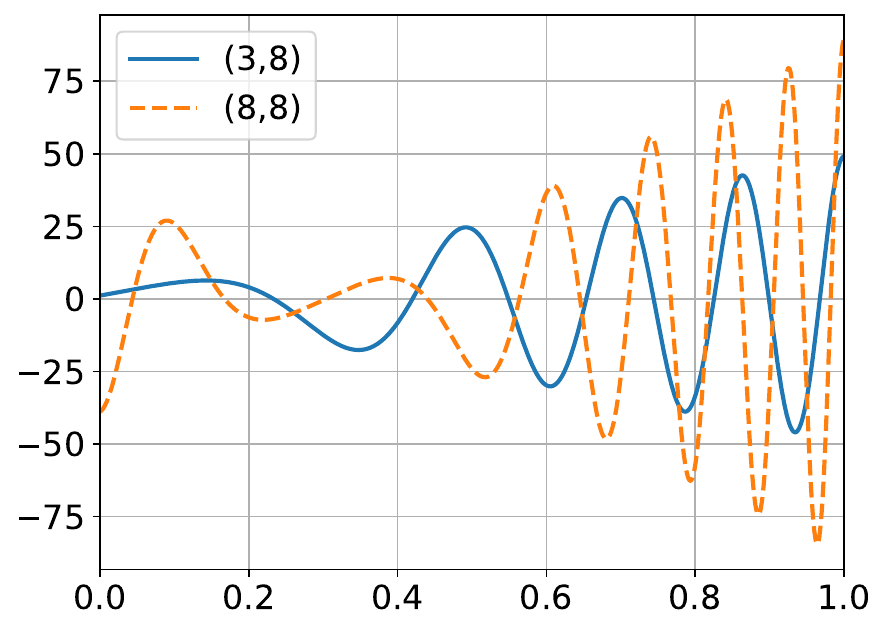}}
	\caption{Derivative of the solution trace on the boundary $y=0$ ($\partial_x u^0(x,0)$) for different values of the parameters $(\beta,\Two{\ell})$ as a function of $x$ , $u^0$ being defined by Eq.~\eqref{eq:def:phi0}.}\label{fig:trace:derivative}
	\vspace*{-0.5cm}
\end{figure}
The accuracy of the solution approximation as a function of $\tau_\text{ODE}$ is plotted on Figs.~\ref{fig:flow:sol:absolute:errors} and \ref{fig:flow:sol:relative:errors} for the absolute and relative errors.
\begin{figure}[!ht]\centering
	\subfloat[Absolute error $\ell_1$-norm for $\phi^0$.\label{fig:flow:sol:absolute:errors:a}]
	{\includegraphics[width=0.49\textwidth]{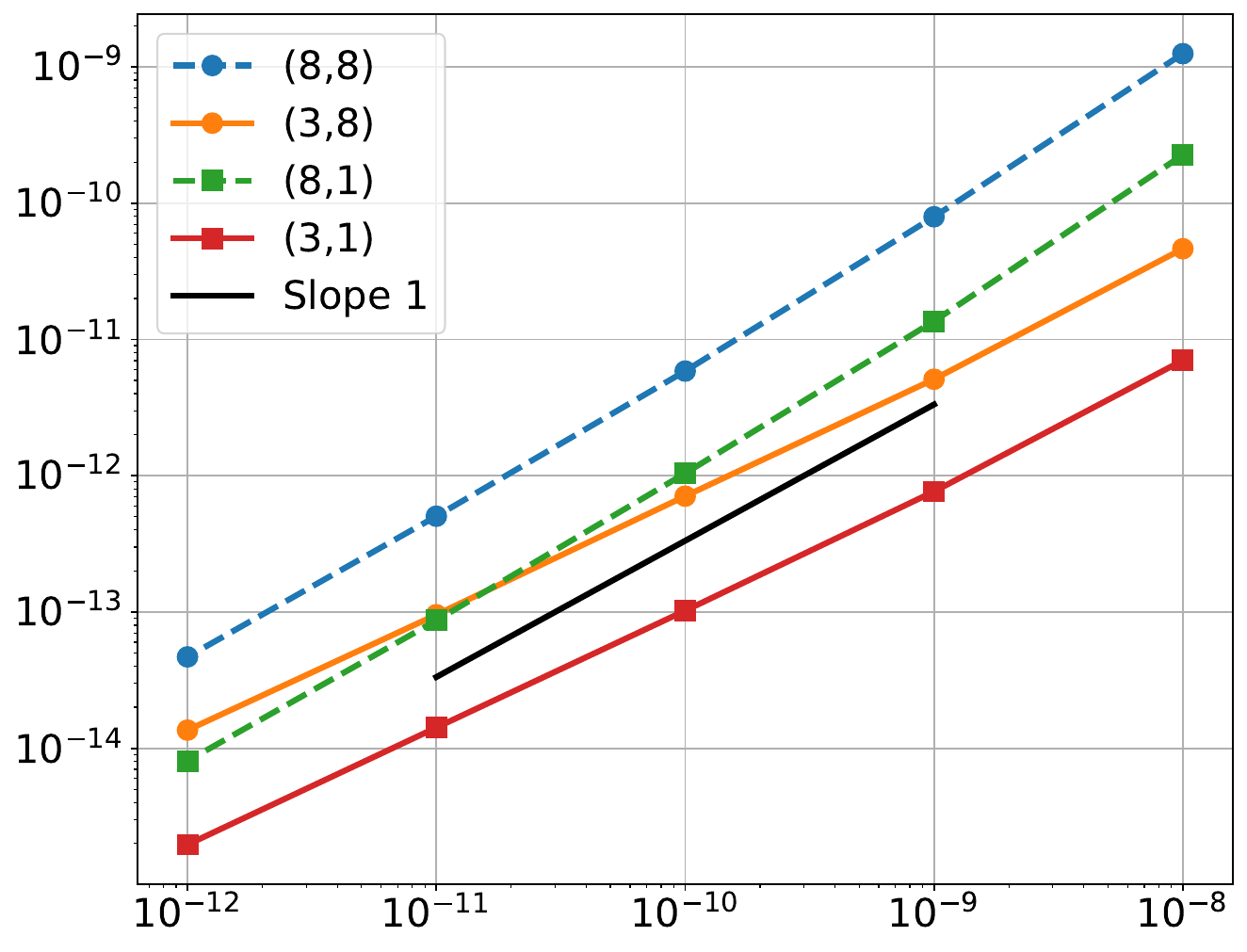}}\hspace*{0.02\textwidth}%
	\subfloat[Absolute error in $\ell_\infty$-norm for $\phi^0$.\label{fig:flow:sol:absolute:errors:b}]
	{\includegraphics[width=0.49\textwidth]{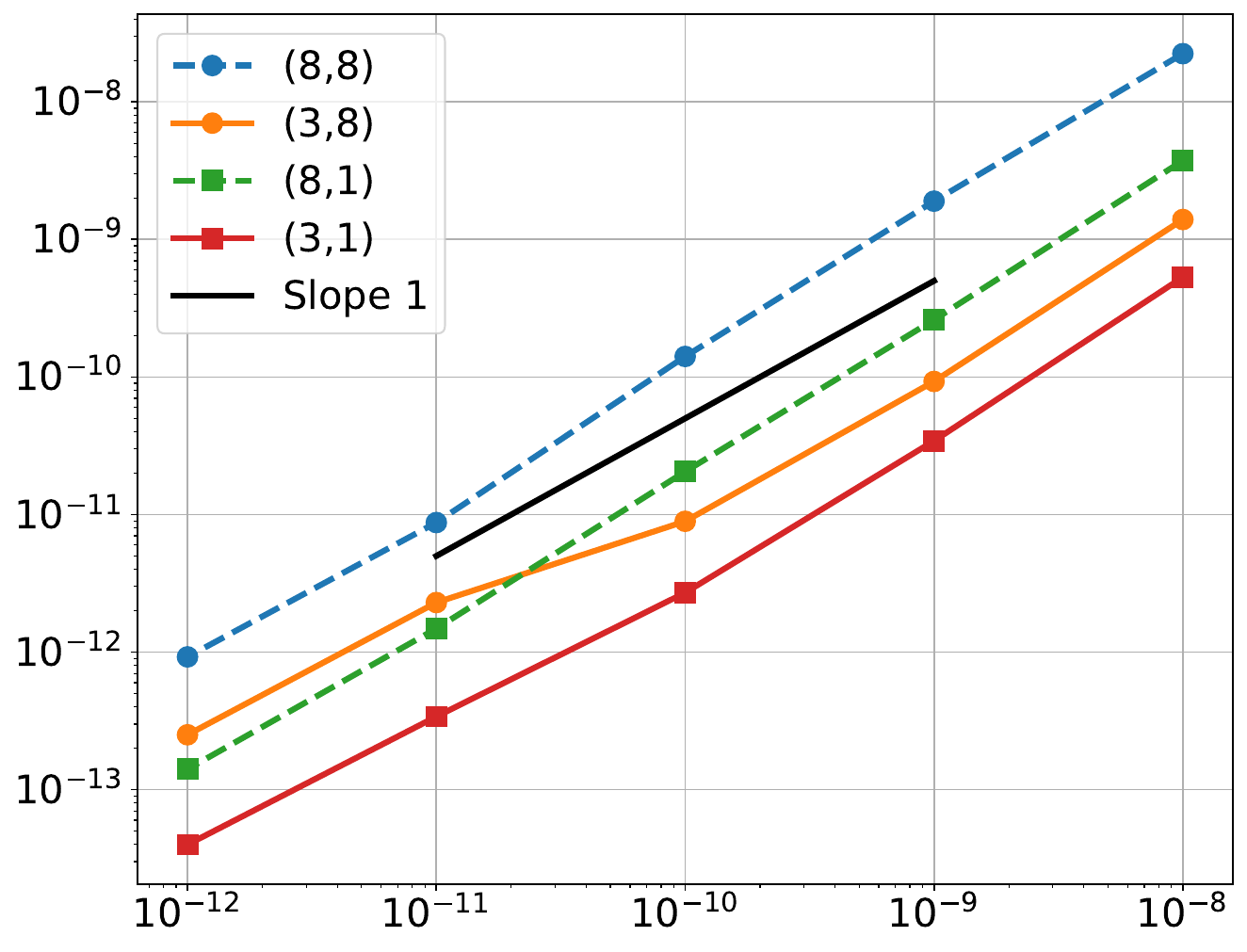}}
	
	\subfloat[Absolute error $\ell_1$-norm for $\Delta \phi^0$.\label{fig:flow:sol:absolute:errors:c}]{\includegraphics[width=0.49\textwidth]{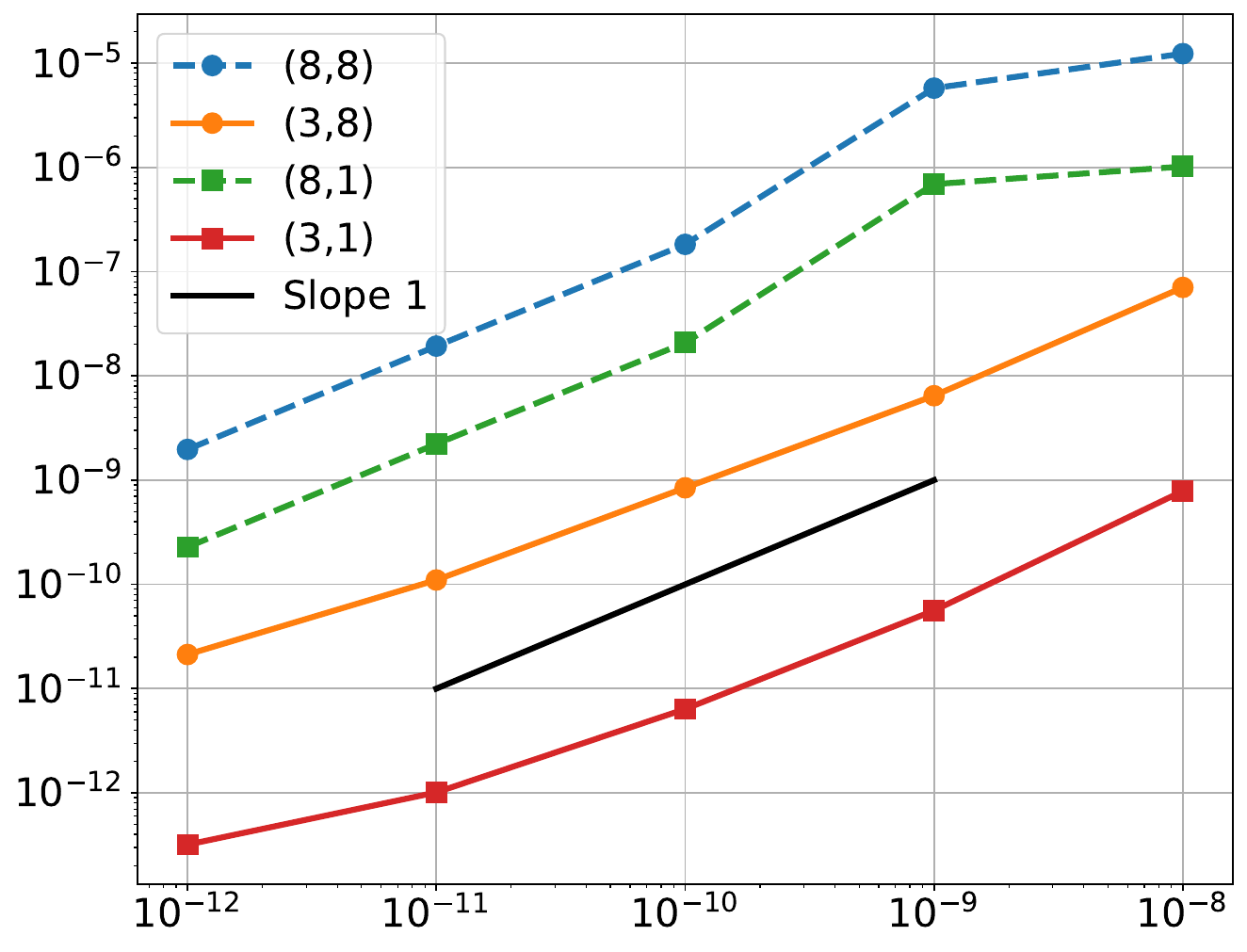}}\hspace*{0.02\textwidth}%
	\subfloat[Absolute error $\ell_\infty$-norm for $\Delta\phi^0$.\label{fig:flow:sol:absolute:errors:d}]{\includegraphics[width=0.49\textwidth]{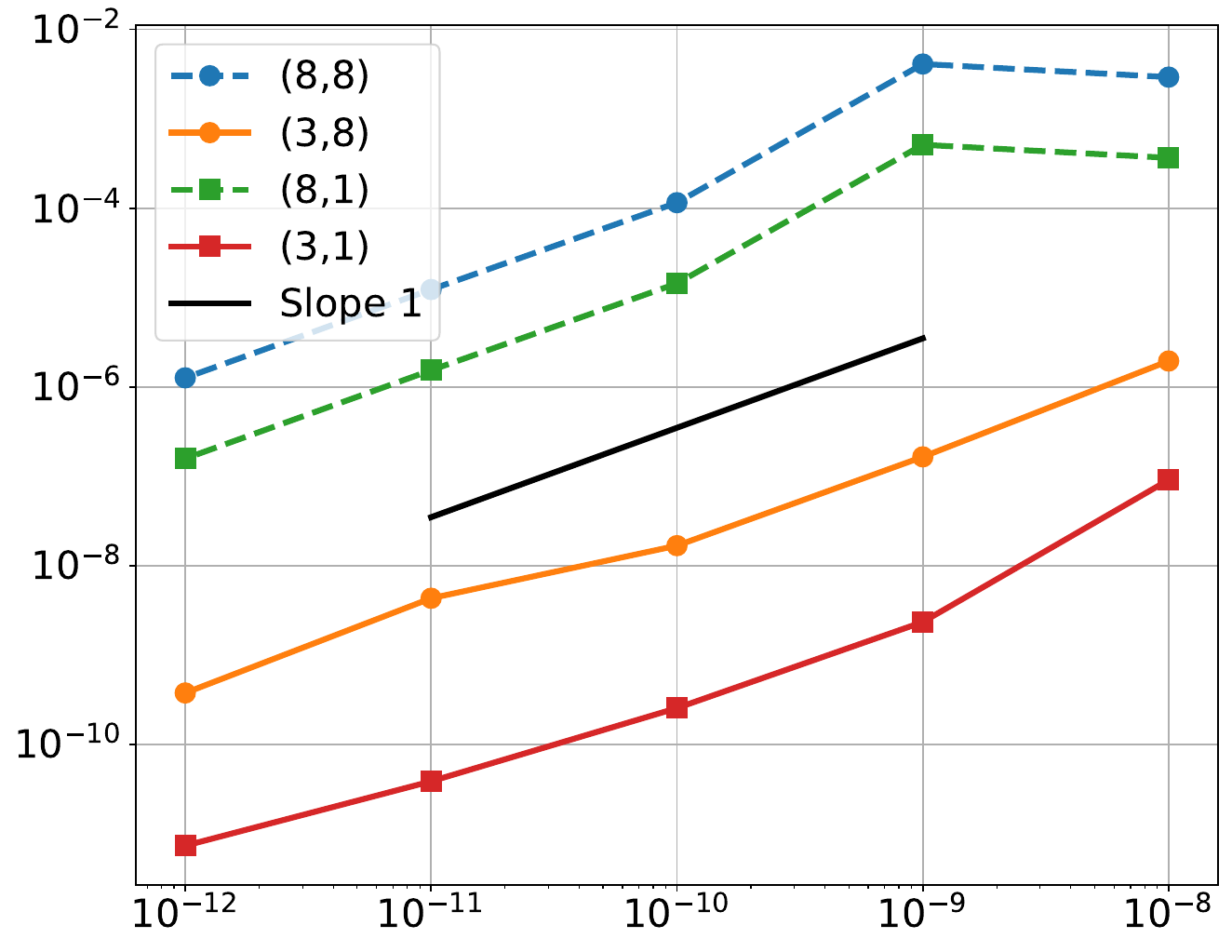}}
	\caption{Precision of the flow transported solutions $\phi^0$ and $\Delta \phi^0$ defined by Eqs.~\eqref{eqs:solution:def}: absolute error norms as functions of the ODEs numerical integrator precision threshold $\tau_\text{ODE}$. The tuples in the legends are the values of the parameters $\beta$ and $\Two{\ell}$.}\label{fig:flow:sol:absolute:errors}
	\vspace*{-0.5cm}
\end{figure}
\begin{figure}[!ht]\centering
	\subfloat[Relative error $\ell_1$-norm for $\phi^0$.]{\includegraphics[width=0.49\textwidth]{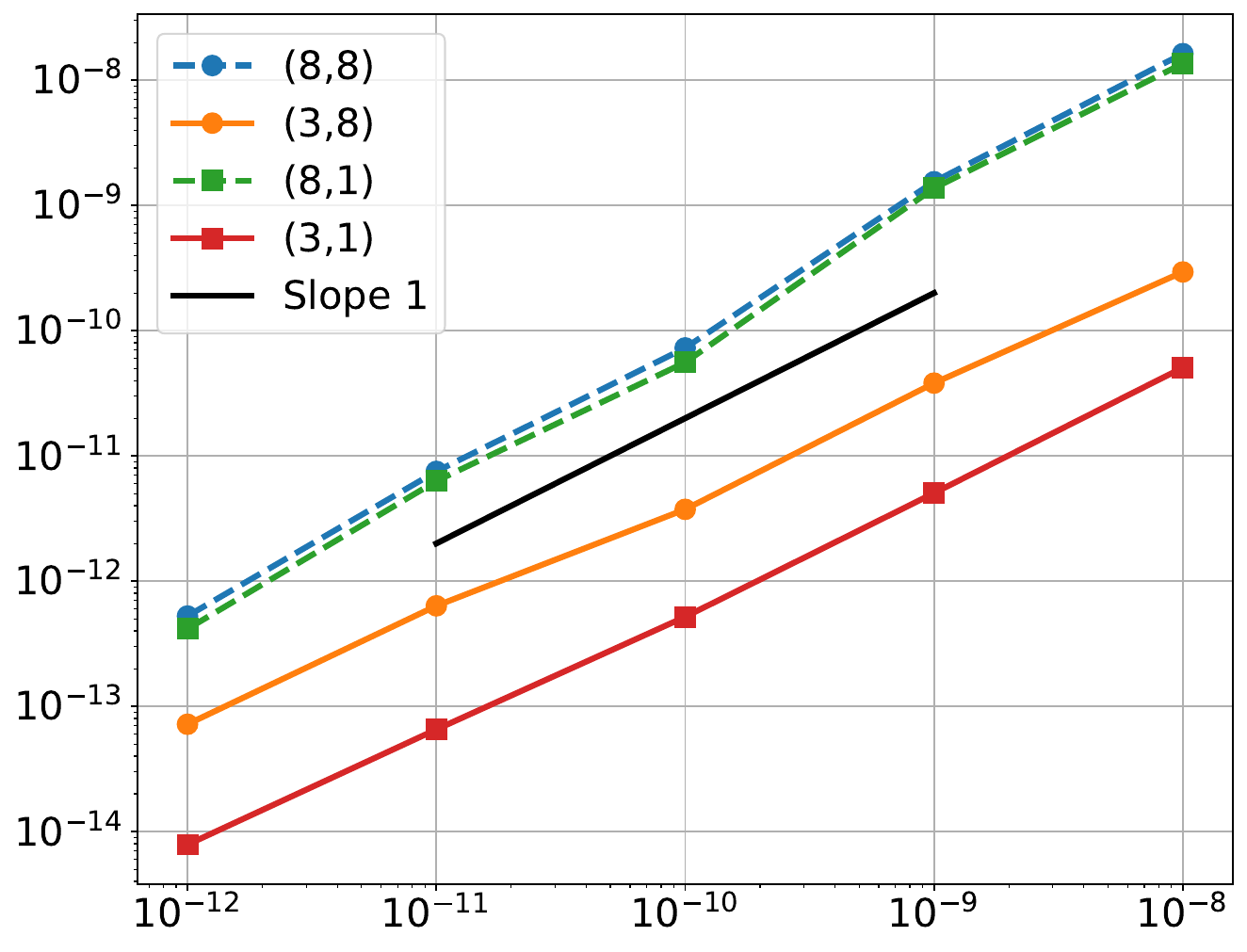}}\hspace*{0.02\textwidth}%
	\subfloat[Relative error in $\ell_\infty$-norm for $\phi^0$.]{\includegraphics[width=0.49\textwidth]{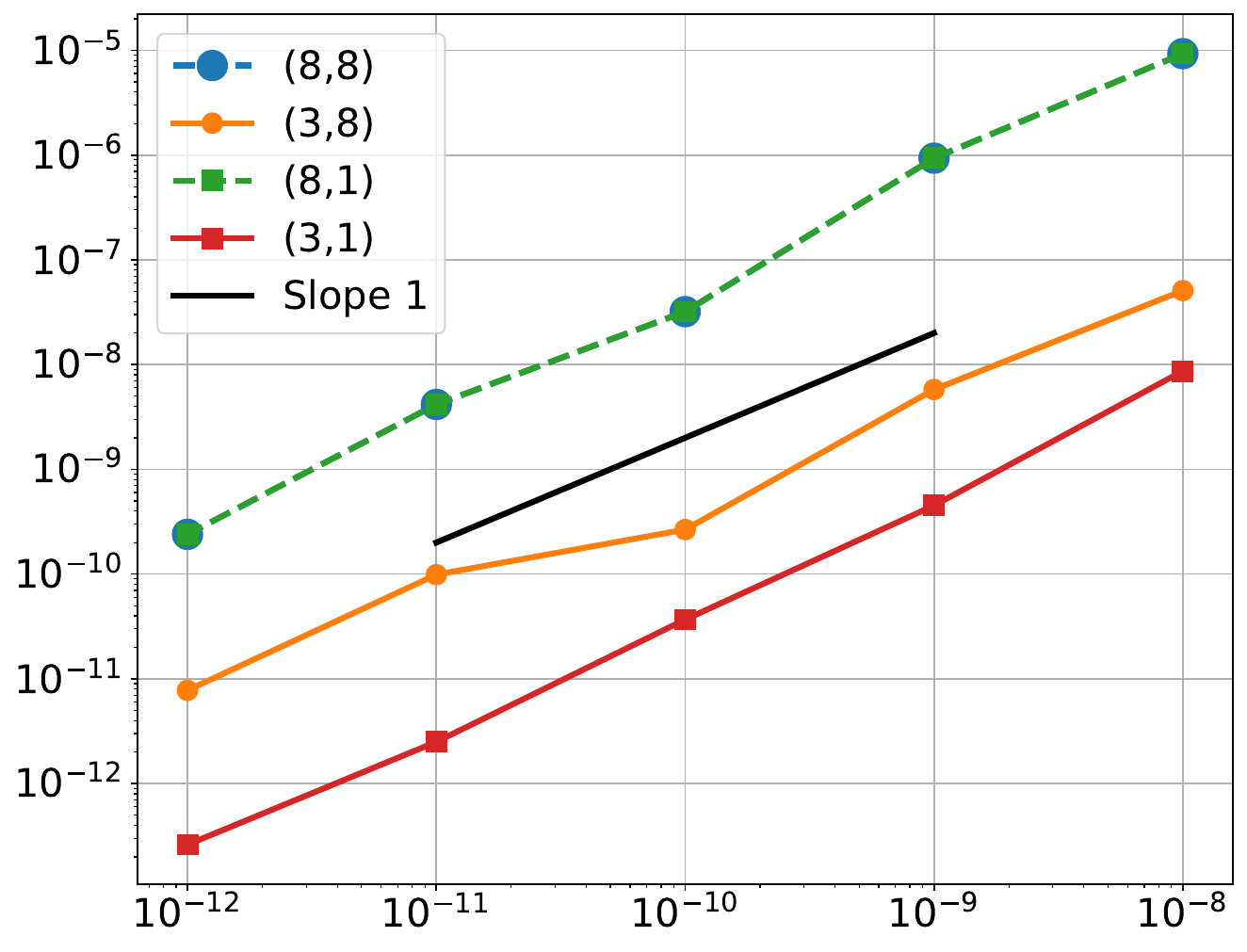}}
	
	\subfloat[Relative error $\ell_1$-norm for $\Delta \phi^0$.]{\includegraphics[width=0.49\textwidth]{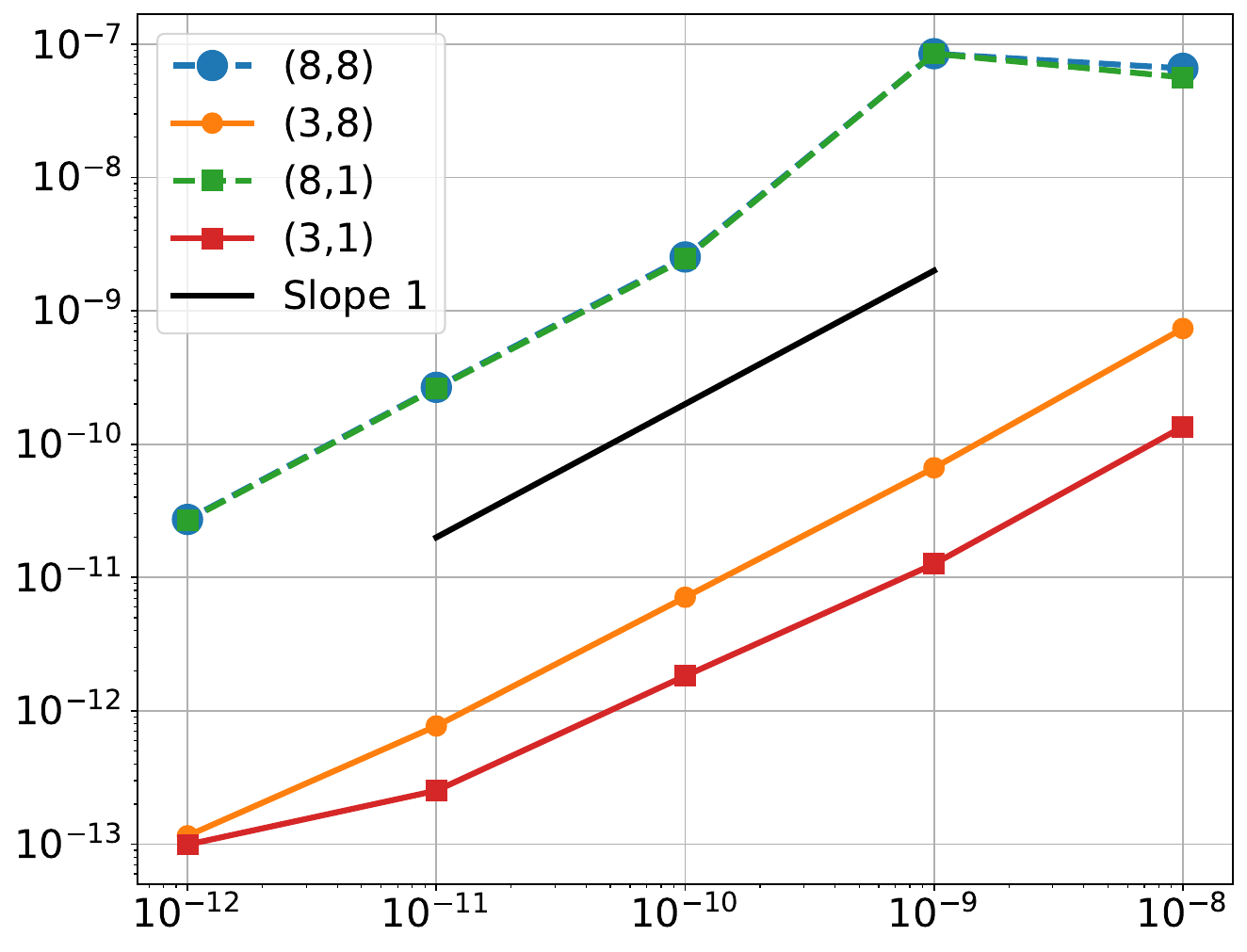}}\hspace*{0.02\textwidth}%
	\subfloat[Relative error $\ell_\infty$-norm for $\Delta\phi^0$.]{\includegraphics[width=0.49\textwidth]{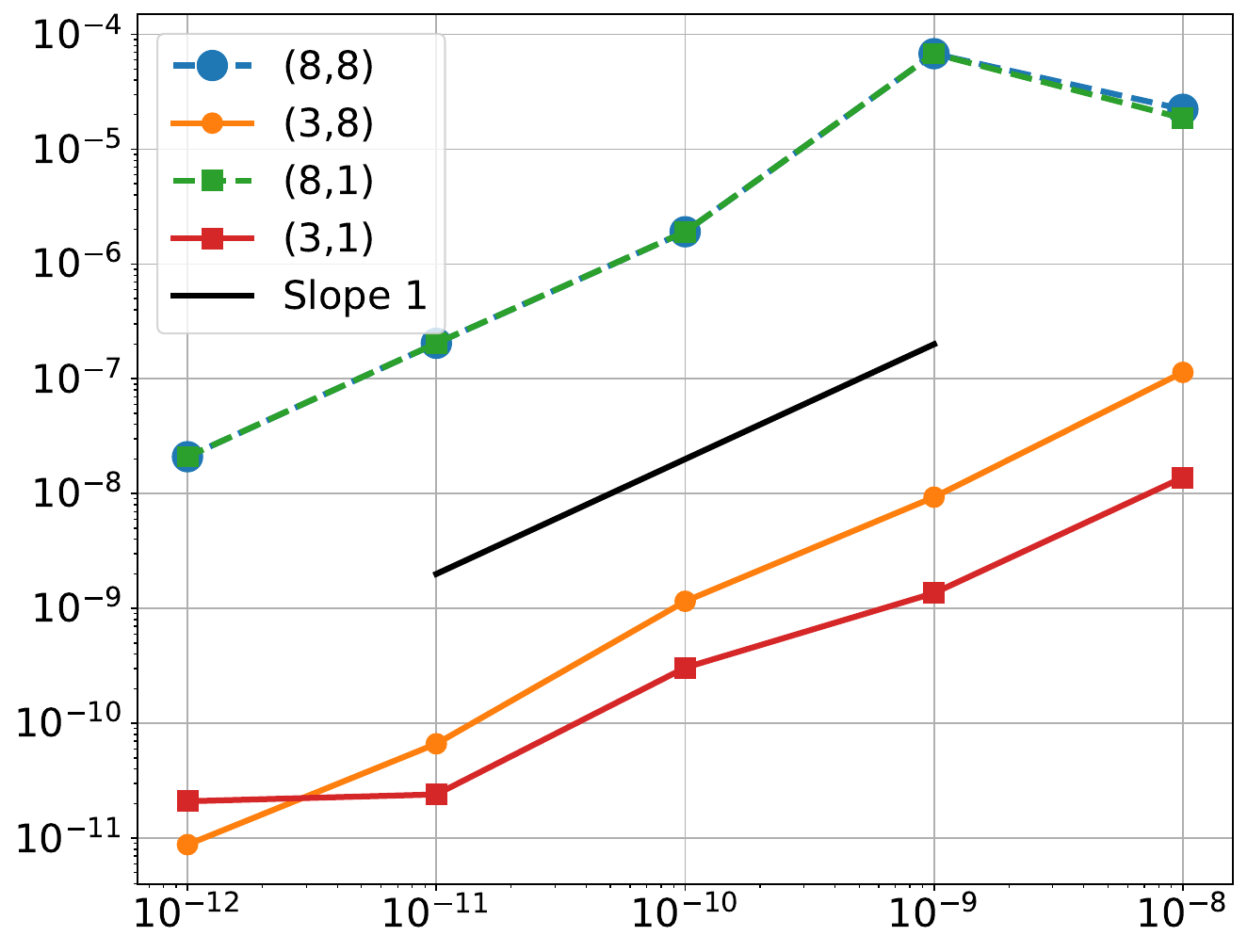}}
	\caption{Precision of the flow transported solutions $\phi^0$ and $\Delta \phi^0$ defined by Eqs.~\eqref{eqs:solution:def}: relative error norms as functions of the ODEs numerical integrator precision threshold $\tau_\text{ODE}$. The tuples in the legends are the values of the parameters $\beta$ and $\Two{\ell}$.}\label{fig:flow:sol:relative:errors}
	\vspace*{-0.5cm}
\end{figure}
The linear decrease of the error with the precision tolerance of the ODE integrator is ascertained by these plots. The increase of the error with the values of the parameters $\beta$ and $k$ controlling the derivatives of $u^0$ on the boundary is also in line with the estimate stated by Eq.~\eqref{eq:phi0:prec:estimate}.

Regarding the anisotropic problem source term contribution, $\Delta u^0$, the absolute and relatives errors of the flow transported approximation are represented on a color scale on Figs.~\ref{fig:flow:sol:1:b}, \ref{fig:flow:sol:1:d}, \ref{fig:flow:sol:2:b} and \ref{fig:flow:sol:2:d}. The same dependency to the ODE integrator tolerance and magnitude of the derivative of solution trace on the boundary is observed on the plots of Figs.~\ref{fig:flow:sol:absolute:errors} and \ref{fig:flow:sol:relative:errors}. The question of the derivation of an estimate of this quantity precision is investigated in the following section.

\subsection{An estimator of the precision of the flow transported solutions}
The precision of the flow transported solution $\phi^0$ may be estimated by evaluating its parallel gradient. Assuming that $u^0$ is regular enough ($u^0\in C^{m+1}(\Omega)$) a high order finite difference scheme may be used to compute the parallel gradient of the flow transported solution. \Two{Let $\partial_x^{h,m}$ and $\partial_y^{h,m}$ denote the discrete partial differential operators with respect to the $x$ and $y$ coordinates, where $h$ represents the spatial step size and $m$ indicates the order of approximation accuracy. These operators are defined as follows:}
\begin{align}\label{eq:Fornberg:formulas:1}
	&\big(\partial_x^{h,m} u^0\big)(x,y) = \frac{1}{h}\sum_{k={-m/2}}^{k=m/2} \omega_k^{(1)}\, u^0(x+k\cdot h,y) = \frac{\partial u^0}{\partial x}(x,y)+ \mathcal{O}(h^m) \,,\\
	&\nabla^{h,m} u^0(x,y) = \Big(\left(\partial_x^{h,m}u^0\right)(x,y), \left(\partial_y^{h,m}u^0\right)(x,y) \Big)^T\,.
\end{align}
\Two{The operator $\partial_y^{h,m}$ is defined analogously to $\partial_x^{h,m}$utilizing function evaluations at distinct $y$-coordinates. These discrete operators incorporate the weights $(\omega_k^{(1)})_k$ derived from Fornberg's formulas for $m\in \{4,6,8\}$. The operator $\partial_y^{h,m}$ is defined analogously to $\partial_x^{h,m}$ by evaluating the function at distinct y-coordinates.} 
Owing to the estimate stated by Eqs.~\eqref{eq:phi0:prec:estimate} and \eqref{eq:Fornberg:formulas:1}, the following bound holds true for the estimator of the flow transported solution parallel gradient 
\begin{equation}\label{eq:estimate:gradPara}
	\begin{aligned}
	\left|\big(b\cdot\nabla^{h,m}\phi^0\big)(x,y)\right| &\leq \mathcal{O}(h^m) + \frac{\widetilde{\varepsilon}_\text{ODE}}{h} \,,\qquad 
	\widetilde{\varepsilon}_\text{ODE} &=  \, \big| \sum_{k=-m/2}^{m/2} \omega_k^{(1)} \big| \, {\varepsilon}_\text{ODE}\,.
	\end{aligned}
\end{equation}
The absolute value of the sum of the weights defining the Fornberg's formulas implemented to define the discrete differential operators are bounded by 1, hence $\widetilde{\varepsilon}_\text{ODE} \sim {\varepsilon}_\text{ODE}$. The computation of the parallel gradient estimator is carried out by multiple evaluation of the flow transported solution: $\phi^0(x+k \cdot h,y)$, $\phi^0(x,y+\ell \cdot h)$ for $-m/2\leq k,\ell\leq m/2$ for arbitrary values of $h$. This quantity as a function of the discretization parameter $h$ is plotted on Fig.~\ref{fig:GradPara:Error} for $4^\text{th}$, $6^\text{th}$ and $8^\text{th}$ order accurate finite difference discretizations ($m\in\{4,6,8\}$) at two different spatial locations. 
\begin{figure}[!ht]\centering
	\subfloat[$\x=(0.3,0.9)$]{\includegraphics[width=0.49\textwidth]{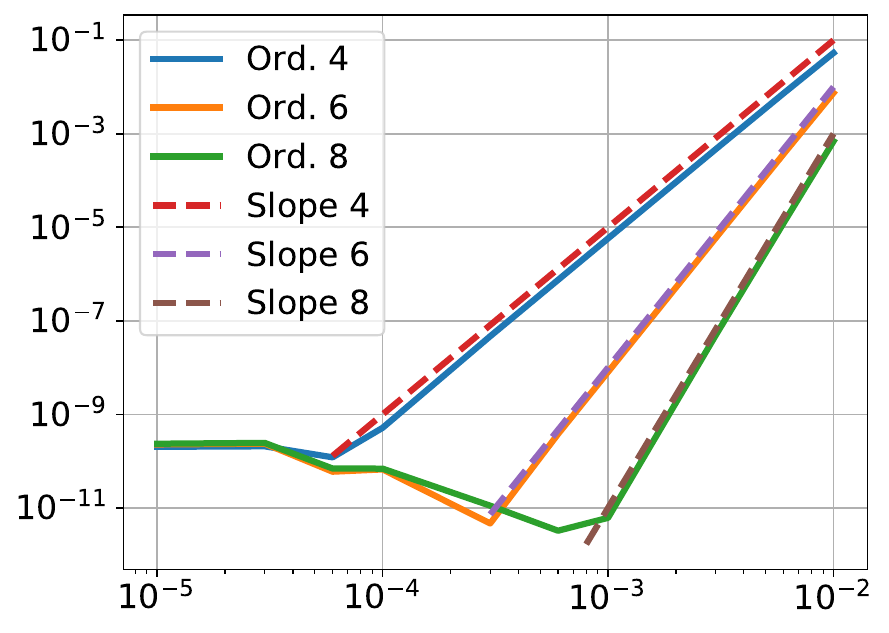}}\hspace*{0.02\textwidth}%
	\subfloat[$\x=(0.4,0.94)$]{\includegraphics[width=0.49\textwidth]{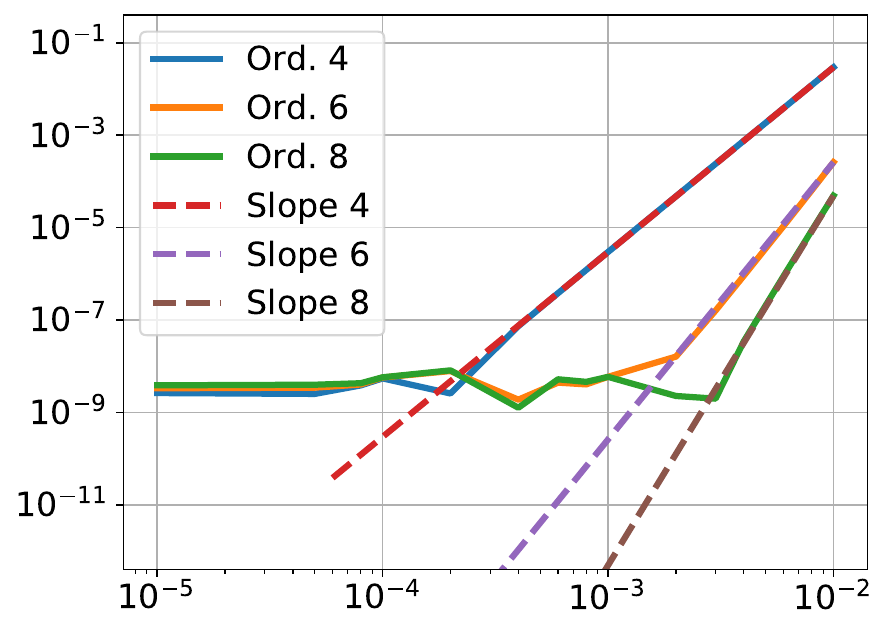}}
	\caption{Magnitude of the parallel gradient of $\phi^0$ carried out thanks to $4^\text{th}$, $6^\text{th}$, $8^\text{th}$ order accurate finite difference approximations as functions of the step size $h$. The computations are carried out with $(\beta,k)=(8,8)$.}\label{fig:GradPara:Error}
	\vspace*{-0.2cm}
\end{figure}
The error component originating from the ODE solver is negligible for large values of the step size $h$, as observed on the plots of Fig.~\ref{fig:GradPara:Error}. The error decreases with a slope $m$ related to the approximation of the finite difference approximation until the magnitude of the component due $\widetilde{\varepsilon}_\text{ODE}$ is reached. From this value, denoted $h^\dagger$, the decrease of the step size increases the error. An estimation of the flow transported accuracy may be deduced from this threshold. From Eq.~\eqref{eq:estimate:gradPara}, the following estimate is recovered
\begin{equation}
	\widetilde{\varepsilon}_\text{ODE} = h^\dagger \cdot \left|\big(b\cdot\nabla^{{h^\dagger},m}\phi^0\big)(x,y)\right|
\end{equation}
The values $h^\dagger$ obtained for the different approximations are reported in Tab.~\ref{tab:accuray:phi0} together with the estimated precision $\widetilde{\varepsilon}_\text{ODE}$ which shall be compared to ${\varepsilon}_\text{ODE}=|u^0(\x)-\phi^0(\x)|$. For the two values of $\x$ the magnitude of ${\varepsilon}_\text{ODE}$ is  accurately captured thanks to the estimator issued from the parallel gradient reconstruction.
\begin{table}	\centering
\begin{tabular}{|c|c|c|}\hline
	&\multicolumn{2}{|c|}{$\x=(0.3,0.9)$}\\ \hline
	 Order & $h^\dagger$ & $\widetilde{\varepsilon}_\text{ODE} $ \\ \hline \hline
	$4^\text{th}$ & $6 \cdot 10^{-5}$ &$4.4 \cdot 10^{-15}$ \\ \hline
	$6^\text{th}$ & $3 \cdot 10^{-4}$ &$1.8 \cdot 10^{-15}$ \\ \hline
	$8^\text{th}$ & $8 \cdot 10^{-4}$ &$9.4 \cdot 10^{-16}$ \\ \hline\hline
	& \multicolumn{2}{|c|}{$|u^0 - \phi^0| \leq 2.8 \cdot 10^{-15} $}\\ \hline
	\end{tabular} %
	\begin{tabular}{|c|c|c|}\hline
		\multicolumn{2}{|c|}{$\x=(0.4,0.94)$}\\ \hline
		$h^\dagger$ & $\widetilde{\varepsilon}_\text{ODE} $ \\ \hline \hline
		 $2 \cdot 10^{-4}$ &$5.2 \cdot 10^{-13}$ \\ \hline
		$8 \cdot 10^{-4}$ &$1.6 \cdot 10^{-12}$ \\ \hline
		 $3 \cdot 10^{-3}$ &$6.0 \cdot 10^{-12}$ \\ \hline\hline
		\multicolumn{2}{|c|}{$|u^0 - \phi^0| \leq 1.2 \cdot 10^{-12} $}\\ \hline
		\end{tabular}
	\caption{Precision of the flow transported solution $\phi^0$ estimated thanks to the error curve of $b\cdot\nabla^{h,m} \phi^0$. The computations are carried out with $(\beta,k)=(8,8)$.}\label{tab:accuray:phi0}
	\vspace*{-0.5cm}
\end{table}

A similar estimator may be proposed for the flow transported Laplacian. High order approximations of the second order derivatives are defined thanks to finite difference formulas \cite{fornberg_generation_1988} using multiple evaluations of the flow transported solution $\phi^0$ and the weights $(\omega_k^{(2)})_k$. Similarly to Eq.~\eqref{eq:Fornberg:formulas:1}, this yields
\begin{equation}\label{eq:Fornberg:formulas}
	\big(\partial_{xx}^{h,m} \phi^0\big)(x,y) = \frac{1}{h^2}\sum_{k={-m/2}}^{k=m/2} \omega_k^{(2)} \, \phi^0(x+k\cdot h,y) \,,
\end{equation}
the approximated Laplacian being defined as
\begin{equation}
	\big(\Delta^{h,m}\phi^0\big)(\x) =  \big(\partial_{xx}^{h,m} \phi^0\big)(\x) + \big(\partial_{yy}^{h,m} \phi^0\big)(\x)\,.
\end{equation}
The precision of the finite difference Laplacian operator may be estimated thanks to
\begin{equation}\label{eq:Difff:Lap:Estimate}
	\varepsilon_{\Delta^{h,m}}(\x)=|\big(\Delta^{h,m}\phi^0\big)(\x) - \Delta u^0(\x) | =   \mathcal{O}\left( h^m \right) + \frac{\widetilde{\varepsilon}_\text{ODE}}{h^2} \,.
\end{equation}
The Fornber's formula implemented in Eq.~\eqref{eq:Fornberg:formulas} meet the following property
\begin{equation}
	\big|\sum_{k=-m/2}^{m/2} \omega_k^{(2)} \big| < \frac{3}{2}\,,
\end{equation}
hence $\widetilde{\varepsilon}_\text{ODE} \sim {\varepsilon}_\text{ODE}$ holds for the second order approximations too.
The precision of the finite difference Laplacian $\Delta^{h,m}\phi^0$ as a function of the step size $h$ is represented on Fig.~\ref{fig:DiffLap:Error} for different orders of approximation $m$.
\begin{figure}[!ht]\centering
	\subfloat[$(x,y)=(0.3,0.9)$]{\includegraphics[width=0.48\textwidth]{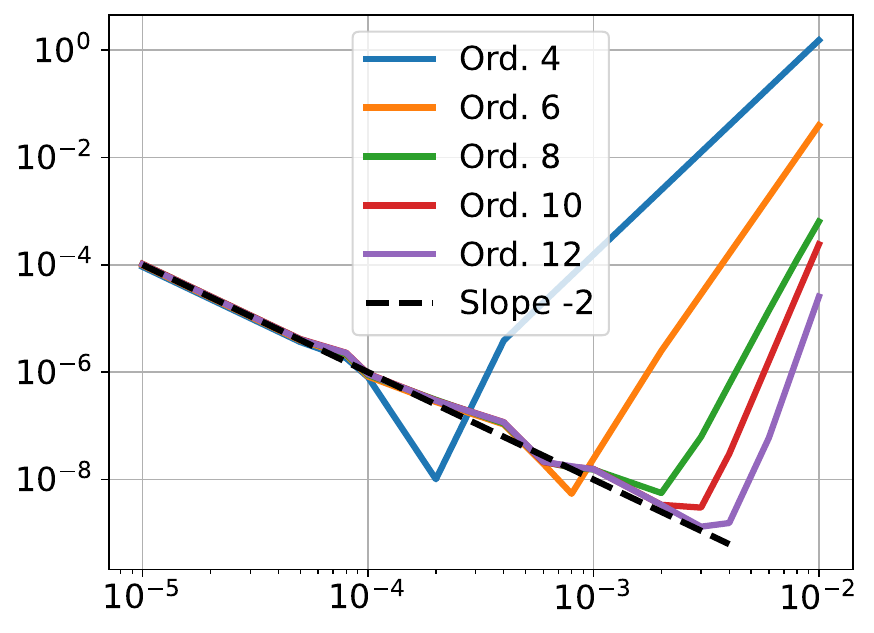}}\hspace*{0.04\textwidth}%
	\subfloat[$(x,y)=(0.4,0.94)$]{\includegraphics[width=0.48\textwidth]{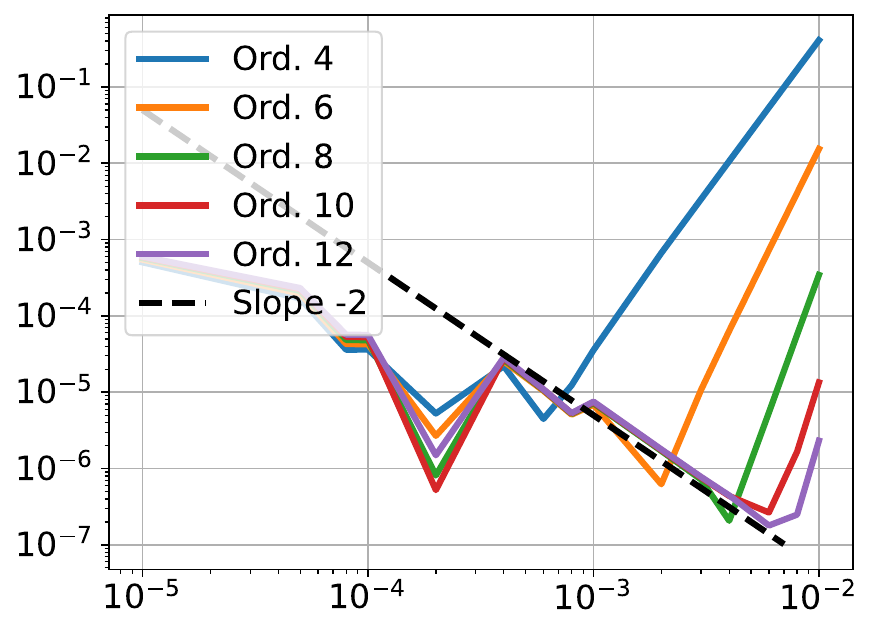}}
	\caption{ Precision of $\Delta^{h,m}\phi^0$ the finite differenced Laplacian of $\phi^0$ as a function of the mesh step. The computations are carried out with $(\beta,k)=(8,8)$.}\label{fig:DiffLap:Error}
	\vspace*{-0.5cm}
\end{figure}
The evolution of the precision is very similar to that of $b\cdot \nabla^{h,m} \phi^0$: for the largest values of the discretization parameter $h$, the error is dominated by the finite difference truncation error $\mathcal{O}(h^m)$. A decrease of the step size drives a decrease of the error with a slope $m$. For $h$-values smaller than the threshold $h^\dagger$, the error of the transported solution $\phi^0$ dominates the component due to the finite difference scheme. For these values of $h<h^\dagger$, a decrease of the step size entails an error increase with a slope 2, as stated by the estimate of Eq.~\eqref{eq:Difff:Lap:Estimate}. The value of the optimal step size and precision obtained from any of the finite difference scheme are gathered in Tab.~\ref{tab:accuray:DFLap:phi0}. An accuracy roughly comparable to that of the flow transported approximation may be obtained thanks to finite difference approximations. 
\begin{table}[!ht]	\centering
\begin{tabular}{|r|c|c|}\hline
	&\multicolumn{2}{|c|}{$(x,y)=(0.3,0.9)$}\\ \hline
	Order & $h^\dagger$ & $ \varepsilon_{\Delta^{h^\dagger,m}}$ \\[0.1em] \hline \hline
	$4^\text{th}$ & $2 \cdot 10^{-4}$ &$1.0 \cdot 10^{-8}$ \\ \hline
	$6^\text{th}$ & $8 \cdot 10^{-4}$ &$5.4 \cdot 10^{-9}$ \\ \hline
	$8^\text{th}$ & $2 \cdot 10^{-3}$ &$5.6 \cdot 10^{-9}$ \\ \hline
	$10^\text{th}$ & $3 \cdot 10^{-3}$ &$3.0 \cdot 10^{-9}$ \\ \hline
	$12^\text{th}$ & $3 \cdot 10^{-3}$ &$1.3 \cdot 10^{-9}$ \\ \hline \hline	
	\multicolumn{3}{|c|}{$|\Delta u^0 - \Delta \phi^0| < 2.2 \cdot 10^{-10} $}\\
	\multicolumn{3}{|c|}{$|\Delta u^0 - \Delta \phi^0|/|\Delta u^0| < 8.3 \cdot 10^{-14} $}\\ \hline
	\end{tabular} %
	\begin{tabular}{|c|c|}\hline
		\multicolumn{2}{|c|}{$(x,y)=(0.4,0.94)$}\\ \hline
		 $h^\dagger$ & $\varepsilon_{\Delta^{h^\dagger,m}}$ \\[0.1em] \hline \hline
		$6 \cdot 10^{-4}$ &$4.5 \cdot 10^{-6}$ \\ \hline
		$2 \cdot 10^{-3}$ &$6.2 \cdot 10^{-7}$ \\ \hline
		$4 \cdot 10^{-3}$ &$2.1 \cdot 10^{-7}$ \\ \hline
		$6 \cdot 10^{-3}$ &$2.7 \cdot 10^{-7}$ \\ \hline
		$6 \cdot 10^{-3}$ &$1.8\cdot 10^{-7}$ \\ \hline\hline
		\multicolumn{2}{|c|}{$|\Delta u^0 - \Delta \phi^0| < 1.4 \cdot 10^{-5} $}\\
	\multicolumn{2}{|c|}{$|\Delta u^0 - \Delta \phi^0|/|\Delta u^0| < 4.7 \cdot 10^{-8} $}\\ \hline
		\end{tabular}
	\caption{Best precision of $\Delta^{h,m}\phi^0$ the finite difference Laplacian of the flow transported solution, the optimal step size is denoted $h^\dagger$. The computations are carried out with $(\beta,k)=(8,8)$.}\label{tab:accuray:DFLap:phi0}
	\vspace*{-0.5cm}
\end{table}

\section{Conclusions}
This article is devoted to the construction of reference solutions used to perform the verification of numerical methods designed to discretize anisotropic elliptic problems arising is plasma physics. The purpose here is to provide a proof of principe of a new method generalizing the so-called solution manufacturing process yielding analytic solutions. So far this derivation is limited to a class of problems for which all the computations can be carried out analytically.

\smallskip\noindent In the specific context of anisotropic problems, the difficulty lies in the definition of a function constant along the field lines and to its Laplacian. Thanks to flow-transported-solution method proposed within the present work, the analytic computations of both quantities are substituted by a numerical procedure relying on the introduction of a local chart which permits the definition of a function constant along the  field lines through a flow integration. The Laplacian of this quantity is then computed by integration of the flow derivatives. The ODE defining these flows are numerically integrated to provide a reference solution obtained for magnetic fields representative of a broader class of applications compared to the solution manufacturing method. 

\smallskip\noindent The accuracy of the aligned gradient free solution is experienced to be close to the computer arithmetic precision and reliable estimates of this accuracy are proposed. The precision of the Laplacian is not straightforwardly related to that of the ODE integrator. The design of a precision estimator is more intricate for this quantity. For regular enough solutions, an estimate constructed by means of finite difference approximations is proposed. It defines either an alternative computation of the Laplacian or a bound estimate of the precision of the flow transported approximation.

\smallskip\noindent The existence of the flows used to approximate the solution and its Laplacian is carefully analyzed, and hypotheses are established within a fairly general framework to ensure the effectiveness of the flow-transported-solution method. In this work, the analysis is conducted for vector fields that are either perpendicular or parallel to the domain boundaries, a constraint that will be lifted in future studies.
 
\smallskip\noindent These initial achievements lay the groundwork for further numerical investigations. The accuracy of the flow-transported approximation, particularly for the Laplacian, could be enhanced by increasing computational precision (e.g., using quadruple precision instead of double precision). This defines a promising direction for future research, especially for approximating solutions with low regularity. For this class of functions, high-order finite difference discretizations cannot provide an accurate approximation of the Laplacian, making the flow-transported-solution method a necessary alternative.

\appendix
\section{Analysis of the flow-transport-solution method}\label{appendix:fred}
\subsection{Objectives and assumptions on the vector field} The purpose of this appendix is to provide a rigorous framework for defining the flow transported solutions and specify a sufficient set of hypotheses. 
The procedure consists in transporting data located on the reference line $\{y=\yp\}$ by flows defined by the $b$-field. Hence, we need to ascertain that the assumptions on the vector field ensure that any field lines reconnect to the reference line to define the foot, denoted $x_0$, of any field line. This issue is investigated within \ref{Sec:Appendix:Flows} together with the definition of the flows relevant for the flow-transported-solution method. The chart is defined in \ref{Sec:Appendix:Chart} and the \One{pullback mapping $F:$} $(x_0,t)\mapsto (x,y)$, where $x_0$ is the foot of the field line issued from $(x,y)$ and $t$ its arc-length, is proved to be a (local) change of variable. \Two{Sections \ref{Sec:Appendix:Lap} and \ref{Sec:Appendix:Parallel:Lap} are devoted to deriving the second-order directional differential operator with spatially varying coefficients, of the form $\nabla_\parallel \cdot (k\nabla_\parallel)$ where $k=k(\mathbf{x})$, hereafter referred to as the heterogeneous parallel Laplacian, as well as} the Laplacian of functions satisfying $\Delta_\parallel u = 0$. These analyses are concluded by the application to the homogeneous elliptic model problem, \Two{referring to a diffusion coefficient $k$ assumed constant and independent of $x$}, provided in \ref{Sec:Appendix:Application}.

A particular attention is paid to the vanishing of the $B$-field entailing the singularity of the normalized field $\vec b$. To conduct this analysis, a vector field incorporating an "X" point, in which the B-field vanishes, is considered as an illustration. This vector field, represented on Fig.~\ref{fig:B:Field} originates from the magnetic fields with CUSPS, particularly studied in \cite{jiang_magnetic_2020,deluzet_numerical_2023,garrigues_acceleration_2024}.%
\begin{figure}[htbp]
    \centering
    \begin{minipage}{0.55\textwidth} 
        \centering
		\vspace*{-0.1cm}
        \includegraphics[width=\linewidth]{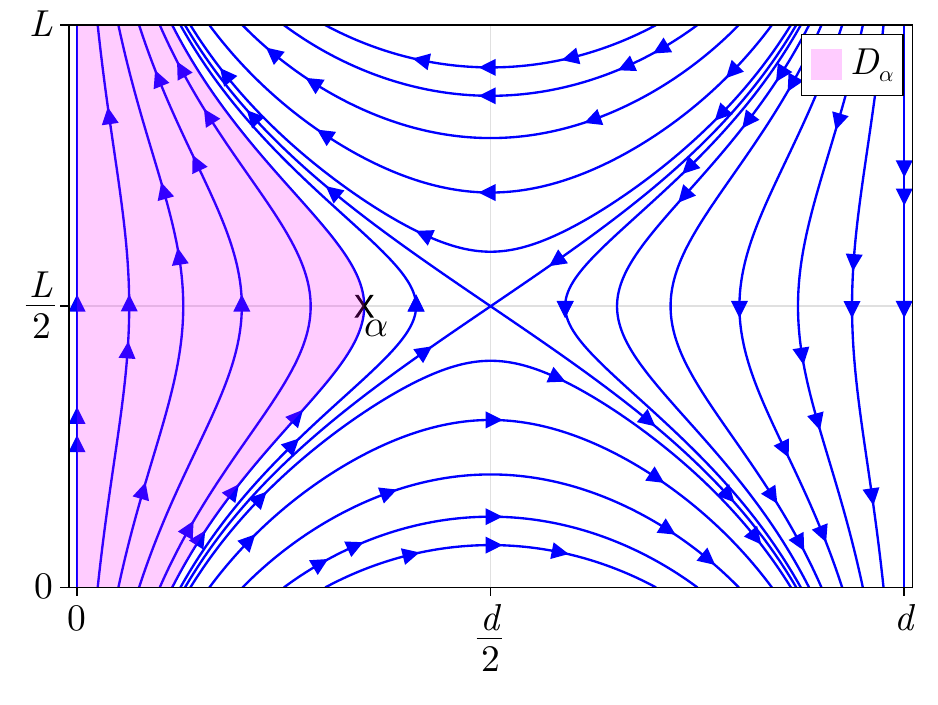}
		\vspace*{-.7cm}
    \end{minipage}%
    \hfill
    \begin{minipage}{0.4\textwidth} 
        \captionof{figure}{Representation of the vector field lines related to the vector field $\vec B=(B_x,B_y)$ in the Cartesian domain. The $B$-field vanishes in the point $S:=(d/2,L/2)$.}
        \label{fig:B:Field}
    \end{minipage}
\end{figure}
The aim now, is to prove, that there exists a domain excluding the X-point, where the flow transported solutions are well-defined. This appendix also provides the expression of the non-homogeneous parallel Laplacian (\Two{with diffusion coefficients depending on $x$}), a relevant quantity for applications.

\begin{remark}
\noindent\Two{This remark provides a comprehensive overview of the modeling choices and the practical implementation of the flow-transported method when extended to complex field topologies.}

\vspace{0.5em}
\noindent\Two{{\bfseries About the choice of the magnetic field at the boundary.} 
To streamline the exposition and the derivation of estimates, we assume magnetic fields that are perpendicular to the horizontal boundaries and parallel to the vertical ones (Hypothesis~\eqref{ppty:B}). Nevertheless, the analytical framework developed in this appendix is robust enough to be extended to fields with more general geometric configurations.}

\vspace{0.5em}
\noindent\Two{{\bfseries About the choice to introduce an $X$-singularity.} 
A magnetic field featuring an $X$-point, where the field vanishes, is considered to demonstrate the method's stability in critical regimes. The primary challenge lies in bounding the time required to reach the reference line where the boundary value is provided; while this duration increases near the singularity, the estimates derived in Lemma~\ref{Claim:2} remain strictly controlled by means of the parameter $\alpha$.}

\vspace{0.5em}
\noindent\Two{{\bfseries About the boundary condition for $u$.} 
For the sake of numerical clarity, we have focused on solutions satisfying a Neumann condition on the horizontal boundary (as in \eqref{eq:def:phi0}). However, the mathematical derivations in Lemmas~\ref{claim:lk} and \ref{claim:der} are conducted in full generality, ensuring that the results remain valid for broader sets of boundary specifications.}

\vspace{0.5em}
\noindent\Two{{\bfseries About the implementation of generalized geometries.} 
In the current framework, boundary data is propagated from segments locally perpendicular to the field vector $\mathbf{b}$. For $X$-point geometries (see Fig.~\ref{fig:B:Field}), this is achieved by prescribing inflow data on the horizontal segment at $y=L/2$ (or the vertical segment at $x=d/2$) and propagating it forward and backward along the flow characteristics. Such a procedure confirms the framework's reliability for segmented boundaries. Future work will focus on relaxing the orthogonality requirement to accommodate more arbitrary geometries, including the treatment of magnetic islands.}
\end{remark}

The domain is $\Omega=(0,\frac{d}{2})\times (0,L)$ and $\yp = L/2$ so that the reference line is defined by $(0,\frac{d}{2})\times \{\yp\}$.
Let $\vec B\in C^2(\rd,\rd)$ be a vector field such that there exists $\eps,d,L>0$ satisfying 
$$\left\{\vec B(x,y)= \vec 0 \hbox{ and }(x,y)\in (-\eps, d+\eps)\times \rr\right\} \Longleftrightarrow (x,y)=S:=\left(\frac{d}{2},\frac{L}{2} \right).$$

\begin{hypothesis}\label{Hyp:A1}
	Defining $B_x,B_y\in C^2(\rd,\rr)$ such that $\vec B:=(B_x,B_y)$, the following properties will be assumed
\begin{align}
&B_x>0\hbox{ in }(0,d)\times (-\infty,L/2)\hbox{ and } B_x<0\hbox{ in }(0,d)\times (L/2,+\infty).\label{H1}\\
&B_x<0\hbox{ in }(-\eps,0)\times (-\infty,L/2)\hbox{ and } B_x>0\hbox{ in }(-\eps,0)\times (L/2,+\infty).\label{H1bis}\\
&B_y>0\hbox{ in }(-\eps,d/2)\times \rr\hbox{ and }B_y<0\hbox{ in }(d/2,d+\eps)\times \rr.\label{H2}\\
&B_x(0,y)=B_x(d,y)=0\hbox{ for all }y\in \rr\label{H3}\\
&B_y(x,L)>0\hbox{ and }B_y(x,0)>0\hbox{ for }0\leq x<d/2.\label{H4}
\end{align}
\end{hypothesis}
As a consequence of these hypotheses, the following identities hold true
\begin{equation}\label{ppty:B}
\left\{\begin{array}{cc}
B(x, L/2)=\Vert B\Vert (0,1)&\hbox{ for }x\in (0,d/2),\\
B(x, L/2)=-\Vert B\Vert (0,1)&\hbox{ for }x\in (d/2,d),\\
B(d/2,y)=\Vert B\Vert (1,0)&\hbox{ for }y\in (0,L/2),\\
B(d/2,y)=-\Vert B\Vert (1,0)&\hbox{ for }y\in (L/2,L).
\end{array}\right.
\end{equation}
Finally, the normalized field is defined as 
$$\vec{b}(x,y)=(b_x,b_y):=\frac{\vec B(x,y)}{\Vert B(x,y)\Vert}\hbox{ for all }(x,y)\in (-\eps, d+\eps)\times \rr\setminus \{S\}.$$

For any $\alpha\in [0,d/2)$, we define
$$\eps_\alpha:=\min_{(x,y)\in [-\eps/2,\alpha]\times [0,L]}b_y(x,y).$$
It follows from Hyp.~\ref{Hyp:A1} that $\eps_\alpha>0$ for all $\alpha\in [0,d/2)$.

\subsection{Definition of the flows}\label{Sec:Appendix:Flows}
\subsubsection{Flow of the $\vec B$ field: $\Phi_t$}
\One{The different flows introduced in the precedent sections are now defined within a rigorous mathematical framework.}
\begin{definition}
	Given any $p\in\left((-\eps, d+\eps)\times \rr\right)\setminus\{S\}$, the flow $t\mapsto \Phi_t(p)$ associated to the field $\vec{b}$ is defined as
\begin{equation}\label{def:flow}
\left\{\begin{array}{l}
\displaystyle \frac{d}{dt}\Phi_t(p)=\vec{b}\big(\Phi_t(p)\big)\,,\\[0.3em]
\phi_0(p)=p\,.
\end{array}\right.
\end{equation}
\end{definition}
Owing to Hyp.~\ref{Hyp:A1} and since $\Vert\vec{b}\Vert=1$ where $B$ does not vanish, $\Phi_t(p)$ is defined for all $t$. The existence and uniqueness follows from the Cauchy Theorem since $\vec{b}\in C^2(\left((-\eps, d+\eps)\times \rr\right)\setminus\{S\})$.

The following assertions are consequence of ODE theory:
\begin{lemma}\label{Claim:def:T} For all $x_0\in (-\eps/2,d/2)$, let $t\mapsto \Phi_t(x_0, L/2)$ be the flow starting at $(x_0,L/2)$. Let $(x(t),y(t)):=\Phi_t(x_0,L/2)$ so that $x,y\in C^1(\rr,\rr)$. Then there exist $T_0(x_0), T_L(x_0)\in \rr$ such that
\begin{equation}\label{def:T}
\left\{\begin{array}{c}
T_0(x_0)<0,\, y(T_0(x_0))=0\hbox{ and }T_L(x_0)>0,\, y(T_L(x_0))=L\\
y'>0\hbox{ on }[T_0(x_0),T_L(x_0)]
\end{array}\right.
\end{equation}
\end{lemma}
\begin{proof}[Proof of Lemma~\ref{Claim:def:T}] As a preliminary remark, note that if $x$ vanishes somewhere, it follows from Hyp.~\eqref{H3} and uniqueness in ODE theory that $x(t)=0$ for all $t\in \rr$. Assume that $x_0>0$, then $x(t)> 0$ for all $t\in \rr$. Owing to properties \eqref{ppty:B} and the ODE \eqref{def:flow}, there exists $\lambda>0$ such that $y'(t)>0$ for ${t \in}(-\lambda,\lambda)$. Assume that $y$ never reaches $L$ on $[0,+\infty)$. Then $y(t)<L$ on $[0,+\infty)$. With Hyp.~\ref{Hyp:A1}, we then get that $\Phi_t(x_0,L/2)=(x(t),y(t))\in [0,x_0]\times [L/2,L)$ for all $t\in [0,+\infty)$, so that $y'(t)\geq \eps_{x_0}$ for all $t\in [0,+\infty)$, and then $\lim_{t\to +\infty}y(t)=+\infty$, which is a contradiction. This proves the existence of a point $t_L>0$ such that $y(t_L)=L$, $T_L(x_0)>0$ being defined as the smallest such point when $x_0>0$. The case $x_0\leq 0$ and the proof of the existence of $T_0(x_0)$ are similar. \end{proof}

\begin{lemma}\label{Claim:2} Let \Two{$T_L: (-\eps/2,d/2)\to \rr$} be the functions defined by Lemma~\ref{Claim:def:T}. Then $T_L\in C^1((-\eps/2, d/2),\rr)$ and for all $\alpha\in [0, d/2)$, the following bounds hold true
$$\frac{L}{2}\leq T_L(x_0)\leq \frac{L}{2\eps_\alpha}\hbox{ for all }x_0\in (-\eps/2,\alpha).$$
\end{lemma}
\begin{proof}[Proof of Lemma~\ref{Claim:2}] For $\alpha\in [0,d/2)$ and $0\leq x_0\leq \alpha$, using the notations of proof of Lemma~\ref{Claim:def:T}, one checks, $(x(t),y(t))\in [0,x_0]\times [L/2,L]$ for all $t\in [0, T_L(x_0)]$. In particular, $y'(t)=b_y(x(t),y(t))\in [\eps_\alpha, 1]$ for all $t\in [0, T_L(x_0)]$. Rolle's theorem yields $\tau\in (0,T_L(x_0))$ such that $L=y(T_L(x_0))=y(0)+y'(\tau)T_L(x_0)=L/2+y'(\tau)T_L(x_0)$. Since $\eps_\alpha\leq y'(\tau)\leq 1$, we get the required inequality on $T_L(x_0)$. The proof is similar for $-\eps/2\leq x_0\leq 0$. \par

\smallskip\noindent We now prove that $x_0\mapsto T_L(x_0)$ is continuous on $[-\eps/2,\alpha]$. Let $x_0\in [0,\alpha]$ and $(x_i)\in [0,\alpha]$ be a sequence such that $\lim_{i\to +\infty}x_i=x_0$. It follows from the preceding inequality that $(T_L(x_i))_i$ is bounded in $\rr$, so, up to extraction, it converges to some $T\in\rr$. For convenience, we let $\Pi_2$ be the projection on the second variable in $\rr^2$. We have that $\Pi_2(\Phi_t(x_i, L/2))<L$ for $0\leq t<T_L(x_i)$ and $\Pi_2(\Phi_{T_L(x_i)}(x_i, L/2))=L$. Passing to the limit $i\to +\infty$ yields $\Pi_2(\Phi_t(x_0, L/2))\leq L$ for $0\leq t<T$ and $\Pi_2(\Phi_{T}(x_0, L/2))=L$. Since $\frac{d}{dt}(\Pi_2(\Phi_t(x_0, L/2)))_{T}\geq \eps_\alpha>0$, we then get that $T=T_L(x_0)$. This limit is independent of the extraction, so $\lim_{i\to +\infty}T_L(x_i)=T_L(x_0)$, and $T_L$ is continuous. The proof is the same for $x_0\leq 0$. 

\smallskip\noindent We finally prove that $x_0\mapsto T_L(x_0)$ is $C^1$ on $(-\eps/2,\alpha)$. For all $x\in (-\eps/2,\alpha)$ and $t\in\rr$ let $G(t,x):=\Pi_2(\Phi_{t}(x,L/2))$, we then get that $G(T_L(x),x)=L$ for all $x\in (-\eps,\alpha)$. Since $\partial_t G(t,x_0)_{|t=T_L(x_0)}\geq \eps_\alpha>0$, it follows from the implicit function theorem that $T_L$ is $C^1$.
\end{proof}

\subsubsection{A first derivative of the flow $\Phi_t$: $\partial_{x_0}\Phi_t$}
\medskip\noindent We now differentiate the flow. Setting $p=(x_0,L/2)$ in Eq.~\eqref{def:flow} and differentiating $\Phi_t$ with respect to $x_0$, we get that for any \Two{$x_0\in (-\eps/2,\alpha)$, for all $\alpha \in [0,d/2)$} and $t\in\rr$,
\begin{equation}\label{def:flow:2}
\left\{\begin{array}{l}
\displaystyle \frac{d}{dt}\left(\partial_{x_0}\Phi_t(x_0,L/2)\right)=D\vec{b}\big(\Phi_t(x_0,L/2)\big)\big[\partial_{x_0}\Phi_t(x_0,L/2)\big]\,,\\[0.5em]
\partial_{x_0}\Phi_t(x_0,L/2)_{|t=0}=(1,0)\,.
\end{array}\right.
\end{equation}
Since $D\vec b$ is $C^1$, this generates a flow $S_t^{(1)}$ depending on the parameter $x_0$.
\begin{definition}\label{Def:A3}  
	For all  $x_0 \in (-\eps/2,\alpha)$, the flow 
%
 $S^{(1)}_t(x_0):=\partial_{x_0}\Phi_t(x_0,L/2)$ satisfies
\begin{equation}\left\{\begin{array}{l}
\displaystyle \frac{d}{dt}S^{(1)}_t(x_0)=D\vec{b}\big(\Phi_t(x_0,L/2)\big)\big[S^{(1)}_t(x_0)\Big]\,,\\[0.5em]
S^{(1)}_0(x_0)=(1,0)\,.
\end{array}\right.
\end{equation}
\end{definition}

\subsubsection{Second order derivative flow $\partial^2_{x_0x_0}\Phi_t$}
Differentiating Eq.~\eqref{def:flow:2} with respect to $x_0$ yields
\begin{equation}\label{def:flow:3}
	\left\{\begin{aligned}
	&\begin{multlined}[0.87\linewidth]
		 \frac{d}{dt}\left(\partial^2_{x_0x_0}\Phi_t(x_0,L/2)\right)=D\vec{b}\big(\Phi_t(x_0,L/2)\big)\big[\partial^2_{x_0x_0}\Phi_t(x_0,L/2)\big] \\
		+ D^2\vec{b}\big(\Phi_t(x_0,L/2)\big)\big[\partial_{x_0}\Phi_t(x_0,L/2),\partial_{x_0}\Phi_t(x_0,L/2)\big]\,,
	\end{multlined}\\
	&\partial^2_{x_0x_0}\Phi_t(x_0,L/2)=(0,0)\,.
	\end{aligned}\right.
	\end{equation}
Since $D^2\vec b\in C^0$, this generates the following flow.
\begin{definition}\label{Def:A4} For all $x_0 \in (-\eps/2,\alpha)$, the flow
%
 $S^{(2)}_t(x_0):=\partial^2_{x_0x_0}\Phi_t(x_0,L/2)$ satisfies
\begin{equation}\label{def:flow:3:S}
\left\{\begin{array}{l}
\displaystyle \frac{d}{dt}S^{(2)}_t(x_0)=D\vec{b}\big(\Phi_t(x_0,L/2)\big)\big[S^{(2)}_t(x_0)\big] \\
\hspace*{10em}+ D^2\vec{b}\big(\Phi_t(x_0,L/2)\big)\big[S^{(1)}_t(x_0),S^{(1)}_t(x_0)\big]\,,
\\[0.5em]
S^{(2)}_0(x_0)=(0,0)\,.
\end{array}\right.
\end{equation}
\end{definition}

\subsection{Definition of the chart and the metric}\label{Sec:Appendix:Chart}
For all $x_0\in (-\eps/2,d/2)$ and $t\in\rr$, we define
\begin{equation}\label{eq:Def:Det}
\mathcal{D}(x_0,t):=\det\left(\partial_{x_0}\Phi_t(x_0, L/2),\partial_t\Phi_t(x_0, L/2)\right).
\end{equation}%
\begin{lemma}\label{5} We have that
$$\mathcal{D}(x_0,t)=\exp \left(\int_{0}^t\hbox{Tr}\left(D\vec{b}(\Phi_s(x_0,L/2))\right)\, ds\right)\hbox{ for all }(x_0,t)\in (-\eps/2,d/2)\times \rr.$$
\end{lemma}%
\begin{proof}[Proof of Lemma~\ref{5}] We differentiate Eq.~\eqref{eq:Def:Det} along the $t-$variable to obtain
\begin{eqnarray*}
\partial_t \mathcal{D}(x_0,t)
&=& \det\left(D\vec{b}(\Phi_t(x_0, L/2))[\partial_{x_0}\Phi_t(x_0, L/2)],\partial_t\Phi_t(x_0, L/2)\right)\\
&&+\det\left(\partial_{x_0}\Phi_t(x_0, L/2),D\vec{b}(\Phi_t(x_0, L/2))[\partial_{t}\Phi_t(x_0, L/2)]\right)\,.
\end{eqnarray*}
Since $\det(AX,Y)+\det(X,AY)=\hbox{Tr}(A) \det(X,Y)$ for all $X,Y\in\rr^2$ and $A\in M_2(\rr)$, we get that
\begin{eqnarray*}
\partial_t \mathcal{D}(x_0,t)&=&  \det\left(\partial_{t,x_0}\Phi_t(x_0, L/2),\partial_t\Phi_t(x_0, L/2)\right)\\
	&& +\det\left(\partial_{x_0}\Phi_t(x_0, L/2),\partial_{t}\vec{b}(\Phi_t(x_0, L/2))\right)\\
&=& \hbox{Tr}(D\vec{b}(\Phi_t(x_0,L/2))) \det\left( \partial_{x_0}\Phi_t(x_0, L/2),\partial_t\Phi_t(x_0, L/2)\right)\\
&=&\hbox{Tr}(D\vec{b}(\Phi_t(x_0,L/2))) \mathcal{D}(x_0,t).
\end{eqnarray*}
The proof is concluded by integrating with respect to $t$ and noting that $\mathcal{D}(x_0,0)=1$.
\end{proof}
\begin{definition}\label{Def:yalpha}
For $\alpha\in (0,d/2)$, we define 
\begin{equation}\label{Eq:Def:yalpha}
	(x_\alpha(t),y_\alpha(t)):=\Phi_t(\alpha, L/2) \mbox{ for all } t\in\rr\,,
\end{equation}
and denote $\One{y_\alpha^{-1}} \, : \, y\mapsto t$, the reciprocal application to $y_\alpha$.
\end{definition}
\begin{lemma}\label{lemma:X} There exists an open interval $I$ such that $y_\alpha: \rr\to I\supset [0,L]$ is a $C^1-$diffeomorphism and 
$$y_\alpha(\rr)=I\hbox{ and }0<y_\alpha'(t)\leq 1\hbox{ for all }t\in \rr.$$
\end{lemma}
\begin{proof}[Proof of Lemma~\ref{lemma:X}] This is a consequence of the definition \eqref{def:T}, the properties \eqref{H1} and \eqref{H2} and the inverse function theorem. 
\end{proof}
\Two{We now defined $D_\alpha$ as illustrated on Fig.~\ref{fig:B:Field} and state some important properties related to this domain.}
\begin{lemma}\label{claim:Dalpha} Let $D_\alpha:=\{(x,y)\in (0,d/2)\times I/\, 0 <x<x_\alpha\circ \One{y_\alpha^{-1}}(y) \}$, then $D_\alpha$ is preserved by the flow $\Phi_t$:  $\Phi_t(D_\alpha)\subset D_\alpha$ for all $t\in\rr$.
\end{lemma}
\begin{proof}[Proof of Lemma~\ref{claim:Dalpha}] We fix $(x,y)\in D_\alpha$ and we define $\gamma(t):=\Phi_t(x,y)$ for all $t\in\rr$. Assume that there exists $t_0\in\rr$ such that $\gamma(t_0)_x=x_\alpha\circ \One{y_\alpha^{-1}}(\gamma(t_0)_y)$. Defining $\tau:=\One{y_\alpha^{-1}}(\gamma(t_0)_y)$, we get that $\gamma(t_0)=\Phi_{\tau}(\alpha, L/2)$. The uniqueness for ODEs yield $\gamma(t)=\Phi_{t+\tau-t_0}(\alpha, L/2)$ for all $t\in\rr$, so that for $t=0$, we get that $x=x_\alpha\circ \One{y_{\alpha}^{-1}}(y)$, contradicting $(x,y)\in D_\alpha$. Similarly, we get that $\gamma(t)_x\neq 0$ for all $t\in \rr$. By continuity, we then get that $\gamma(t)\in D_\alpha$ for all $t\in\rr$, this concludes the proof. 
\end{proof}
\begin{definition} \Two{Let $D_\alpha$ be given by Lemma~\ref{claim:Dalpha}, the pullback mapping} $F$ is defined as 
\begin{equation}
	\begin{array}{cccc}
F: & (0,\alpha)\times \rr &\to & D_\alpha\\
& (x_0,t) &\mapsto & \Phi_{t}(x_0, L/2)
\end{array}
\end{equation}
\end{definition}

\begin{lemma}\label{lemma:Pullback} $F$ is a $C^1-$diffeomorphism. Moreover, 
$$\det\left(\partial_{x_0} F(x_0, t),\partial_t F(x_0, t)\right)=\mathcal{D}(x_0, t)\hbox{ for all }(x_0,t)\in (0,\alpha)\times \rr.$$
\end{lemma}
\begin{proof}[Proof of Lemma~\ref{lemma:Pullback}] We first prove the bijectivity. We let $p=(x,y)\in D_\alpha$ be such that $y\geq L/2$. We consider the backward flow $\frac{d}{dt}\check{\Phi}_t(p)=-\vec{b}(\check{\Phi}_t(p))$ which turns to be defined for all $t\in\rr$. Arguing as in previous claims, there exists $t_0\geq 0$ and $x_0\in (0, d/2)$  such that $\check{\Phi}_{t_0}(p)=(x_0, L/2)$. Since $t\mapsto \check{\Phi}_{t_0-t}(p)$ and $t\mapsto \Phi_t(x_0, L/2)$ coincide at $t=0$ and satisfy the ODE \eqref{def:flow}, uniqueness yields  $\check{\Phi}_{t_0-t}(p)= \Phi_t(x_0, L/2)$ for all $t\in\rr$, so that $p= {\Phi}_{t_0}(x_0, L/2)=F(t_0, x_0)$. Note that $D_\alpha$ is also preserved by the flow $\check{\Phi}_t$, and then $(x_0,L/2)\in D_\alpha$, and then $0<x_0<\alpha$. The proof is similar for $y\leq L/2$. This proves bijectivity. The equality for the determinant is a reformulation of Lemma~\ref{5}. The diffeomorphism is a consequence of the inverse function theorem.
\end{proof}
We pull-back the Euclidean metric, denoted Eucl, via $F$ to get the new Riemannian metric $g:=F^\star\hbox{Eucl}$.
\begin{definition}
	The metric of the pullback F is the $2\times 2$ matrix denoted $g$, $g^{-1}$ being its inverse, defined by
	\begin{subequations}	\label{eqs:g:coeffg}	
	\begin{equation}
		g=\left(\begin{matrix}
			g_{tt} & g_{t x_0}\\ g_{tx_0} & g_{x_0x_0}
			\end{matrix}\right) \,, \qquad 
		g^{-1}=(g^{ij})=\frac{1}{\det g} \left(\begin{matrix}
	g_{x_0x_0} & -g_{t x_0}\\ -g_{tx_0} & g_{tt}
	\end{matrix}\right)\,,\label{exp:invg}
	\end{equation}
with, denoting $\langle\cdot,\cdot\rangle$ the usual scalar product in $\rr^2$
	\begin{equation}	\label{eq:g:coeff}	
	\left\{\,\,\begin{aligned}
		&g_{tt}=\langle \partial_t F,\partial_t F\rangle=1\,,\\
		&g_{tx_0}=g_{x_0t}=\langle \partial_t F,\partial_{x_0} F\rangle=\langle \vec{b}\circ \Phi_t(x_0, L/2), S_t^{(1)}(x_0)\rangle\,,\\
		&g_{x_0x_0}=\langle \partial_{x_0}F,\partial_{x_0} F\rangle=\Vert S_t^{(1)}(x_0)\Vert^2\,.
		 \end{aligned}\right.
		\end{equation}
	\end{subequations}
\end{definition}

\begin{lemma}\label{exp:detg} We have that $\det g=\mathcal{D}(x_0, t )^2$ for all $(x_0,t)\in (0,\alpha)\times\rr$.
\end{lemma}
\begin{proof}[Proof of Lemma~\ref{exp:detg}] This equality relies on a more general identity. Let $(E, \langle\cdot,\cdot\rangle)$ be a Euclidean space of dimension $n$, $\beta$ be an orthonormal basis and $X_1,...,X_n$ be vectors of $E$. We set $A=\hbox{Mat}_\beta(X_1,...,X_n)$. Writing the coordinates of the vectors in $\beta$, we get that $(A^TA)_{ij}=\langle X_i,X_j\rangle$ for all $i,j\in \{1,...,n\}$. Taking the determinant yields $\det_\beta(X_1,...,X_n)^2=\det(\langle X_i,X_j\rangle)$. Taking $n=2$ together with both the vectors $\partial_t F$ and $\partial_{x_0}F$, we obtain the result.
\end{proof}
\subsection{Expression of $\Delta_{\parallel}$ in the chart}\label{Sec:Appendix:Lap}
Let $k\in C^1(\rr^2)$ be a function such that $k(x,y)\neq 0$ for all $(x,y)\in\rr^2$. We define the parallel Laplacian as 
$$\Delta_{\parallel} u=\hbox{div}\left(k \vec{b}\otimes \vec{b} \cdot\nabla u\right)=\sum_{i,j}\partial_i (k b^i b^j\partial_ju).$$
\begin{lemma}\label{11} Given $0<\alpha<d/2$ and $(x_0,t)\in (0,\alpha)\times \rr$, we have that
\begin{equation}\label{exp:delta:p}
(\Delta_{\parallel}u)\circ F(x_0,t)=k\circ F\left( \partial_{tt}(u\circ F)+H(\Phi_t(x_0, L/2))\partial_t (u\circ F)\right),
\end{equation}
where
$$H(x,y)=\hbox{Tr}\left(D\vec{b}(x,y)\right)+\frac{\langle\nabla k,\vec{b}\rangle(x,y)}{k(x,y)}.$$
\end{lemma}
\begin{proof}[Proof of Lemma~\ref{11}] \One{This result can be obtained through direct computation, which we omit for brevity. Instead, we present a proof based on Riemannian geometry, where $\hbox{div}_gX=\sum_{i,j}g^{ij}(\nabla X)_{ij}$ via the Levi-Civita connection. The key observation} is that $\vec{b}(F(x_0,t))=\partial_t F(x_0,t)=dF(x_0,t)(\vec{e}_t)$ where $\vec{e}_t$ is the second vector of the canonical basis of $\rr^2$. Setting $\tilde{k}:=k\circ F$ and $\tilde{u}=u\circ F$, we have that
\begin{eqnarray*}
(\Delta_{\parallel}u)\circ F&=& \left(\hbox{div}\left(k(x,y)\vec{b}\otimes \vec{b} \cdot\nabla u\right)\right)\circ F\\
&=& \left(\hbox{div}_{\hbox{Eucl}}\left(k(x,y)Du(x,y)[\vec{b}]\vec{b}\right)\right)\circ F\\
&=& \hbox{div}_{F^\star\hbox{Eucl}}\left((k\circ F) D(u\circ F)[\vec{e}_t]\vec{e}_t\right)=\\
&=&\hbox{div}_{g}\left(\tilde{k} \partial_t\tilde{u}\vec{e}_t\right)=\sum_{i\in \{x_0,t\}}\nabla_i \left(\tilde{k} \partial_t\tilde{u}\vec{e}_t\right)^i\\
&=&\sum_{i\in \{x_0,t\}}\Big\{\partial_i \left(\tilde{k} \partial_t\tilde{u}\vec{e}_t\right)^i+\sum_{j\in \{x_0,t\}}\Gamma_{ij}^j\left(\tilde{k} \partial_t\tilde{u}\vec{e}_t\right)^i\Big\}\\
&=& \partial_t \left(\tilde{k} \partial_t\tilde{u}\right)+\sum_{i\in \{x_0,t\}}\sum_{j\in \{x_0,t\}}\Gamma_{ij}^i\left(\tilde{k} \partial_t\tilde{u}\vec{e}_t\right)^j\\
&=&\partial_t \left(\tilde{k} \partial_t\tilde{u}\right)+\sum_{i\in \{x_0,t\}} \Gamma_{it}^i \tilde{k} \partial_t\tilde{u}\\
\end{eqnarray*}
where the $\Gamma_{ij}^k$'s are the Christoffel symbols, that is
\begin{equation}\label{def:christoffel}
\Gamma_{ij}^k:=\frac{1}{2}\sum_{m}g^{km}\left(\partial_i g_{mj}+\partial_j g_{im}-\partial_m g_{ij}\right).
\end{equation}
Using the symmetries of $g$, we have that
\begin{equation*}
\sum_i\Gamma_{it}^i:=\frac{1}{2}\sum_{i,m}g^{im} \partial_t g_{im}.
\end{equation*}
Independently, the differential of the determinant yields
\begin{eqnarray*}
\partial_t \det g&=& \det(g)(\hbox{Tr}(g^{-1}\partial_t g))=\det(g)(\sum_{i,j}g^{ij}\partial_t g_{ij})
\end{eqnarray*}
which, thanks to Lemma~\ref{5} yields
$$\sum_i\Gamma_{it}^i=\frac{1}{2}\frac{\partial_t \det g}{\det g}=\frac{\partial_t \mathcal{D}}{\mathcal{D}}=\hbox{Tr}(D\vec{b}(\Phi_t(x_0,L/2))).$$
Finally, noting that $\langle \nabla k, \vec{b}\rangle\circ F=\partial_t \tilde{k}$, the proof is concluded by plugging all these identities together.
\end{proof}

Integrating Eq.~\eqref{exp:delta:p}, we get that
\begin{lemma}\label{exp:u} Fix $\alpha\in (0,d/2)$ and let $u\in C^2(D_\alpha)$ be a function. Then $\Delta_{\parallel}u=0$ in $D_\alpha$ if and only if there exists $P,Q\in C^2((0,\alpha))$ such that
	\begin{equation}\label{eq:Lemma:A14}
		\left\{\,\begin{aligned}
			u\circ F(x_0,t)&=P(x_0)+Q(x_0)\int_0^t\Psi_{(x_0,L/2)}(s)\, ds \,, \quad \forall (x_0,t)\in (0,\alpha)\times\rr\,, \\
			\Psi_{(x_0,L/2)}(s) &= \exp \left(-\int_0^s H\big(\Phi_\tau(x_0, L/2)\big)\, d\tau\right) \,.
		\end{aligned}		\right.
	\end{equation}
\end{lemma}

\begin{lemma}\label{fact:pres:1} Fix $\alpha\in (0,d/2)$ and let $u\in C^2(D_\alpha)$ be such that $\Delta_{\parallel}u=0$ in $D_\alpha$ and let $P,Q$ defined by Lemma~\ref{exp:u}. Then 
$$P(x_0)=u(x_0, L/2)\hbox{ and }Q(x_0)=\langle\nabla u, \vec{b}\rangle(x_0, L/2).$$
\end{lemma}

\subsection{Expression of $\Delta u$ where $\Delta_\parallel u=0$}\label{Sec:Appendix:Parallel:Lap}
For $K\in C^2(\rr^2)$, we compute $L_K(u)(x,y):=\hbox{div} (K(x,y)\nabla u(x,y))$ where $u\in C^2(D_\alpha)$ is a solution to $\Delta_{\parallel}u=0$. In order to make the exposition more convenient to tackle, the computation is organized into the three following Lemmas that are devoted to be merged to get an expression of $L_Ku$. 
\begin{lemma}\label{claim:lk}  Let $K\in C^1(\rr^2)$ be a function and define the operator $L_K:=\hbox{div} (K(x,y)\nabla \cdot)$. Fix $\alpha\in (0,d/2)$ and let $u\in C^2(D_\alpha)$ be a function. Then, in the chart $F$ we have that
\begin{equation}
	\begin{multlined}[0.85\textwidth]
(L_K u)\circ F(x_0,t)=\sum_{i,j}g^{ij}(K\circ F)\partial_{ij}(u\circ F)\\
+\sum_k\left(-\sum_{p,q}(K\circ F)g^{pq}\Gamma_{pq}^k+\sum_pg^{pk}\partial_p (K\circ F)\right)\partial_k(u\circ F)\,.
	\end{multlined}
\end{equation}
\end{lemma}
\begin{proof}[Proof of Lemma~\ref{claim:lk}] \One{Again, this identity can be verified through direct calculation; however, we provide a proof based on Riemannian geometry. Following the same approach as for $\Delta_\parallel$, we set $\tilde{K}:=K\circ F$ and $\tilde{u}:=u\circ F$.} We have that
\begin{align*}
(L_K u)\circ F&= (\hbox{div}_{\hbox{Eucl}} (K\nabla u))\circ F=\hbox{div}_{F^\star \hbox{Eucl}} ((K\circ F)\nabla (u\circ F))=\hbox{div}_{g} (\tilde{K}\nabla\tilde{u})\\
&= \sum_{ij} g^{ij}\nabla_i(\tilde{K}\partial_j\tilde{u})=\sum_{ij} g^{ij}\left(\partial_i(\tilde{K}\partial_j\tilde{u})-\sum_k\Gamma_{ij}^k\tilde{K}\partial_k\tilde{u}\right)
\end{align*}
which yields the expected result.
\end{proof}

\begin{lemma}\label{claim:der} Fix $\alpha\in (0,d/2)$ and let $u\in C^2(D_\alpha)$ be such that $\Delta_{\parallel}u=0$ in $D_\alpha$ and let $P,Q\in C^2((0,\alpha))$ be as in Lemma~\ref{exp:u}. Assume that $B\in C^3(\rr^2,\rr^2)$ and $k\in C^3(\rr^2)$.
Denoting
\begin{equation}
	\Theta_{(x_0,L/2)}(t) = -\int_0^t DH\big(\Phi_s(x_0, L/2)\big)\left[S_s^{(1)}(x_0)\right]\, ds \,,
\end{equation}
we have
\begin{align*}
&\partial_t(u\circ F)=Q(x_0)\Psi_{(x_0,L/2)}(t) \,, \\[0.8em]
&\partial_{tt}(u\circ F)=-H\circ\Phi_t(x_0,\tau)Q(x_0)\Psi_{(x_0,L/2)}(t)\,, \\[0.8em]
&\begin{multlined}[\textwidth]
\partial_{x_0}(u\circ F)=\partial_{x_0}P(x_0)+\partial_{x_0}Q(x_0)\int_0^t\Psi_{(x_0,L/2)}(s) ds\\[-0.5em]
+Q(x_0)\int_0^t \Theta_{(x_0,L/2)}(s) \Psi_{(x_0,L/2)}(s)\, ds\,,
\end{multlined}\\[0.8em]
&\partial_{t x_0}(u\circ F)= \partial_{x_0}Q(x_0)\Psi_{(x_0,L/2)}(t)+Q(x_0)\Theta_{(x_0,L/2)}(t)\Psi_{(x_0,L/2)}(t)\,,
\end{align*}
\begin{align*}
&\begin{multlined}[\textwidth]
	\partial_{x_0 x_0}(u\circ F)=\partial_{x_0x_0}P(x_0)+\partial_{x_0x_0}Q(x_0)\int_0^t\Psi_{(x_0,L/2)}(s) ds\\
	+Q(x_0)\int_0^t\left(-\int_0^sD^2H\big( \Phi_\tau(x_0, L/2)\big)\left[S_\tau^{(1)}(x_0),S_\tau^{(1)}(x_0)\right]\, d\tau\right)\Psi_{(x_0,L/2)}(s)\, ds\\
	+Q(x_0)\int_0^t \Theta_{(x_0,L/2)}(s)^2\Psi_{(x_0,L/2)}(s)\, ds\\
	-Q(x_0)\int_0^t\Psi_{(x_0,L/2)}(s)\left(\int_0^s DH\big(\Phi_\tau (x_0, L/2)\big)\left[S_\tau^{(2)}(x_0)\right]\, d\tau\right)\, ds \\
+2\partial_{x_0}Q(x_0)\int_0^t\Theta_{(x_0,L/2)}(s)\Psi_{(x_0,L/2)}(s)\, ds\,.
\end{multlined}
\end{align*}
\end{lemma}
\begin{remark}
	Focusing on functions $u$ that are constant along the magnetic field lines, we have 
$	Q(x_0)= \partial_{x_0} Q(x_0) = \partial^2_{x_0x_0} Q(x_0) = 0 \,.$
	For this class of functions, the derivatives of $u\circ F$ simplify into Eqs.~\eqref{eq:derivatives:UcircF:simple}.
\end{remark}
\begin{lemma}\label{claim:christo} Assuming that $g^{-1}$ is defined by Eqs.~\eqref{eqs:g:coeffg}, the following identity holds
\begin{equation}
\sum_{i,j}g^{ij}\Gamma_{ij}^k=\sum_{i,j,m}g^{ij}g^{km}\partial_ig_{mj}-\frac{1}{2}\sum_{i,j,m}g^{ij}g^{km}\partial_m g_{ij},
\end{equation}
with
\begin{eqnarray*}
\partial_t g_{tt}=\partial_{x_0}g_{tt}&=&0\\
\partial_t g_{t x_0}&=& \langle D\vec{b}(\Phi_t(x_0, L/2))[\vec{b}\circ\Phi_t(x_0, L/2)], S_t^{(1)}(x_0)\rangle\\
\partial_{x_0} g_{t x_0}&=& \langle  \vec{b}\circ\Phi_t(x_0, L/2), S_t^{(2)}(x_0)\rangle\\
&&+\langle D\vec{b}(\Phi_t(x_0, L/2))[S_t^{(1)}(x_0)], S_t^{(1)}(x_0)\rangle\\
\partial_t g_{x_0 x_0}&=& 2\langle D\vec{b}(\Phi_t(x_0, L/2))[S_t^{(1)}(x_0)], S_t^{(1)}(x_0)\rangle\\
\partial_{x_0} g_{x_0 x_0}&=& 2\langle S_t^{(2)}(x_0), S_t^{(1)}(x_0)\rangle
\end{eqnarray*}
\end{lemma}
\begin{proof}[Proof of Lemma~\ref{claim:christo}] The computations are straightforward, only $\partial_t g_{t x_0}$ requires a bit of detail. Indeed, a direct computation yields
\begin{align*}
\partial_t g_{t x_0}&=\partial_t \langle \partial_t F,\partial_{x_0}F\rangle =\partial_t \langle \vec{b}\circ F,\partial_{x_0}F\rangle=\langle D\vec{b}(F)[\partial_t F],\partial_{x_0}F\rangle+\langle \vec{b}\circ F,\partial_{tx_0}F\rangle\\
&=\langle D\vec{b}(F)[\vec{b}\circ F],S^{(1)}_t(x_0)\rangle+\langle \vec{b}\circ F,D\vec{b}(F)[\partial_{x_0}F])\rangle\,.
\end{align*}
Independently, for any point $p\in \rr^2$ such that $\vec B(p)\neq \vec 0$ and any vector $X\in \rr^2$, we have that
\begin{eqnarray*}
2\langle \vec{b}(p),D\vec{b}(p)[X])=\frac{d}{ds} \Vert \vec{b}(p+sX)\Vert^2_{|s=0}=0\hbox{ since }\Vert \vec{b}\Vert=1.
\end{eqnarray*}
This yields the expected expression of $\partial_t g_{t x_0}$.
\end{proof}

\One{The preceding results lead to the following concluding theorem:}
\begin{theorem} Let $\vec B\in C^3(\rr^2,\rr^2)$ be a vector field such that \eqref{H1}-\eqref{H4} hold. Let $k\in C^3(\rr^2)$ and $K\in C^1(\rr^2)$ be two  functions such that $k(x,y)\neq 0$ for all $(x,y)\in\rr^2$. Let us fix $\alpha\in (0, d/2)$ and let us consider $D_\alpha\subset (0,d/2)\times \rr$ defined in Lemma \ref{claim:Dalpha}. Let us consider $u\in C^2(D_\alpha)$ be such that $\Delta_\parallel u=0$ in $D_\alpha$. Let $P,Q\in C^2((0,\alpha))$ be given by Lemma~\ref{exp:u}. Then $L_Ku:=\hbox{div}(K\nabla u)$ is given by the expression of Lemma~\ref{claim:lk} where the derivatives of $u\circ F$ are given in Lemma~\ref{claim:der} and the expression involving the $\Gamma_{ij}^k$'s is in Lemma~\ref{claim:christo}.
\end{theorem}

\subsection{Application of the flow-transported-solution method to the homogeneous anisotropic problem}\label{Sec:Appendix:Application}
Finally, we provide a justification for the so-called flow-transported solution procedure, which is used to generate reference solutions for strongly anisotropic problems and thus validate numerical methods. This procedure is based on the decomposition of the anisotropic problem solution $u^\eps$ in the form $u^\eps=u^0+\eps u^1$, where $u^0$ is a function with a vanishing parallel Laplacian. Within the solution manufacturing method, $u^0$ as well as its Laplacian are determined analytically. These quantities are computed thanks to the integration of convenient boundary quantities along flows defined by the $b$-field. 

We keep the same notations as in the preceding section. Integrating the expression \eqref{exp:delta:p} with $k\equiv 1$, we get
\begin{proposition}\label{prop:existence} Let $\vec B\in C^3(\rr^2,\rr^2)$ be a vector field such that \eqref{H1}-\eqref{H4} hold. Let us fix $\alpha\in (0, d/2)$ and let us consider $D_\alpha\subset (0,d/2)\times \rr$ defined in Lemma \ref{claim:Dalpha}. Then, for all $f\in C^2(D_\alpha)$ and for all $P,Q\in C^2((0,\alpha))$, there exists a unique function $u\in C^2(D_\alpha)$ such that
\begin{equation}
\left\{\begin{array}{cc}\label{eq:Lap:Parallel}
-\Delta_\parallel u:=-\nabla\cdot\big((\vec{b}\otimes\vec{b})\nabla u\big)=f\,, &\hbox{in }D_\alpha\,,\\
u(x_0, L/2)=P(x_0)\,,&\hbox{ for all }x_0\in (0,\alpha)\,,\\
\langle \nabla u,\vec{b}\rangle(x_0,L/2)=\partial_yu(x_0, L/2)=Q(x_0) \,,&\hbox{ for all }x_0\in (0,\alpha)\,.
\end{array}\right.
\end{equation}
Moreover, the explicit expression in the chart $F$ is
\begin{equation}\label{eq:explicit:expression}
	\begin{multlined}[0.85\linewidth]
u\circ F(x_0,t)=P(x_0)+Q(x_0)\int_0^t\exp \Big(-\mathcal{H}(x_0,s)\Big)\, ds\\
+\int_0^t\Bigg(\exp \Big(-\mathcal{H}(x_0,v)\Big)\int_0^v\tilde{f}(x_0,s)\exp \Big(\mathcal{H}(x_0,s)\Big)\, ds\Bigg)\, dv \,,
	\end{multlined}
\end{equation}
for all $(x_0,t) \in (0,\alpha)\times\rr$, where $\tilde{H}(x_0,t):=H(\Phi_t(x_0, L/2))=\hbox{Tr}(Db(\Phi_t(x_0, L/2)))$, $\mathcal{H}(x_0,s):= \int_0^s \tilde{H}(x_0,t) dt $ and $\tilde{f}(x_0,t):={-}f(\Phi_t(x_0, L/2))$ for all $x_0\in (0,\alpha)$ and $t\in\rr$.
\end{proposition}
From the explicit expression stated by Eq.~\eqref{eq:explicit:expression}, it is straightforward to get
\begin{coro}[Stability]\label{coro:stab} Under the hypotheses and notation of Prop.~\ref{prop:existence}, let $u_{f,P,Q}$ be the unique solution to Eq.~\eqref{eq:Lap:Parallel} for given $f,P,Q$. Then for all $\delta\in (0,\alpha/2)$, all  $f,f_0\in C^2(D_\alpha)$ and for all $P,P_0,Q,Q_0\in C^2((0,\alpha))$, there exists $C(\alpha,\delta,\vec{b})$ such that
\begin{equation}
	\begin{multlined}[0.85\textwidth]
\Vert u_{f,P,Q}-u_{f_0,P_0,Q_0}\Vert_{C^2(D^\delta_\alpha)}\leq C(\alpha,\delta,\vec{b})\Big(\Vert f-f_0\Vert_{C^2(D^\delta_\alpha)}\\+\Vert P-P_0\Vert_{C^2(\delta,\alpha-\delta)}+\Vert Q-Q_0\Vert_{C^2(\delta,\alpha-\delta)}\Big)\,,
	\end{multlined}
\end{equation}
where $D^\delta_\alpha:=F((\delta,\alpha-\delta)\times (-\delta^{-1},\delta^{-1}))$.
\end{coro}

\begin{theorem}\label{Th:B2} Let $B\in C^5(\rr^2,\rr^2)$ be a vector field such that \eqref{H1}-\eqref{H4} hold. Let us fix $\alpha\in (0, d/2)$ and let us consider $D_\alpha\subset (0,d/2)\times \rr$ defined by Lemma~\ref{claim:Dalpha}. We fix $f^0\in C^2(D_\alpha)$ and $P,Q\in C^4((0,\alpha))$. Then there exists families $(u^\eps)_\eps\in C^2(D_\alpha)$ and $(f^\eps)_\eps\in C^2(D_\alpha)$ such that
\begin{equation}
\left\{\begin{array}{cc}
-\frac{1}{\eps}\Delta_\parallel u^\eps-\Delta_\perp u^\eps=f^\eps\,, &\hbox{in }D_\alpha\,,\\
u^\eps(x_0, L/2)=P(x_0)\,,&\hbox{ for all }x_0\in (0,\alpha)\,,\\
\langle \nabla u^\eps,\vec{b}\rangle(x_0,L/2)=\partial_yu^\eps(x_0, L/2)=Q(x_0)\,,&\hbox{ for all }x_0\in (0,\alpha)\,,
\end{array}\right.
\end{equation}
with
$$\lim_{\eps\to 0}u^\eps=u^0\hbox{ and }\lim_{\eps\to 0}f^\eps=f^0\hbox{ in }C^0_{loc}(D_\alpha)\,,$$
where $u^0\in C^2(D_\alpha)$ is the unique solution to
\begin{equation}
\left\{\begin{array}{cc}
\Delta_\parallel u^0:=\nabla \cdot\big((\vec{b}\otimes\vec{b})\cdot\nabla u^0\big)=0\,, &\hbox{in }D_\alpha\,,\\
u^0(x_0, L/2)=P(x_0)\,,&\hbox{ for all }x_0\in (0,\alpha)\,,\\
\langle \nabla u^0,\vec{b}\rangle(x_0,L/2)=\partial_yu^0(x_0, L/2)=Q(x_0)\,,&\hbox{ for all }x_0\in (0,\alpha)\,.
\end{array}\right.
\end{equation}
\end{theorem}
\begin{proof}[Proof of Theorem~\ref{Th:B2}] The existence of $u^0\in C^2(D_\alpha)$ follows from Proposition \ref{prop:existence}. Since $P,Q\in C^4$ and $B\in C^5$, then $H\in C^4$ and we get that $u^0\in C^4(D_\alpha)$. We let $u^1\in C^2(D_\alpha)$ be such that
\begin{equation}
\left\{\begin{array}{cc}
{-}\Delta_\parallel u^1:= f^0{+}\Delta u^0 \,,&\hbox{in }D_\alpha\,,\\
u^1(x_0, L/2)=\partial_yu^1(x_0, L/2)=0\,,&\hbox{ for all }x_0\in (0,\alpha)\,.
\end{array}\right.
\end{equation}
Here again, the existence follows from Proposition \ref{prop:existence}. We define $u^\eps:=u^0+\eps u^1$ and  $f^\eps:=f^0{-}\eps\Delta_\perp u^1$. As one checks, these families satisfy the conclusion of the Theorem.
\end{proof}

\bibliographystyle{abbrv}
\bibliography{bib}

\end{document}